\documentclass[12pt,psamsfonts]{amsart}
\usepackage{amsmath,amsthm,amsfonts,amssymb}
\usepackage{eucal,amscd}
\usepackage{graphicx}
\usepackage[all,knot]{xy}
\xyoption{arc}

\newcommand{\e}{{\epsilon}}

\newcommand{\X}{\mathcal{X}}
\newcommand{\Mcg}{\mathcal{M}}

\newcommand{\mV}{\mathcal{V}}

\newcommand{\T}{\mathcal{T}}

\newcommand{\tr}{{\rm tr}}

\newcommand{\la}{\lambda}

\newcommand{\g}{\mathfrak{g}}

\newcommand{\Q}{\mathbb Q}
\newcommand{\R}{\mathbb{R}}

\DeclareMathOperator{\Hom}{Hom}

\newcommand{\C}{\mathbb C}

\newcommand{\scc}{\mathcal{S}}
\newcommand{\Z}{\mathbb Z}

\newcommand{\ts}{\tilde{s}}

\newcommand{\Cc}{{\mathcal C}}
\newcommand{\F}{{\mathcal F}}
\newcommand{\Rr}{{\mathcal R}}
\newcommand{\Aa}{{\mathcal A}}
\newcommand{\V}{{\mathcal V}}

\newcommand{\rd}{\mathcal{R}_d}
\def\grpgen#1{<\!\!#1\!\!>}

\def\T{\textrm}

\def\p{\partial}

\def\e{\epsilon}
\def\f{\frac}
\def\g{\gamma}

\def\la{\longrightarrow}
\def\ol{\overline}
\def\p{\partial}
\def\q{\quad}

\def\S{\Sigma}

\def\tr{\triangle}

\def\tn{\textnormal}

\def\M{\textrm{Mor}}
\def\TL{\textrm{TL}}
\def\TLJ{\textrm{TLJ}}
\def\dim{\textrm{dim}}
\def\Pic{\textrm{Pic}}

\numberwithin{equation}{section}

\newtheorem{theorem}[equation]{Theorem}

\newtheorem{corollary}[equation]{Corollary}
\newtheorem{lemma}[equation]{Lemma}

\newtheorem{prop}[equation]{Proposition}

\theoremstyle{definition}

\newtheorem{definition}[equation]{Definition}

\begin{document}

\title[Picture TQFTs]
{On Picture (2+1)-TQFTs}

\author{Michael Freedman}
\email{michaelf@microsoft.com}

\author{Chetan Nayak}
\email{nayak@kitp.ucsb.edu}

\author{Kevin Walker}
\email{kevin@canyon23.net}

\author{Zhenghan Wang}
\email{zhenghwa@microsoft.com}
\address{Microsoft Station Q\\CNSI Bldg Rm 2243\\
    University of California\\
    Santa Barbara, CA 93106-6105\\
    U.S.A.}

\dedicatory{Dedicated to the memory of Xiao-Song Lin}

\thanks{The fourth author is partially supported by NSF FRG grant DMS-034772.}

\begin{abstract}
The goal of the paper is an exposition of the simplest $(2+1)$-TQFTs in a sense following a
pictorial approach.  In the end, we fell short on details in the later sections where new results are
stated and proofs are outlined.  Comments are welcome and should be sent to the 4th author.
\end{abstract}

\maketitle

\section{Introduction}
Topological quantum field theories (TQFTs) emerged into physics
and mathematics in the 1980s from the study of three distinct
enigmas: the infrared limit of $1+1$ dimensional conformal field
theories, the fractional quantum Hall effect (FQHE), and the
relation of the Jones polynomial to $3-$manifold topology.  Now 25
years on, about half the literature in $3-$dimensional topology
employs some $\lq\lq$quantum" view point, yet it is still difficult
for people to learn what a TQFT is and to manipulate the simplest
examples. Roughly (axioms will follow later), a ($2+1)-$dimensional
TQFT is a functor which associates a vector space $V(Y)$ called
$\lq\lq$modular functor" to a closed oriented surface $Y$ (perhaps
with some extra structures); sends disjoint union to tensor
product, orientation reversal to dual, and is natural with respect
to transformations (diffeomorphisms up to isotopy or perhaps a
central extension of these) of $Y$.  The empty set $\emptyset$ is
considered to be a manifold of each dimension: $\{0,1,\cdots\}$.
 As a closed surface, the
 associated vector space is $\C$, i.e., $V(\emptyset)=\C$.
 Also if $Y=\p X$, $X$ an oriented $3-$manifold (also perhaps with
some extra structure), then a vector $Z(X) \in V(Y)$ is
determined (surfaces $Y$ with boundary also play a role but we
pass over this for now.)  A closed $3-$manifold $X$ determines a
vector $Z(X) \in V(\emptyset)= \C$, that is a number. In the case
$X$ is the $3-$sphere with $\lq\lq$extra structure" a link $L$, then
Witten's $\lq\lq SU(2)-$family" of TQFTs yields a Jones polynomial
evaluation $Z(S^3 , L)= J_L(e^{2 \pi i/r})$, $r =3,4,5, \ldots,$
as the $\lq\lq$closed $3-$manifold" invariants, which mathematically
are the Reshetikhin-Turaev invariants based on quantum groups
\cite{jones84}\cite{witten89}\cite{reshetikhinturaev}. This is the
best known example. Note that physicists tend to index the same
family by the levels $k=r-2$. The shift $2$ is the dual Coxeter
number of $SU(2)$.  We will use both indices. Most of the
$\lq\lq$quantum" literature in topology focuses on such closed
$3-$manifold invariants but there has been a growing awareness
that a deeper understanding is locked up in the representation
spaces $V(Y)$ and the $\lq\lq$higher algebras" associated to boundary
$(Y)$ (circles) and points \cite{freedquinn}\cite{freed}. Let us
explain this last statement. While invariants of $3-$manifolds may
be fascinating in their interrelations there is something of a
shortage of work for them within topology. Reidemeister was
probably the last topologist to be seriously puzzled as to whether
a certain pair of $3-$manifolds were the same or different and,
famously, solved his problem by the invention of $\lq\lq$torsion". (In
four dimensions the situation is quite the opposite, and the
closed manifold information from $(3+1)$ dimensional TQFTs would
be most welcome. But in this dimension, we do not yet know
interesting examples of TQFTs.) So while the subject in dimension
$3$ seems to be maturing away from the closed case it is running
into a pedological difficulty.  It is hard to develop a solid
understanding of the vector spaces $V(Y)$ even for simple
examples.  Our goal in these notes is, in a few simple examples to
provide an intuition and understanding on the same level of
$\lq\lq$admissible pictures" modulo relations, just as we understand
homology as cycles modulo boundaries. This is the meaning of
$\lq\lq$picture" in the title.  A picture TQFT is one where $V(Y)$ is
the space of formal $\C-$linear combinations of $\lq\lq$admissible"
pictures drawn on $Y$ modulo some \underline{local} (i.e. on a
disk) linear relations.
 We will use the terms formal links, or formal
 tangles, or formal pictures , etc. to mean $\C$-linear combinations of
 links, tangles, pictures, etc.  Formal tangles in $3$-manifolds are also commonly
 referred to as $\lq\lq$skein"s.  Equivalently, we can adopt a dual
 point of view: take the space of linear functionals on multicurves and
 impose linear constraints for functionals.  This point of view is
 closer to the physical idea of $\lq\lq$amplitude of an
 eigenstate": think of a functional $f$ as a wavefunction and its
 value $f(\g)$ on a multicurve $\g$ as as the amplitude of the
 eigenstate $\g$.  Then quotient spaces of pictures become
 subspaces of wavefunctions.

Experts may note that central charge $c \neq 0$ is an obstruction
to this $\lq\lq$picture formulation":  the mapping class group
$\Mcg(Y)$ acts directly on pictures and so induces an action on
any $V(Y)$ defined by pictures. As $c$ determines a central
extension $\widetilde{\Mcg}^c(Y)$ which acts in place of
$\Mcg(Y)$, the feeling that all interesting theories must have $c
\neq 0$ may have discouraged a pictorial approach. However this is
not true: for any $V(Y)$ its endomorphism algebra End $V\cong
V^\ast \bigotimes V$ has central charge $c=0 \, \left(c (V^\ast )
=-c (V)\right)$ and remembers the original projective
representation faithfully. In fact, all our examples are either of
this form or slightly more subtle quantum doubles or Drinfeld
centers in which the original theory $V$ violates some axiom (the
nonsingularity of the $S-$matrix) but this deficiency is
$\lq\lq$cured" by doubling \cite{kassel}\cite{mugerII}. Although those
notes focus on picture TQFTs based on variations of the
Jones-Wenzl projectors, the approach can be generalized to an
arbitrary spherical tensor category.  The Temperley-Lieb
categories are generated by a single $\lq\lq$fundamental"
representation, and all fusion spaces are of dimension $0$ or $1$,
so pictures are just $1$-manifolds. In general, $1$-manifolds need
to be replaced by tri-valent graphs whose edges carry labels. But
$c=0$ is not sufficient for a TQFT to have a picture description.
Given any two TQFTs with opposite central charges, their product
has $c=0$, e.g. TQFTs with $\Z_n$ fusion rules have $c=1$, so the
product of any theory with the mirror of a different one has
$c=0$, but such a product theory does not have a picture
description in our sense.

While these notes describe the mathematical side of the story, we
have avoided jargon which might throw off readers from physics.
When different terminologies prevail within mathematics and
physics we will try to note both. Within physics, TQFTs are
referred to as $\lq\lq$anyonic systems" \cite{wilczek}\cite{DFN}.
These are $2$-dimensional quantum mechanical systems with point
like excitations (variously called $\lq\lq$quas-particle" or just
$\lq\lq$particle", anyon, or perhaps $\lq\lq$nonabelion") which under
exchange exhibit exotic statistics: a nontrival representation of
the braid groups acting on a finite dimensional Hilbert space $V$
consisting of $\lq\lq$internal degrees of freedom". Since these
$\lq\lq$internal degrees of freedom" sound mysterious, we note that
this information is accessed by fusion: fuse pairs of anyons along
a well defined trajectory and observe the outcome. Anyons are a
feature of the fractional quantum Hall effect; Laughlin's 1998
Nobel prize was for the prediction of an anyon carrying  change
$e/3$ and with braiding statistics $e^{2\pi i/3}$.  In the FQHE
central charge $c \neq 0$ is enforced by a symmetry breaking
magnetic field B.  It is argued in \cite{freedmanmag} that solid
state realizations of doubled or $\lq\lq$picture" TQFTs may - if found
- be more stable (larger spectral gap above the degenerate ground
state manifold) because no symmetry breaking is required. The
important electron - electron interactions would be at a lattice
spacing scale $\sim 4\AA$ rather than at a $\lq\lq$magnetic length"
typically around $150 \AA$. So it is hoped that the examples which
are the subject of these notes will be the low energy limits of
certain microscopic solid state models.  Picture TQFTs have a
Hamiltonian formulation, and describe string-net condensation in
physics, which serve as a classification of non-chiral topological
phases of matter. An interesting mathematical application is the
proof of the asymptotic faithfulness of the representations of the
mapping class groups.

As mentioned above, these notes are primarily about examples
either of the form $V^\ast \bigotimes V$ or with a related but
more general doubled structure $\mathcal{D}(V)$.  In choosing a
path through this material there seemed a basic choice:  (1)
present the picture (doubled) theories in a self contained way in
two dimensions with no reference to their twisted $(c \neq 0)$ and
less tractable parent theories $V$ or (2) weave the stories of
$\mathcal{D}(V)$ and $V$ together from the start and exploit the
action of $\mathcal{D}(V)$ on $V$ in analyzing the structure of
$\mathcal{D}(V)$.  In the end, the choice was made for us:  we did
not succeed in finding purely combinatorial $\lq\lq$picture-proofs"
for all the necessary lemmas --- the action on $V$ is indeed very
useful so we follow course (2). We do recommend to some interested
brave reader that she produce her own article hewing to course
(1).

In the literature \cite{BHMV} comes closest to the goals of the
notes, and \cite{walker06} exploits deeply the picture theories in
many directions. Actually, a large part of the notes will follow
from a finished \cite{walker06}.  If one applies the set up of
[BHMV] to skeins in surface cross interval, $Y \times I$, and then
resolves crossings to get a formal linear combination of
$1-$submanifolds of $Y=Y \times \f{1}{2} \subset Y \times I$ one
arrives at (an example of) the ``pictures" we study. In this
doubled context there is no need for the $p_1-$structure (or
``two-framing") intrinsic to the other approaches.  To readers
familiar with \cite{BHMV} one should think of skeins in a handle
body $H$, $\p H=Y$, when an undoubled theory $V(Y)$ is being
discussed, and skeins in $Y \times I$ when $\mathcal{D}V(Y)$ is
under consideration.

By varying pictures and relations we produce many examples, and in
the Temperley-Lieb-Jones context give a complete analysis of the
possible local relations.  Experts have long been troubled by
certain sign discrepancies between the $S-$matrix arising from
representations (or loop groups or quantum
groups)\cite{mooreseiberg}\cite{witten89}\cite{kirbymelvin} on the
one hand and from the Kauffman bracket on the other
\cite{lickorish}\cite{turaev}\cite{kauffmanlins}. The source of
the discrepancy is that the fundamental representation of $SU(2)$
is anti-symmetrically self dual whereas there is no room in
Kauffman's spin-network notation to record the antisymmetry.  We
rectify this by amplifying the pictures slightly, which yields
\underline{exactly} the modular functor $V$ coming from
representation theory of $SU(2)_q$.

\iffalse We treat in detail (verify all axioms) the modular
functors associated to $U(1)$, Kauffman bracket, $SU(2)$ and
$SO(3)$ (in that order) each in an undoubled and then a doubled
version.  A theory $V$ or $\mathcal{D}(V)$ is \underline{unitary}
if the vector spaces $V$ have natural positive definite Hermitian
structures.  Only unitary theories will have physical relevance so
we decide, for each theory.  If it is unitary (or perhaps
unitarizable). \vspace{.3in}

Table of Contents:

\vspace{.1in}

Section 2:  Jones representations

Section 3:  Diagram TQFTs for closed manifolds

Section 4:  Morita equivalence and cut-paste topology

Section 5:  Representation of Tempeley-Lieb-Jones categories

Section 6:  Diagram TQFTs

Section 7:  Jones-Kauffman TQFTs

Section 8:  Black-White TQFTs

Section 9:  Witten-Reshetikhin-Turaev and Turaev-Viro TQFTs

Section 10: Classification and unitarity

\vspace{.2in}

Appendices:

Appendix A:  Topological phases of matter

Appendix B:  Categories and functors

Appendix C:  Representation of linear categories

\vspace{.3in}\fi

The content of each section is as follows.  In Sections
\ref{jonesrepandjoneswenzl}, \ref{joneswenzlclosed}, we treat
diagram TQFTs for closed manifolds.  In Sections
\ref{moritacutpaste}, \ref{repoftlj}, \ref{joneswenzlfull}, we
handle boundaries.  In Sections \ref{joneskauffman},
\ref{blackwhite}, \ref{wrttqfts}, we cover the related
Jones-Kauffman TQFTs, and the Witten-Reshetikhin-Turaev
$SU(2)$-TQFTs which have anomaly, and non-trivial Frobenius-Schur
indicators, respectively. In Section \ref{classificationunitary},
we first prove the uniqueness of TQFTs based on Jones-Wenzl
projectors, and then classify them according to the Kauffman
variable $A$.  A theory $V$ or $\mathcal{D}(V)$ is
\underline{unitary} if the vector spaces $V$ have natural positive
definite Hermitian structures. Only unitary theories will have
physical relevance so we decide for each theory if it is unitary.

\section{Jones representations}\label{jonesrepandjoneswenzl}

\subsection{Braid statistics}

Statistics of elementary particles in $3$-dimensional space is
related to representations of the permutation groups $S_n$. Since
the discovery of the fractional quantum Hall effect, the existence
of anyons  in $2$-dimensional space becomes a real possibility.
Statistics of anyons is described by unitary representations of
the braid groups $B_n$.  Therefore, it is important to understand
unitary representations of the braid groups $B_n$.  Statistics of
$n$ anyons is given by unitary representation of the $n$-strand
braid group $B_n$.  Since statistics of anyons of different
numbers $n$ is governed by the same local physics, unitary
representations of $B_n$ have to be compatible for different $n$'s
in order to become possible statistics of anyons.  One such
condition is that all representations of $B_n$ come from the same
unitary braided tensor category.

There is an exact sequence of groups:
$ 1\longrightarrow PB_n \longrightarrow B_n\longrightarrow
S_n\longrightarrow 1,$
where $PB_n$ is the $n$-strand pure braid group.  It follows that
every representation of the permutation group $S_n$ gives rise to
a representation of the braid group $B_n$.  An obvious fact for
such representations of the braid groups is that the images are
always finite.  More interesting representations of $B_n$ are
those that do not factorize through $S_n$, in particular those
with infinite images.

To construct representations of the braid groups $B_n$, we recall
the construction of all finitely dimensional irreducible
representations (irreps) of the permutation groups $S_n$:  the
group algebra $\C [S_n]$, as a representation of $S_n$, decomposes
into irreps as $\C[S_n]\cong \bigoplus_{i}{\C}^{\dim V_i}\otimes
V_i$, where the sum is over all irreps $V_i$ of $S_n$.  This
construction cannot be generalized to $B_n$ because $B_n$ is an
infinite group for $n\geq 2$.  But by passing to various different
finitely dimensional quotients of $\C[B_n]$, we obtain many
interesting representations of the braid groups. This class of
representations of $B_n$ is Schur-Weyl dual to the the class of
braid group representations from the quantum group approach and
has the advantage of being manifestly unitary. This approach,
pioneered by V.~Jones \cite{jones84}, provides the best understood
examples of unitary braid group representations besides the Burau
representation, and leads to the discovery of the celebrated Jones
polynomial of knots \cite{jones87}. The theories in this paper are
related to the quantum $SU(2)_q$ theories.

\subsection{Generic Jones representation of the braid groups}\label{genericjones}

The $n$-strand braid group $B_n$ has a standard presentation with
generators $ \{\sigma_i, i=1,2,\cdots, n-1\}$ and relations:
\begin{equation}\label{farcommunitivity}
\sigma_i \sigma_j=\sigma_j\sigma_i,\;\;\; \textrm{if} \;\;\; |i-j|
\geq 2,
\end{equation}
\begin{equation}\label{braidrelation}
\sigma_i \sigma_{i+1}\sigma_i=\sigma_{i+1} \sigma_{i}\sigma_{i+1}.
\end{equation}

If we add the relations $\sigma_i^2=1$ for each $i$, we recover
the standard presentation for $S_n$.  In the group algebra
$k[B_n]$, where $k$ is a field (in this paper $k$ will be either
$\C$ or some rational functional field $\C(A)$ or $\C(q)$ over
variables $A$ or $q$), we may deform the relations $\sigma_i^2=1$
to linear combinations (superpositions in physical parlance)
$\sigma_i^2=a\sigma_i+b$ for some $a,b\in k$. By rescaling the
relations, it is easy to show that there is only $1$-parameter
family of such deformations.  The first interesting quotient
algebras are the Hecke algebras of type A, denoted by $H_n(q)$,
with generators $1,g_1,g_2,\cdots, g_{n-1}$ over $\Q(q)$ and
relations:
\begin{equation}\label{farcommunitivityH}
g_i g_j=g_jg_i,\;\;\; \textrm{if} \;\;\; |i-j| \geq 2,
\end{equation}
\begin{equation}\label{braidrelationH}
g_i g_{i+1}g_i=g_{i+1} g_{i}g_{i+1}.
\end{equation}
and
\begin{equation}\label{heckerelation}
g_i^2=(q-1)g_i+q.
\end{equation}

The Hecke relation \ref{heckerelation} is normalized to have roots
$\{-1,q\}$ when the corresponding quadratic equation is solved.
The Hecke algebras $H_n(q)$  at $q=1$ become $\C[S_n]$, hence they
are deformations of $\C[S_n]$. When $q$ is a variable, the irreps
of $H_n(q)$ are in one-to-one correspondence with the irreps of
$\C[S_n]$.

To obtain the Hecke algebras as quotients of $\C[B_n]$, we set
$q=A^4$, and $g_i=A^{3}\sigma_i$, where $A$ is a new variable,
called the Kauffman variable since it is the conventional variable
for the Kauffman bracket below. Note that $q=A^2$ in
\cite{kauffmanlins}.  The prefactor $A^3$ is introduced to match
the Hecke relation \ref{heckerelation} exactly to a relation in
the Temperley-Lieb algebras using the Kauffman bracket. In terms
of the new variable $A$, and new generators $\sigma_i$'s, the
Hecke relation \ref{heckerelation} becomes
\begin{equation}\label{heckerelationB}
\sigma_i^2=(A-A^{-3})\sigma_i+A^{-2}.
\end{equation}

The Kauffman bracket $<>$ is defined by the resolution of a
crossing in Figure \ref{kauffmanbracket}

\begin{figure}\label{kauffmanbrackets}
\centering
\includegraphics[scale=.3]{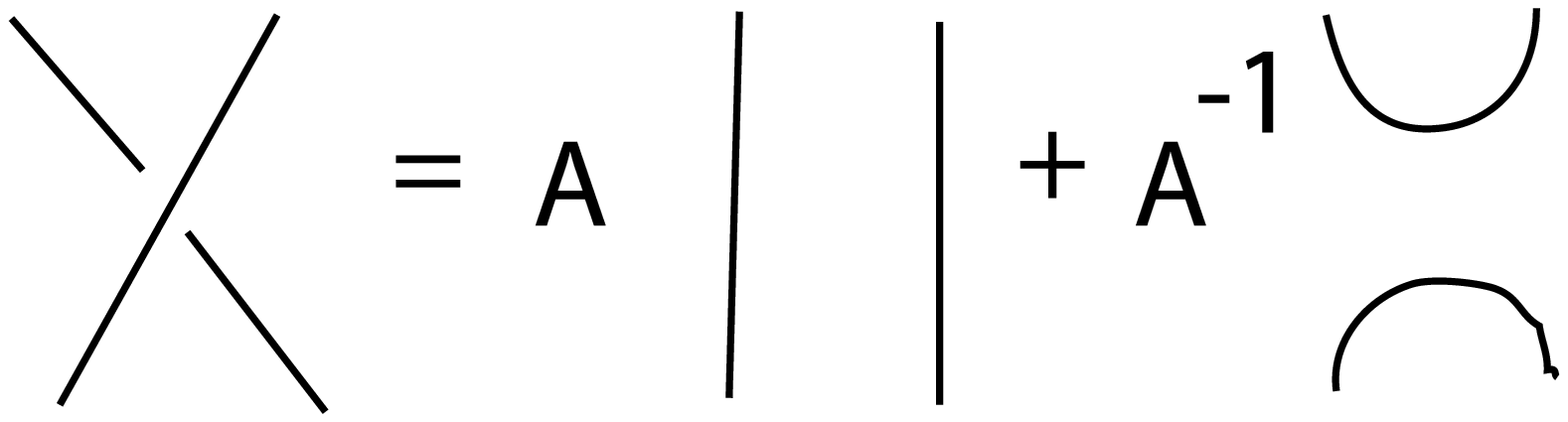}
\caption{Kauffman bracket}\label{kauffmanbracket}
\end{figure}

As a formula, $\sigma_i=A\cdot \textrm{id}+A^{-1}U_i$, where $U_i$
is a new generator.  The Hecke algebra $H_n(q)$ in variable $A$
and generators $1,U_1,U_2,\cdots, U_{n-1}$ is given by relations:
\begin{equation}\label{farcommunitivityTL}
U_i U_j=U_jU_i,\;\;\; \textrm{if} \;\;\; |i-j| \geq 2,
\end{equation}
\begin{equation}\label{braidrelationTL}
U_i U_{i+1}U_i-U_i=U_{i+1} U_{i}U_{i+1}-U_{i+1},
\end{equation}
and
\begin{equation}\label{heckerelationTL}
U_i^2=dU_i,
\end{equation}
where $d=-A^2-A^{-2}$.

The relation \ref{heckerelationTL} is the same as relation
\ref{heckerelationB}, which is the Hecke relation
\ref{heckerelation}.  The relation \ref{braidrelationTL} is the
braid relation \ref{braidrelationH}.

The Temperley-Lieb (TL) algebras, denoted as $\TL_n(A)$, are
further quotients of the Heck algebras.  In the TL algebras, we
replace the relations \ref{braidrelationTL} by
\begin{equation}
U_iU_{i\pm 1}U_i=U_i,
\end{equation}
i.e., both sides of relation \ref{braidrelationTL} are set to $0$.

\begin{prop}\label{KBalgebra}

The Kauffman bracket $<>: k[B_n]\la \TL_n(A)$ is a surjective
algebra homomorphism, where $k=\C(A)$.

\end{prop}

The proof is a straightforward computation.

When $A$ is generic,  the TL algebras $\TL_n(A)$ are semi-simple,
hence $\TL_n(A)\cong \oplus_i \tn{Mat}_{n_i}(\C(A))$, where
$\tn{Mat}_{n_i}$ are $n_i\times n_i$ matrices over $\C(A)$ for
some $n_i$'s.

The generic Jones representation of the braid groups $B_n$ is
defined as follows:

\begin{definition}
By the decomposition $\tn{TL}_n(A)\cong \oplus_i
\tn{Mat}_{n_i}(\C(A))$, each braid $\sigma \in B_n$ is mapped to a
direct sum of matrices under the Kauffman bracket.  It follows
from Prop. \ref{KBalgebra} that the image matrix of any braid is
invertible and the map is a group homomorphism when restricted to
$B_n$.
\end{definition}

It is an open question whether or not the generic Jones
representation is faithful, i.e., are there non-trivial braids
which are mapped to the identity matrix?

\subsection{Unitary Jones representations}

The TL algebras $\TL_n(A)$ have a beautiful picture description by
L.~Kauffman, inspired by R.~Penrose's spin-networks, as follows:
fix a rectangle $\Rr$ in the complex plane with $n$ points at both
the top and the bottom of $\Rr$ (see Fig.\ref{TLgenerators}),
$\TL_n(A)$ is spanned formally as a vector space over $\C(A)$ by
embedded curves in the interior of $\Rr$ consisting of $n$
disjoint arcs connecting the $2n$ boundary points of $\Rr$ and any
number of simple closed loops.  Such an embedding will be called a
diagram or a multi-curve in physical language, and a linear
combination of diagrams will be called a formal diagram. Two
diagrams that are isotopic relative to boundary points represent
the same vector in $\TL_n(A)$. To define the algebra structure, we
introduce a multiplication: vertical stacking from bottom to top
of diagrams and extending bilinearly to formal diagrams;
furthermore, deleting a closed loop must be compensated for by
multiplication by $d=-A^2-A^{-2}$. Isotopy and the deletion rule
of a closed trivial loop together will be called $\lq\lq$d-isotopy".

\begin{figure}[tbh]
\centering
\includegraphics[width=3.45in]{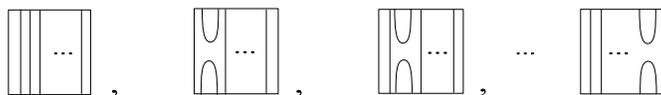}
\caption{Generators of TL}\label{TLgenerators}
\end{figure}

For our application, the variable $A$ will be evaluated at a
non-zero complex number.  We will see later that when
$d=-A^2-A^{-2}$ is not a root of a Chebyshev polynomial
$\Delta_i$, $\TL_n(A)$ is semi-simple over $\C$, therefore,
isomorphic to a matrix algebra. But when $d$ is a root of some
Chebyshev polynomial, $\TL_n(A)$ is in general not semi-simple.
Jones discovered a semi-simple quotient by introducing local
relations, called the Jones-Wenzl projectors
\cite{jones83rep}\cite{wenzlinduction}\cite{kauffmanlins}.
Jones-Wenzl projectors have certain rigidity. Represented by
formal diagrams in TL algebras, Jones-Wenzl projectors make it
possible to describe two families of TQFTs labelled by integers.
Conventionally the integer is either $r\geq 3$ or $k=r-2\geq 1$.
The integer $r$ is related to the order of $A$, and $k$ is the
level related to the $SU(2)$-Witten-Chern-Simons theory.  One
family is related to the $SU(2)_k$-Witten-Reshetikhin-Turaev (WRT)
TQFTs, and will be called the Jones-Kauffman TQFTs. Although
Jones-Kauffman TQFTs are commonly stated as the same as WRT TQFTs,
they are really not. The other family is related to the quantum
double of Jones-Kauffman TQFTs, which are of the Turaev-Viro type.
Those doubled TQFTs, labelled by a level $k\geq 1$, are among the
easiest in a sense, and will be called diagram TQFTs. The level
$k=1$ diagram TQFT for closed surfaces is the group algebras of
$\Z_2$-homology of surfaces. Therefore, higher level diagram TQFTs
can be thought as quantum generalizations of the $\Z_2$-homology,
and the Jones-Wenzl projectors as the generalizations of the
homologous relation of curves in Figure \ref{homology}.

\begin{figure}\label{homologys}
\centering
\includegraphics[scale=1]{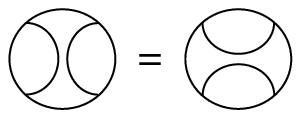}
\caption{$Z_2$ homology}\label{homology}
\end{figure}

The loop values $d=-A^2-A^{-2}$ play fundamental roles in the
study of Temperley-Lieb-Jones theories, in particular the picture
version of $\TL_n(A)$ can be defined over $\C(d)$, so we will also
use the notation $\TL_n(d)$. In the following, we focus the
discussion on $d$, though for full TQFTs or the discussion of
braids in $\TL_n(A)$, we need $A$'s.
 Essential to the proof and to the understanding of the exceptional
values of $d$ is the trace tr: TL$_{n}(d) \la \C$ defined by
Fig.\ref{markovtrace}. This Markov trace is defined on diagrams by
(and then extended linearly) connecting the endpoints at the top
to the endpoints at the bottom of the rectangle by $n$ non
crossing arcs in the complement of the rectangle, counting the
number $\#$ of closed loops (deleting the rectangle), and then
forming $d^{\#}$.

\begin{figure}[tbh] \label{markovtraces}
\centering
\includegraphics[width=3.45in]{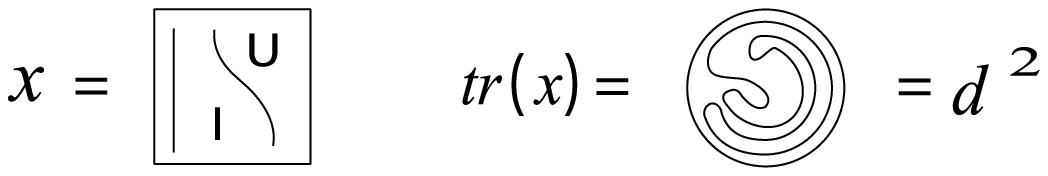}
\caption{Markov Trace} \label{markovtrace}
\end{figure}

The Markov trace $(x, y)\rightarrowtail \tn{tr}(\ol{x}y)$ extends
to a sesquilinear pairing on $\TL_n(d)$, where bar (diagram) is
reflection in a horizontal middle-line and bar(coefficient) is
complex conjugation.

Define the n$^{\tn{th}}$ Chebyshev polynomial $\tr_n (x)$
inductively by $\tr_0 =1, \tr_1 =x$, and $\tr_{n+1} (x) = x \tr_n
(x) - \tr_{n-1}(x)$.  Let $c_n$ be the Catalan number
$c_n=\frac{1}{n+1}{{2n}\choose{n}}$.  There are $c_n$ different
diagrams $\{D_i\}$ consisting of $n$ disjoint arcs up to isotopy
in the rectangle $\Rr$ to connect the $2n$ boundary points of
$\Rr$. These $c_n$ diagrams generate $\TL_n(d)$ as a vector space.
Let $M_{c_n\times c_n}=(m_{ij})$ be the matrix of the Markov trace
Hermitian pairing in a certain order of $\{D_i\}$, i.e.
$m_{ij}=\tn{tr}(\ol{D_i}D_j)$, then we have:

\begin{equation}\label{meanderformula}
\tn{Det}(M_{c_n\times c_n})=\prod_{i=1}^n{\Delta_i(d)}^{a_{n,i}},
\end{equation}

where
$a_{n,i}={{2n}\choose{n-i-2}}+{{2n}\choose{n-i}}-2{{2n}\choose{n-i-1}}$.

This is derived in \cite{meander}.

As a quick consequence of this formula, we have:

\begin{lemma}
 The dimension of $\TL_n(d)$ as a vector space over $\C(d)$
is $c_n$ if $d$ is not a root of the Chebyshev polynomials
$\Delta_i,1\leq i \leq n$, where
$c_n=\frac{1}{n+1}{{2n}\choose{n}}$.
\end{lemma}

\begin{proof}
By the formula \ref{meanderformula}, if $d$ is not a root of
$\Delta_i, 1\leq i \leq n$, then $\{D_i\}$ are linearly
independent. As a remark, since each $D_i$ is a monomial of
$U_i$'s, it follows that $\{U_i\}$ generate $\TL_n(d)$ as an
algebra.

\end{proof}

Next we show the existence and uniqueness of the Jones-Wenzl
projectors.

\begin{theorem}\label{joneswenzlprojectors}
For $d \in \C$ that is not a root of $\tr_k$ for all $k<n$, then
$\tn{TL}_{n}(d)$ contains a unique element $p_n$ characterized by:
$p_n^2 = p_n\neq 0$ and $U_i p_n = p_n U_i=0$ for all $1\leq i\leq
n-1$. Furthermore $p_n$ can be written as $p_n = 1 +U$ where
$U=\sum c_j h_j$, $h_j$ a product of $U_i^{'}$s, $1 \leq i \leq
n-1$ and $c_j\in \C$.
\end{theorem}

\begin{proof}
Suppose $p_n$ exists and can be expanded as $p_n =a 1
+U$, then $p_n^2 = p_n (a1 +U) = p_n (a1) = a p_n =a^2 1+a U$, so
$a=1$. Now check uniqueness by supposing $p_n = 1+U$ and $p_n^{'}
=1 +V$ both have the properties above and expand $p_n p_n^{'}$
from both sides:
$$
p_n^{'} = 1\cdot p_n^{'} = (1+ U) p_n^{'} = p_n p_n^{'} = p_n
(1+V) = p_n\cdot 1 =p_n .
$$

The proof is completed by H.~Wenzl's \cite{wenzlinduction}
inductive construction of $p_{n+1}$ from $p_n$ which also reveals
the exact nature of the $\lq\lq$generic" restriction on $d$.  The
induction is given in Figure \ref{jwinduction}, where
$\mu_n=\frac{\Delta_{n-1}(d)}{\Delta_n(d)}$.

\begin{figure}\label{jwinductions}
\centering
\includegraphics[scale=1]{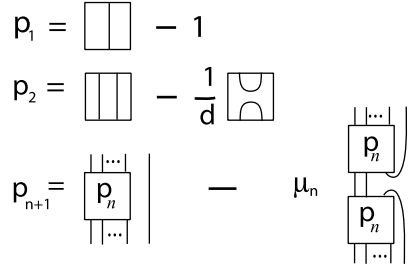}
\caption{Jones Wenzl projectors}\label{jwinduction}
\end{figure}

Tracing the inductive definition of $p_{n+1}$ yields
$\tn{tr}(p_{n+1}) =d \,\,\tn{tr}(p_n ) - \f{\tr_{n-1}}{\tr_n}
\tn{tr}(p_n )$ showing $\tn{tr}(p_n)$ satisfies the Chebyshev
recursion (and the initial data). Thus $\tn{tr}(p_n) = \tr_n$.

It is not difficult to check that $U_i p_n = p_n U_i =0$, $i<n$.
(The most interesting case is $U_{n-1}$.) Consult
\cite{kauffmanlins} or \cite{turaev} for details.
\end{proof}

The idempotent $p_n$ is called the {\it Jones-Wenzl idempotent},
or the Jones-Wenzl projector, and plays an indispensable role in
the pictorial approach to TQFTs.

\begin{theorem}
(1): If $d\in \C$ is not a root of Chebyshev polynomials $\Delta_i,
1\leq i \leq n$, then the TL algebra $\tn{TL}_n(d)$ is semisimple.

(2): Fixing an integer $r\geq 3$, a non-zero number $d$ is a root
of $\tr_i$, $i<r$ if and only if $d=-A^2-A^{-2}$ for some $A$ such
that $A^{4l} =1,l\leq r$.  If $d=-A^2 -A^{-2}$ for a primitive
$4r$-th root of unity $A$ for some $r\geq 3$ or a primitive $2r$th
or $r$th for $r$ odd, then the TL algebras $\{\tn{TL}_n(d)\}$
modulo
 the Jones-Wenzl idempotent $p_{r-1}$ are semi-simple.
\end{theorem}

\begin{proof}

(1):  $\TL_n(d)$ is a $*$-algebra.  By formula
\ref{meanderformula}, the determinant of the Markov trace pairing
is $\prod_{i=1}^n{\Delta_i(d)}^{a_{n,i}}$, hence the $*$-structure
is non-degenerate.  By Lemma \ref{starsemi}, $\TL_n(d)$ is
semi-simple.

(2):  The first part follows from $\Delta_n(d)=(-1)^n
\frac{A^{2n+2}-A^{-2n-2}}{A^2-A^{-2}}.$  In Section
\ref{repoftlj}, we will show that the kernel of the Markov trace
Hermitian pairing is generated by $p_{r-1}$, and the second part
follows.

\end{proof}

The semi-simple quotients of $\TL_n(d)$ in the above theorem will
be called the Temperley-Lieb-Jones (TLJ) algebras or just Jones
algebras, denoted by $\TLJ_n(d)$. The TLJ algebras are semi-simple
algebras over $\C$, therefore it is isomorphic to a direct sum of
matrix algebras, i.e.,
\begin{equation}\label{jonesdecom}
\TLJ_n(d)\cong \oplus_i
Mat_{n_i}(\C).
\end{equation}
As in the generic Jones representation case, the Kauffman bracket
followed by the decomposition yields a representation of the braid
groups.

\begin{prop}

(1): When the Markov trace Hermitian paring is $\pm$-definite, then
Jones representations are unitary, but reducible.  When $A=\pm i
e^{\pm \frac{2\pi i}{4r}}$, the Markov trace Hermitian pairing is
$+$-definite for all $n$'s.

(2): Given a braid $\sigma\in B_n$, the Markov trace is a weighted
trace on the matrix decomposition \ref{jonesdecom}, and when
multiplied by $(-A)^{-3\sigma}$ results in the Jones polynomial of
the braid closure of $\sigma$ evaluated at $q=A^4$.

\end{prop}

Unitary will be established in Section
\ref{classificationunitary}, and reducibility follows from the
decomposition \ref{jonesdecom}. That the Markov trace, normalized
by the framing-dependence factor, is the Jones polynomial follows
from direct verification of invariance under Reidermeister moves
or Markov's theorem (see e.g. \cite{kauffmanlins}).

\subsection{Uniqueness of Jones-Wenzl projectors}\label{tldproofsection}
Fix an $r\geq 3$ and a primitive $4r$th root of unity or a
primitive $2r$th or $r$th root of unity for $r$ odd, and
$d=-A^2-A^{-2}$. In this section, we prove that $\TL_d$ has a
unique ideal generated by $p_{r-1}$. When $A$ is a primitive
$4r$th root of unity, this is proved in the Appendix of
\cite{freedmanmag} by F.~ Goodman and H.~ Wenzl. Our elementary
argument works for all $A$ as above.

Notice that $\TL_d$ admits the structure of a (strict) monoidal
category, with the tensor product given by horizontal
``stacking'', e.g., juxtaposition of diagrams. This tensor product
(denoted $\otimes$) is clearly associative, and $1_0$, the
identity on $0$ vertices or the empty object, serves as a unit.
The tensor product and the original algebra product on $\TL_d$
satisfy the interchange law,
$ (f\otimes g)\cdot(f'\otimes g') = (f\cdot f')\otimes(g\cdot g'), $
whenever the required vertical composites are defined.

We may use this notation to recursively define the projectors
$p_k$:
$ p_{k+1} = p_k\otimes 1_1 - \mu_k(p_k\otimes 1_1)U_k^{k+1}
    (p_k\otimes 1_1).$
We define $p_0=1_0$, $p_1=1_1$ and
$\mu_k=\frac{\Delta_{k-1}}{\Delta_k}$. Using this we can prove a
sort of ``decomposition theorem'' for projectors:
\begin{prop}
$p_k=\left(\bigotimes_{i=1}^{\lfloor\frac{k}{r}\rfloor}p_r\right)
\otimes p_{(k\bmod r)}$.
\end{prop}
\begin{proof} We proceed by induction, using the recursive definition
of the Jones-Wenzl projectors. For $p_1$ the statement is trivial.
Assuming the assertion holds for $p_k$, we then have (let
$m=k\bmod r$):
\begin{eqnarray*} p_{k+1} & = &
   p_k\otimes 1_1 - \mu_k(p_k\otimes 1_1)U_k^{k+1}
    (p_k\otimes 1_1) \\
 & = & (\bigotimes_{i=1}^{\lfloor\frac{k}{r}\rfloor}p_r)\otimes
    p_{m}\otimes 1_1 - \mu_k(
    (\bigotimes_{i=1}^{\lfloor\frac{k}{r}\rfloor}p_r)\otimes
    p_{m}\otimes 1_1)U_k^{k+1}((\bigotimes_{i=1}^{\lfloor\frac{k}{r}
    \rfloor}p_r)\otimes p_{m}\otimes 1_1) \\
& & \hspace{-1.5cm}\text{Then, if $m\ne0$,} \\
 & = & (\bigotimes_{i=1}^{\lfloor\frac{k}{r}\rfloor}p_r)\otimes p_{m}
    \otimes 1_1 - \\
 &   & \hspace{1.5cm} \mu_k((\bigotimes_{i=1}^{\lfloor\frac{k}{r}\rfloor}p_r)
    \otimes p_{m}\otimes 1_1)(1_{k-m}\otimes
    U_{m}^{m+1})((\bigotimes_{i=1}^{\lfloor\frac{k}{r}
    \rfloor}p_r)\otimes p_{m}\otimes 1_1) \\
 & = & (\bigotimes_{i=1}^{\lfloor\frac{k}{r}\rfloor}p_r)\otimes\left(
    p_{m}\otimes 1_1 - \mu_{m}(p_{m}\otimes1_1)
    U_{m}^{m+1}(p_{m}\otimes1_1)\right) \\
 &   & \hspace{-1.5cm}\text{(The $\lfloor\frac{k}{r}\rfloor$ copies of $p_r$ can be
    factored out of the second term by $p_rp_r=p_r$.)} \\
 & = & (\bigotimes_{i=1}^{\lfloor\frac{k}{r}\rfloor}p_r)\otimes p_{m+1}
\end{eqnarray*}
If $m<r-1$, then $m+1=(k+1)\bmod r$; if $m=r-1$, then we get one
more copy of $p_r$, as needed. So it remains to consider the case
above where $k\bmod r=0$. But then $\mu_k=\mu_{k\bmod r}=\mu_0=0$,
so that
\[ p_{k+1} =  (\bigotimes_{i=1}^{\frac{k}{r}}p_r)\otimes 1_1 =
    (\bigotimes_{i=1}^{\lfloor\frac{k+1}{r}\rfloor}p_r)\otimes p_1 \]
as desired.
\end{proof}

In analogy with the standard notion from ring theory, an ideal in
$\TL$ is defined to be a class of morphisms which is internally
closed under addition, and externally closed under both the
vertical product (composition) $\cdot$ and the horizontal product
$\otimes$. Given such an ideal $I$, we may form the quotient
category $\TL/I$, which has the same objects as $\TL$, and
hom-sets formed by taking the usual quotient of
$\Hom(\mathbf{m},\mathbf{n})$ by those morphisms in
$I\cap\Hom(\mathbf{m},\mathbf{n})$.

We can prove that $\grpgen{p_{r-1}}$ is an ideal.
\begin{lemma}\label{properideal} The
ideal $\rd=\grpgen{p_{r-1}}$ is a proper ideal.
\end{lemma}
\begin{proof} It suffices to show that the $\otimes$-identity $1_0$ is
not in the ideal. In order for $1_0$ to be in the ideal, it would
have to be obtained from some closed network (e.g., element of
$\Hom(\mathbf{0},\mathbf{0})$) which contains at least one copy of
$p_{r-1}$. Fixing such a projector, we expand all other terms in
the network (this includes getting rid of closed loops), so that
we are left with a linear combination of closed networks, each
having exactly one $r-1$ strand projector. Now, considering each
term seperately, if there are any strands that leave and re-enter
the projector on the same side, then the network is null (since
$p_{r-1}U_i^{r-1}=0$). So the only remaining terms will be strand
closures of $p_{r-1}$; but by the above, these are null as well,
so that every term in the expansion vanishes.

Since every closed network with $p_{r-1}$ is null, it follows that
$1_0\not\in\rd$, and therefore $\rd$ is a proper ideal of $\TL$.
\end{proof}

In fact, this same ideal is generated by any $p_k$ for $k\ge r-1$;
this is established via a sequence of lemmas.

\begin{lemma}
$\grpgen{p_r}=\grpgen{p_{r-1}}$
\end{lemma}

\begin{proof} It is clearly sufficient to show
$p_{r-1}\in\grpgen{p_r}$.  Set $x=p_r\otimes1_1$, and expand $p_r$
in terms of $p_{r-1}$ according to the recursive definition. Then
connect the rightmost two strands in a loop: e.g., by pre- and
post-multiplying by the appropriate elements of
$\Hom(\bf{r-1},\bf{r+1})$ and $\Hom(\bf{r+1},\bf{r-1})$,
respectively.  Using the fact that $p_{r-1}p_{r-1}=p_{r-1}$, the
resulting diagram simplifies to $(d-\mu_{r-1})p_{r-1}$; and since
$\mu_{r-1}\ne d$, the coefficient is invertible, so that
$p_{r-1}\in\grpgen{p_r}$.
%\begin{figure}[h]\includegraphics[scale=.7]{tld-r-proj.eps}\caption{Obtaining $p_{r-1}$ from $p_r$.}\end{figure}
\end{proof}

\begin{lemma}
For any integer $k\ge1$, $\grpgen{p_{kr}}=
\grpgen{p_{r-1}}$.
\end{lemma}

\begin{proof} By induction; the base case is established in the
previous Lemma. For $k\ge2$, we can write
$p_{kr}=p_{(k-2)r}\otimes p_r\otimes p_r$, and then consider the
tangle $p_{kr}\otimes 1_r$. By again pre- and post-multiplying by
appropriate tangles, and using $p_rp_r=p_r$, we see that
$p_{(k-1)r}\in\grpgen{p_{kr}}$.
 \iffalse
\begin{figure}[h]
\includegraphics[scale=.7]{tld-kr-proj.eps}
\caption{Obtaining $p_{(k-1)r}$ from $p_{kr}$.}
\end{figure}\fi \end{proof}

\begin{lemma}\label{projgenrd}
For any $k\ge r-1$,
$\grpgen{p_k}=\grpgen{p_{r-1}}$.
\end{lemma}

\begin{proof} This basically uses the same technique as the previous
lemma, combined with the fact that $p_r(p_l\otimes1_{r-l})=p_r$
(see \cite{kauffmanlins}).

Let $m=k\bmod r$; if $m=0$, then this falls under the case of the
previous lemma, so $0<m<r$. Now consider $x=p_k\otimes1_{2r-m}$;
we can use the technique of the previous lemma to merge the last
three groups of $r$ strands into one, so that the resulting
element $x'=p_{\lfloor\frac{k}{r}\rfloor r}\otimes(p_r(p_l\otimes
1_{r-l})1_r)$. But $p_r(p_l\otimes1_{r-l})=p_r$, so that
$x'=p_{(\lfloor\frac{k}{r}\rfloor+1)r}$, whence, by the previous
lemma, $\grpgen{p_{r-1}}=\grpgen{p_k}$.
\end{proof}

Thus, in the quotient category $\TL/\rd$, all $k$-projectors, for
$k\ge r-1$, are null.

%\subsubsection{The quotient $\TL/\rd$}

We have shown that $\rd$ is an ideal; our strategy in showing that
$\rd$ is unique will be to show that it has no proper ideals, and
that the quotient $\TL/\rd$ has no nontrivial ideals. To show the
latter fact, we will show that the ideal (in the quotient)
generated by any element is in fact all of $\TL/\rd$.

We note also that $\TL/\rd$ may be described succinctly as the
subcategory of $\TL$ whose tangles have less than $r-1$
``through-passing'' strands. This subcategory does not close under
$\otimes$ as described above, but can be shown to be well-defined
under the reduction $1_{r-1}\leadsto (1_{r-1}-p_{r-1})$. This view
is not necessary in what follows, so we do not pursue it further;
but it may be useful in thinking about the quotient category.

A preliminary observation is that $\TL/\rd$ has no zero divisors:
\begin{lemma} Let $x,y\in\TL/\rd$. If $x\otimes y=0$, then $x=0$ or
$y=0$. \end{lemma}
\begin{proof} The statement clearly holds in $\TL$; so the only way
it could fail in the quotient is if $p_{r-1}$ had a tensor
decomposition.

So, suppose, $x\otimes y=p_{r-1}$, where $x$ is a tangle on $k>0$
strands and $y$ is a tangle on $l>0$ strands, both nontrivial
(that $\textrm{dom} (x)=\textrm{cod} (x)$ and $\textrm{dom}
(y)=\textrm{cod} (y)$ follows from the fact that
$p_{r-1}p_{r-1}=p_{r-1}$). Then the properties of projectors and
the interchange law give:
\[ x\otimes y=(x\otimes y)(x\otimes y)=xx\otimes yy \Longrightarrow
xx=x, yy=y \] Further $(x\otimes y)U^{k+l}_i=0$ for all $i$, so
that $xU^k_i\otimes y=0 \Longrightarrow xU^k_i=0$, and likewise
$yU^l_i=0$. Thus both $x$ and $y$ are projectors. But the strand
closure of $p_k\otimes p_l$ is $\Delta_k\Delta_l$, which are both
nonzero, and the strand closure of $p_{r-1}$ is zero, so we have
reached a contradiction.
\end{proof}

\iffalse {\bf n.b.}: is there a quick argument to show that $x$
and $y$ can't be in, say, $\Hom(k,l)$ and $\Hom(r-1-k,r-1-l)$, for
$k\ne l$.?\fi

The next lemma introduces an algorithm that is the key to the rest
of the proof:

\begin{lemma}\label{homr-3} Any nonzero
ideal $I\subset\TL/\rd$ contains at least one element of
$\Hom(\bf{r\!-\!3},\bf{r\!-\!3})$. \end{lemma}
\begin{proof} Let $x\ne0\in I$, say $x\in\Hom(\bf{m},\bf{n})$. First,
if $m\ne n$, then we can tensor with the unique basis element in
either $\Hom(\bf{0}, \bf{2})$ or $\Hom(\bf{2},\bf{0})$ the
appropriate number of times so that we get an
$x'\in\Hom(\bf{k},\bf{k})\in I$, where $k=\max\{m,n\}$. (By the
previous lemma, $x'\ne0$.) If $k\leq r-3$, then
$x'\otimes1_{r-3-k}\in\Hom(\bf{r-3}, \bf{r-3})$ is an element of
the ideal; so it remains to show the case where $k>r-3$.

First, assume $k$ and $r-3$ have the same parity; if not, use
$x'\otimes1_1$ instead of $x'$. Then let $k_0=k,x'_0=x'$, and use
the following algorithm (starting with $i=0$):
\begin{enumerate}

\item If $k_i=r-3$, then stop: $x'_i\in\Hom(\bf{r-3},\bf{r-3})$ is
in the ideal.

\item Since $k_i\geq r-1$, and $x'_i\ne0$, it follows that
$x'_i\ne\alpha p_{k_i}$, since all $r-1$ and above projectors are
null in $\TL/\rd$. Recall that $p_{k_i}$ is the unique element in
$\Hom(\bf{k_i},\bf{k_i})$ such that (i)
$U_j^{k_i}p_{k_i}=p_{k_i}U_j^{k_i}=0$ for $1\leq j<k_i$; and (ii)
$p_{k_i}p_{k_i}=p_{k_i}$. From this it follows that the only
elements which satisfy (i) are $\alpha p_{k_i}$, for some
$\alpha\in\C$. Therefore, since $x'_i\ne\alpha p_{k_i}$, there
exists some $U_i=U_{j_i}^{k_i}$ such that $U_ix'\ne0$.

\item Using an argument similar to the above, there exists some
$U'_i= U_{j'_i}^{k_i}$ such that $(U_ix')U'_i\ne0$.

\item Set $V_i$ to be the unique basis element in
$\Hom(\bf{k_i-2},\bf{k_i})$ which connects the $j_i$ and $j_i+1$
vertices on the top (codomain) objects, and connects the remaining
$k-2$ vertices on top and bottom to each other. Then
$V_iU_ix'U'_i$ can be described as being exactly like $U_ix'U'_i$,
except that the top half-loop of the $U_i$ has been factored out
as $d$, thus reducing the domain object by two vertices. It is
thus clear that $V_iU_ix'U'_i\ne0$.

\item Similarly, choose $V'_i$ to be the unique element in
$\Hom({\bf k_i},{\bf k_{i-2}})$ connecting the $j'_i$ and $j'_i+1$
vertices of the domain object, thus closing the half loop of
$U'_i$. Then $V_iU_ix'U'_iV'_i\ne0$.

\item Set $x'_{i+1}=V_iU_ix'U'_iV'_i$, $k_{i+1}=k_i-2$, and return
to step (1).

\end{enumerate}

After $j=\frac{1}{2}(k_0-(r-3))$ passes through the algorithm, the
desired element $x'_j\in I$ is produced.
\end{proof}

The proof of the previous Lemma is useful in establishing that
$\rd$ has no proper sub-ideals.

\begin{lemma}\label{nosubideal} For any $x\in\rd, x\neq 0$, then
$\grpgen{x}=\rd$. \end{lemma}

\begin{proof} Use the techniques previous Lemma to get an element
$x'\in\grpgen{x}$ such that $x'\in\Hom(\bf{k},\bf{k})$, and
$k\equiv r-1\bmod 2$. Then follow the algorithm, except for on
steps (2) and (3): for, since $x'_i\in\rd$, it is possible that
$x'_i=\alpha p_{k_i}$. If this is not the case, proceed with the
algorithm as it is stated. However, if $x'_i=\alpha p_{k_i}$, then
it follows that $\grpgen{x}=\rd$, by Lemma \ref{projgenrd}. So it
only remains to show that this does happen at some point before
the algorithm terminates: e.g., that for some $i$, $x'_i=\alpha
p_{k_i}$.

But, suppose this didn't happen; then, the algorithm goes through
to completion, yielding an element $y\in\grpgen{x}$ such that
$y\in\Hom(\bf{r-3},\bf{r-3})$, $y\ne0$. But then $y\not\in\rd$,
since every nonzero element of $\rd$ must have at least $r-1$
strands. This contradicts the fact that
$y\in\grpgen{x}\subset\rd$; therefore, there must be some $i$ such
that $x'_i=\alpha p_{k_i}$, and so the lemma follows.
\end{proof}

Now we can put all of this together to obtain our desired result:

\begin{theorem} $\TL_d$ has a unique
proper nonzero ideal when $A$ is as in Lemma \ref{jwcases}.
\end{theorem}
\begin{proof} By Lemma \ref{properideal}, $\rd=\grpgen{p_{r-1}}$ is a
proper ideal, which, by Lemma \ref{nosubideal}, has no proper
sub-ideals. To prove the theorem, therefore, it suffices to show
that the quotient category $\TL/\rd$ has no proper nonzero ideals.

Consider $\grpgen{x}$, for any $x\in\TL/\rd$. By Lemma
\ref{homr-3}, there exists some $y\in\grpgen{x}$ such that
$x\in\Hom(\bf{r-3},\bf{r-3})$. But now, instead of stopping at
this point in the algorithm, we continue the loop, with the
possibility that $x'_i$ might actually be a projector. So we again
modify steps (2) and (3), as below:
\begin{itemize}
\item[(1')] If $k_i=0$, stop; set $x'=x'_i$. \item[(2')] If
$x'_i=\alpha p_{k_i}$ for some constant $\alpha$, then stop, with
$x'=x'_i$. Otherwise, proceed with step (2) of the original
algorithm. \item[(3')] If $U_ix'_i=\alpha p_{k_i}$ for some
constant $\alpha$, then stop, with $x'=x'_i$. Otherwise, proceed
with step (3) of the original algorithm.
\end{itemize}
So now, when the algorithm terminates, we are left with some
$x'\in\grpgen{x}$, with either: (a) $x'=\alpha1_0$, $\alpha\ne0$;
or (b) $x'=\alpha p_k$ for some $1\leq k<r-1$, $\alpha\ne0$. In
case (a), we have that $1_0\in\grpgen{x}$, so that
$\grpgen{x}=\TL/\rd$. In case (b), consider the element
$y=\frac{1}{\alpha}x'\otimes1_k=p_k\otimes1_k\in\grpgen{x}$. We
can then pre- and post-multiply by the elements of
$\Hom(\bf{0},\bf{2k})$ and $\Hom(\bf{2k},\bf{0})$, respectively,
which join the left group of $k$ strands to the right group of $k$
strands. In other words, the resulting element is simply
$\Delta_k$, the strand closure of $p_k$, times $1_0$. Since
$\Delta_k\ne0$ for $1\leq k\leq r-2$, it follows that
$1_0\in\grpgen{x}$, so that we still have $\grpgen{x}=\TL/\rd$.

So we have shown that $\TL/\rd$ has no proper nonzero ideals, and
therefore, that $\TL$ has the unique ideal $\rd$.
\end{proof}

\iffalse \begin{itemize} \item should we attempt to extend this
result to $E\TL(S)$ for an arbitrary surface $S$? Is there any
obvious equivalent? By taking the quotient by $p_r$, it seems we
want to say, ``replace any $r$ contiguous strands by $1_r-p_r$.''
Thus we are not taking the quotient by any single element/ideal
(e.g., there are two different $p_r$ on the torus corresponding to
the different essential annuli).
\end{itemize}

\fi

As a corollary, we have the following:

\begin{theorem}
1):If $d$ is not a root of any Chebyshev polynomial $\Delta_k,
k\geq 1$, then the Temperley-Lieb category $\TL_d$ is semisimple.

2): Fixing an integer $r\geq 3$, a non-zero number $d$ is a root
of $\tr_k$, $k<r$ if and only if $d=-A^2-A^{-2}$ for some $A$ such
that $A^{4l} =1,l\leq r$.  If $d=-A^2 -A^{-2}$ for a primitive
$4r$-th root of unity $A$ or $2r$-th $r$ odd or $r$-th $r$ odd for
some $r\geq 3$, then the tensor category $\TLJ_d$ has a unique
nontrivial ideal generated by
 the Jones-Wenzl idempotent $p_{r-1}$.  The quotient
 categories $\TLJ_d$ are semi-simple.
\end{theorem}

\section{Diagram TQFTs for closed
manifolds}\label{joneswenzlclosed}

\subsection{$\lq\lq$d-isotopy", local relation, and skein
relation}\label{localrelation}

Let $Y$ be an oriented compact surface, and $\g \subset Y$ be an
imbedded unoriented $1$-dimensional submanifold.  If $\partial Y
\neq \phi$ then fix a finite set $F$ of points on $\partial Y$ and
require $\partial \gamma = F$ transversely. That is, $\g$ a
disjoint union of non-crossing loops and arcs, a ``multi-curve".
Let $\scc$ be the set of such $\g$'s. To $\lq\lq$linearize" we
consider the complex span $\C [ \scc]$ of $\scc$, and then impose
linear relations. \iffalse , or dually we may take the functions
$\F = \C^\Cc$ and impose linear constraints. Because it is close
to the idea of $\lq\lq$amplitude of an eigenstate" in quantum
mechanics, we usually work with $\F$. \fi We \underline{always}
impose the $\lq\lq$isotopy" constraint ${\g}'= \g,$ if ${\g}'$ is
isotopic to $\g$.  We also always impose a constraint of the form
$\g \cup \mathbf{O} =d\cdot \g$ for some $d \in
\C\backslash\{0\}$, independent of $\g$ (see an example below that
we do not impose this relation). The notation $\g \cup \mathbf{O}$
means a multi curve made from $\g$ by adding a disjoint loop
$\mathbf{O}$ to $\g$ where $\mathbf{O}$ is $\lq\lq$trivial" in the
sense that it is the boundary of a disk $B^2$ in the interior of
$Y$. Taken together these two constraints are $\lq\lq d-$isotopy"
relation: $\g' - \frac{1}{d} \cdot (\g \cup \mathbf{O}) =0$ if
$\g'$ is isotopic to $\g$.

A diagram \emph{local relation} or just a local relation is a
linear relation on multicurves $\g_1, \ldots,\g_m$ which are
identical outside some disk $B^2$ in the interior of $Y$, and
intersect $\p B^2$ transversely. By a disk here, we mean a
topological disk, i.e., any diffeomorphic image of the standard
$2$-disk in the plane. Local relations are usually drawn by
illustrating how the $\g_i$ differ on $B^2$. So the $\lq\lq$isotopy" constraint has the
form:

\begin{figure}
\centering
\includegraphics[scale=1]{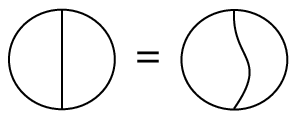}
\caption{Isotopy}\label{isotopy}
\end{figure}

\noindent and the $\lq\lq d-$constraint" has the form:

\begin{figure}
\centering
\includegraphics[scale=1]{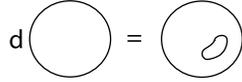}
\caption{d constraint}\label{dloop}
\end{figure}

Local relations have been explored to a great generality in
\cite{walker06} and encode information of topologically invariant
partition functions of a ball. We may filter a local relation
according to the number of points of $\g_i \cap \p B^2$ which may
be $0, 2, 4, 6, \ldots$ since we assume $\g$ transverse to $\p
B^2$. $\lq\lq$Isotopy" has degree $=2$ and $\lq\lq d$-constraint" degree
$=0$.

Formally, we define a local relation and a skein relation as
follows:

\begin{definition}

\begin{enumerate}

\item Let $\{D_i\}$ be all the diagrams on a disk $B^2$ up to
diffeomorphisms of the disk and without any loops.  The diagrams
$\{D_i\}$ are filtered into degrees $=2n$ according to how many
points of $D_i\cap \p B^2$, and there are Catalan number $c_n$
many diagrams of degree $2n$ ($c_0=1$ which is the empty diagram).
A degree $=2n$ diagram local relation is a formal linear equation
of diagrams $\sum_i c_i D_i=0$, where $c_i\in \C$, and $c_i=0$ if
$D_i$ is not of degree $=2n$.

\item  A skein relation is a resolution of over-/under-crossings
into formal pictures on $B^2$.  If the resolutions of crossings
for a skein relation are all formal diagrams, then the skein
relation induces a \emph{set} map from $\C[B_n]$ to $TL_n(d)$.

\end{enumerate}

\end{definition}

The most interesting diagram local relations are the Jones-Wenzl
projectors (the rectangle $\Rr$ is identified with a disk $B^2$ in
an arbitrary way). When we impose a local relation on $\C[\scc]$,
we get a quotient vector space of $\C[\scc]$ as follows: for any
multi-curve $\g$ and a disk $B^2$ in the interior of $\S$, if $\g$
intersects $B^2$ transversely and the part $\g\cap {B^2}$ of $\g$
in $B^2$ matches one of the diagram $D_j$ topologically in the
local relation $\sum_i c_i D_i=0$, and $c_j\neq 0$, then we set
$\g=-\sum_{i\neq j}\frac{c_i}{c_j}\g_i'$ in $\C[\scc]$ where
$\g_i'$ is obtained from $\g$ by replacing the part $\g\cap {B^2}$
of $\g$ in $B^2$ by the diagram $D_i$.

Kauffman bracket is the most interesting skein relation in this
paper. More general skein relations can be obtained from minimal
polynomials of $R$-matrices from a quantum group. Kauffman bracket
is an unoriented version of the $SU(2)_q$ case.

As a digression we describe an unusual example where we impose
$\lq\lq$isotopy" but \underline{not} the $\lq\lq d$-constraint".  It is
motivated by the theory of finite type invariants. A singular
crossing (outside $\scc$) suggests the $\lq\lq$type $1$" relation in
Figure \ref{homology}.

This relation is closely related to $\Z_2-$homology and is
compatible with the choice $d=1$. We will revisit it again under
the name $\Z_2-$gauge theory.

Now consider the $\lq\lq$type 2" relation Figure \ref{resolution}
which comes by resolving the arc in Figure \ref{singulararc} using
either arrow along the arc. (Reversing the arrow leaves the
relation Figure \ref{resolution} on unoriented diagrams
unchanged.)

\begin{figure}
\centering
\includegraphics[scale=1]{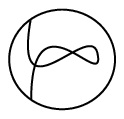}
\caption{Singular arc}\label{singulararc}
\end{figure}

\begin{figure}
\centering
\includegraphics[scale=1]{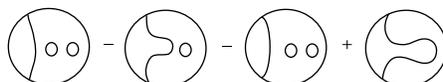}
\caption{Resolution relation}\label{resolution}
\end{figure}

Formally we may write the resolution relation Figure
\ref{resolution} as the square of the $2$ term relation drawn in
Figure \ref{twoterm}.

\begin{figure}
\centering
\includegraphics[scale=1]{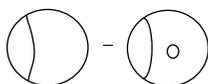}
\caption{Two term relation}\label{twoterm}
\end{figure}

Interpreting $\lq\lq$times" as $\lq\lq$vertical stacking" makes the claim
immediate as shown in Figure \ref{squaretwoterm}.

\begin{figure}
\centering
\includegraphics[scale=1]{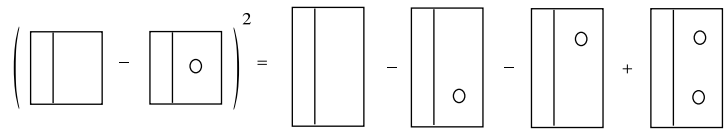}
\caption{Two term squared}\label{squaretwoterm}
\end{figure}

Since the two term relation Figure \ref{twoterm} does not appear
to be a consequence of the resolution relation Figure
\ref{resolution}, dividing by the resolution relation induces
nilpotence in the algebra (of degree $=2$ diagrams under vertical
stacking). By imposing the $\lq\lq d$" relation we find that only
semi-simple algebras are encountered. This is closer to the
physics (the simple pieces are symmetries of a fixed particle type
or $\lq\lq$super-selection sector") and easier mathematically so
henceforth we \underline{always} assume a $\lq\lq d-$constraint" for
some $d \in \C\backslash\{0\}$.

\subsection{Picture classes}

Fix a local relation $R=0$.  Given an oriented closed surface $Y$.
The vector space $\C[\scc]$ is infinitely dimensional.  We define
a finitely dimensional quotient of $\C[\scc]$ by imposing the
local relation $R$ as in last section:
 $\C[\scc]$ modulo the local relation.
 The resulting quotient vector space will be called the picture space, denoted as
 $\Pic^{R}(Y)$. Elements of $\Pic^{R}(Y)$ will be called picture classes.
 We will denote $\Pic^{R}(Y)$ as $\Pic(Y)$ when $R$ is clear or irrelevant
 for the discussion.

\begin{prop}

\begin{enumerate}

\item $\Pic(Y)$ is independent of the orientation of $Y$.

\item $\Pic(S^2)=\C \emptyset$, so it is either $0$ or $\C$.

\item $\Pic(Y_1 \amalg Y_2)\cong \Pic(Y_1)\otimes \Pic(Y_2)$.

\item $\Pic(Y)$ is a representation of the mapping class group
$\Mcg(Y)$.  Furthermore, the action of $\Mcg(Y)$ is compatible
with property (3).

\end{enumerate}

\end{prop}

\begin{proof}

Properties (1) (3) and (4) are obvious from the definition.  For
(2), since every simple closed curve on $S^2$ bounds a disk, a
multicurve with $m$ loops is $d^m \emptyset$ by $\lq\lq$d-isotopy".
Therefore, if $\emptyset$ is not $0$, it can be chosen as the
canonical basis.

\end{proof}

For any choice of $A\neq 0$, we may impose the Jones-Wenzl
projector as a local relation.  The resulting finitely dimensional
vector spaces $\Pic(Y)$ might be trivial.  For example, if we
choose a $d\neq \pm 1$ and impose the Jones-Wenzl projector
$p_2=0$ as the local relation. To see that the resulted picture
spaces $=0$, we reconnect two adjacent loops in a disk into one
using $p_2=0$; this gives the identity $(d^2-1)\emptyset=0$.  If
$d\neq \pm 1$, then $\emptyset=0$, hence $\Pic(S^2)=0$. Even if
$\Pic(Y)$'s are not $0$, they do not necessarily form a TQFT in
general.  We do not know any examples.  If exist, such non-trivial
vector spaces might have interesting applications because they are
representations of the mapping class groups. In the cases of
Jones-Wenzl projectors, only certain special choices of $A$'s lead
to TQFTs.

\subsection{Skein classes}

Fix a $d\in \C\backslash \{0\}$, a skein relation $K=0$ and a
local relation $R=0$. Given an oriented $3$-manifold $X$ (possibly
with boundaries). Let $\mathcal{F}$ be all the non-crossing loops
in $X$, i.e., all links $l$'s in the interior of $X$, and
$\C[\mathcal{F}]$ be their linear span. We impose the
$\lq\lq$d-isotopy" relation on $\C[\mathcal{F}]$, where a knot is
trivial if it bounds a disk in $X$.  For any $3$-ball $B^3$ inside
$X$ and a link $l$, the part $l\cap B^3$ of $l$ can be projected
onto a proper rectangle $\Rr$ of $B^3$ using the orientation of
$X$ (isotopy $l$ if necessary). Resolving all crossings with the
given skein relation $K=0$, we obtain a formal diagram in $\Rr$,
where the local relation $R=0$ can be applied. Such operations
introduce linear relations onto $\C[\mathcal{F}]$. The resulting
quotient vector space will be called the skein space, denoted by
$S_{d,K,R}(X)$ or just $S(X)$, and elements of $S(X)$ will be
called skein classes.

As mentioned in the introduction, the empty set $\emptyset$ has
been regarded as a manifold of each dimension.  It is also
regarded as a multicurve in any manifold $Y$ or a link in any $X$,
and many other things.  In the case of skein spaces, the empty
multicurve represents an element of the skein space $S(X)$. For a
closed manifold $X$, this would be the canonical basis if the
skein space $S(X)\cong \C$. But the empty skein is the $0$ vector
for some closed $3$-manifolds. In these cases, we do not have a
canonical basis for the skein space $S(X)$ even if $S(X)\cong \C$.

Skein spaces behave naturally with respect to disjoint union,
inclusion of spaces, orientation reversal, and
self-diffeomorphisms: the skein space of a disjoint union is
isomorphic to the tensor product; an orientation preserving
embedding from $X_1\rightarrow X_2$ induces a linear map from
$S(X_1)$ to $S(X_2)$, orientation reversal induces a
conjugate-linear map on $S(X)$, and diffeomorphisms of $X$ act on
$S(X)$ by moving pictures around, therefore $S(X)$ is a
representation of the orientation preserving diffeomorphisms of
$X$ up to isotopy.

\begin{prop}\label{algebrastructure}

\begin{enumerate}

\item If $Y$ is oriented, then $\Pic(Y)$ is an algebra.

\item If $\p X=Y$, then $\Pic(Y)$ acts on $S(X)$.  If $Y$ is
oriented, then $S(X)$ is a representation of $\Pic(Y)$.

\end{enumerate}
\end{prop}

\begin{proof}

(1):  Given two multicurves $x,y$ in $Y$, and consider $Y\times
[-1,1]$, draw $x$ in $Y\times 1$ and $y$ in $Y\times -1$.  Push
$x$ into the interior of $Y\times [0,1]$ and $y$ into $Y\times
[-1,0]$. Isotope $x,y$ so that their projections onto $Y\times 0$
are in general position. Resolutions of the crossings using the
given skein relation result in a formal multicurve in $Y$, which
is denoted by $xy$.  We define $[x][y]=[xy]$, where $[\cdot]$
denotes the picture class.  Suppose the local relation is $R=0$,
and let $\hat{R}$ be a multicurve obtained from the closure of $R$
arbitrarily outside a rectangle $\Rr$ where the local relation
resides. To show that this multiplication is well-defined, it
suffices to show that $\hat{R}y=0$. By general position, we may
assume that $y$ miss the rectangle $\Rr$. Then by definition,
$\hat{R}y=0$ no matter how we resolve the crossings away from the
local relation $R$.  It is easy to check that this multiplication
yields an algebra structure on $\Pic(Y)$.

(2):  The action is defined by gluing a collar of the boundary and
then re-parameterizing the manifold to absorb the collar. Let
$Y_{\e}$ be $Y\times [0,\e]$, which can be identified with a small
collar neighborhood of $Y$ in $X$. Given a multicurve $x$ in $X$
and $\g$ in $Y$, draw $\g$ on $Y\times 0$ and push it into
$Y_{\e}$. Then the union $\g \cup x$ is a multicurve in
$X_{+}=Y_{\e}\cup_{Y}X$. Absorbing $Y_{\e}$ of $X_{+}$ into $X$
yields a multicurve $\g\cup x$ in $X$, which is defined to be
$\g.x$.

\end{proof}

\subsection{Recoupling theory}\label{recouplingtheory}

In this section, we recall some results of the recoupling theory
in \cite{kauffmanlins}, and deduce some needed results for later
sections.

Fix a $A\in \C\backslash \{0\}$, two families of numbers are
important for us: the Chebyshev polynomials $\Delta_n(d)$ and the
quantum integers $[n]_A=\frac{A^{2n}-A^{-2n}}{A^2-A^{-2}}$.  When
$A$ is clear from the context, we will drop the $A$ from $[n]_A$.  The
Chebyshev polynomials and quantum integers are related by the
formula $\Delta_n(d)=(-1)^n[n+1]_A$.

Note that $[-n]_A=-[n]_A, [n]_{-A}=[n]_{\bar{A}}=[n]_A$,
$[n]_{iA}=(-1)^{n+1}[n]_A$.  Some other relations of quantum
integers depend on the order of $A$.

\begin{lemma}\label{jwcases}
Fix $r\geq 3$.

\begin{enumerate}
\item If $A$ is a primitive $4r$th root of unity, then
$[n+r]=-[n]$ and $[r-n]=[n]$.  The Jones-Wenzl projectors
$\{p_i\}$ exist for $0\leq i \leq r-1$, and
$\tn{Tr}(p_{r-1})=\Delta_r=0$.

\item If $r$ odd and $A$ is a primitive $2r$th root of unity, then
$[n+r]=[n]$ and $[r-n]=-[n]$.  The Jones-Wenzl projectors
$\{p_i\}$ exist for $0\leq i\leq r-1$, and
$\tn{Tr}(p_{r-1})=\Delta_r=0$.

\item If $r$ odd and $A$ is a primitive $r$th root of unity, then
$[n+r]=[n]$ and $[r-n]=-[n]$.  The Jones-Wenzl projectors
$\{p_i\}$ exist for $0\leq i \leq r-1$, and
$\tn{Tr}(p_{r-1})=\Delta_r=0$.
\end{enumerate}

\end{lemma}

The proof is obvious using the induction formula for $p_n$ in
Lemma \ref{joneswenzlprojectors}, and $[n]\neq 0$ for $0\leq n\leq
r-1$ for such $A$'s.

Fix an $r$ and $A$ as in Lemma \ref{jwcases}, and let $I$ be the
range that $p_i$ exists and $\tn{Tr}(p_i)\neq 0$.  Let
$L_A=\{p_i\}_{i\in I}$, then $I=\{0,1,\cdots, r-2\}$.  Both $L_A$
and $I$ will be called the label set.  Note that if $A$ is a
primitive $2r$th root of unity and $r$ is even, then $\{p_i\}$
exist for $0\leq i \leq \frac{r-2}{2}$, and
$\tn{Tr}(p_{\frac{r-2}{2}})=0$.

Given a ribbon link $l$ in $S^3$, i.e. each component is a thin
annulus, also called a framed link, then the Kauffman bracket of
$l$, i.e. the Kauffman bracket and $\lq\lq$d-isotopy" skein class of
$l$, is a framed version of the Jones polynomial of $l$, denoted
by $<l>_A$.  The Kauffman bracket can be generalized to colored ribbon
links:  ribbon links that each component carries a label from
$L_A$; the Kauffman bracket of a colored ribbon link $l$ is the
Kauffman bracket of the formal ribbon link obtained by replacing
each component $a$ of $l$ with its label $p_i$ inside the ribbon
$a$ and thickening each component of $p_i$ inside $a$ into small
ribbons. Since $S^3$ is simply-connected, the Kauffman bracket of
any colored ribbon link is a Laurent polynomial in $A$, hence a
complex number.

Let $H_{ij}$ be the colored ribbon Hopf link in the plane labelled
by Jones-Wenzl projectors $p_i$ and $p_j$, then the Kauffman
bracket of $H_{ij}$ is
\begin{equation}\label{Smatrix}
\ts_{ij}=(-1)^{i+j}[(i+1)(j+1)]_A.
\end{equation}
 The matrix
$\ts=(\ts_{ij})_{i,j\in I}$ is called the modular $\ts$-matrix.
Let $\ts_{\tn{even}}$ be the restriction of $\ts$ to even labels.
Define $\bar{i}=k-i=r-2-i$.

\begin{lemma}\label{smatrix}

\begin{enumerate}

\item If $A$ is a primitive $4r$th root of unity, then the modular
$\ts$ matrix is non-singular.

\item If $r$ is odd and $A$ is a primitive $2r$th or $r$th root of
unity, then $\ts_{\bar{i}j}=\ts_{ij}$.

\item If $r$ odd, and $A$ is a primitive $2r$th root of unity or
$r$th root of unity, then the modular $\ts$ has rank
$=\frac{r-1}{2}$. Moreover, $\ts=\ts_{\tn{even}}\otimes
\begin{pmatrix}
   1& 1 \\
   1 & 1
 \end{pmatrix}$.

\end{enumerate}

\end{lemma}

\begin{proof}

Since $\ts$ is a symmetric real matrix, so the rank of $\ts$ is
the same as $\ts^2$.

By the formula \ref{smatrix}, we have $(\ts^2)_{ij}$
$$=\frac{(-1)^{i+j}}{(A^2-A^{-2})^2} \sum_{l=0}^{r-2}
[A^{2(i+1)(l+1)}-A^{-2(i+1)(l+1)}][A^{2(l+1)(j+1)}-A^{-2(l+1)(j+1)}]$$
$$=\frac{(-1)^{i+j}}{(A^2-A^{-2})^2} \sum_{l=0}^{r-2}
[A^{2(i+1)(l+1)+2(l+1)(j+1)}+A^{-2(i+1)(l+1)-2(l+1)(j+1)}$$
$$-A^{2(i+1)(l+1)-2(l+1)(j+1)}-A^{-2(i+1)(l+1)+2(l+1)(j+1)}].$$

The first sum $\sum_{l=0}^{r-2} A^{2(i+1)(l+1)+2(l+1)(j+1)}$ is a
geometric series
$=\frac{A^{2(i+j+2)}-(A^{2(i+j+2)})^{r}}{1-A^{2(i+j+2)}}$ if
$A^{2(i+j+2)}\neq 1$.  The second sum $\sum_{l=0}^{r-2}
A^{-2(i+1)(l+1)-2(l+1)(j+1)}$ is the complex conjugate of the
first sum.

The third sum $-\sum_{l=0}^{r-2} A^{2(i+1)(l+1)-2(l+1)(j+1)}$ is
also a geometric series
$=-\frac{A^{2(i-j)}-(A^{2(i-j)})^{r}}{1-A^{2(i-j)}}$ if
$A^{2(i-j)}\neq 1$.  The $4$th sum $-\sum_{l=0}^{r-2}
A^{-2(i+1)(l+1)+2(l+1)(j+1)}$ is the complex conjugate of the
third sum.

If $A$ is a $4r$th primitive, since $0\leq i,j \leq r-2$, we have
$4\leq 2(i+j+2)\leq 4r-4$ and $-(r-2)\leq i-j\leq r-2$. Hence,
$A^{2(i+j+2)}\neq 1$ and $A^{2(i-j)}\neq 1$ unless $i=j$. The
first sum and the second sum add to
$\frac{A^{2(i+j+2)}-(-1)^{i+j}-1+(-1)^{i+j}A^{2(i+j+2)}}{1-A^{2(i+j+2)}}$.
So if $i+j$ is even, then $=-2$; if $i+j$ is odd, then $=0$.  If
$A^{2(i-j)}\neq 1$, then the third and $4$th add to
$-\frac{A^{2(i-j)}-(-1)^{i-j}-1+(-1)^{i-j}A^{2(i-j)}}{1-A^{2(i-j)}}$.
It is $=2$ if $i-j$ is even and $=0$ if $i-j$ odd.  Therefore, if
$i\neq j$, the four sums add to $0$ and if $i=j$, then they add to
$-2-2(r-1)=-2r$.  It follows that $\ts^2$ is a diagonal matrix
with diagonal entries $=\frac{-2r}{(A^2-A^{-2})^2}$.

If $A$ is a $2r$th or $r$th primitive, if $A^{2(i+j+2)}\neq 1$,
then the first sum $-1$.  The second sum is also $-1$ since it is
the complex conjugate.  If $A^{2(i-j)}\neq 1$, then the third sum
is $=1$ and so is the $4$th sum.  It follows that if neither
$A^{2(i+j+2)}=1$ nor $A^{2(i-j)}=1$, then the $(i,j)$th entry of
$\ts^2$ is $0$.

If $A^{2(i+j+2)}=1$, then $2(i+j+2)=r,2r,3r$ as $0\leq i,j \leq
r-2$ and $4\leq 2(i+j+2)\leq 4r-4$.  When $r$ is odd,
$2(i+j+2)=2r$, so $i=\bar{j}$. Therefore, if $i+j+2\neq r$, the
first and second sum is $-1$.  If $i+j+2=r$, then the first and
second sum both are $r-1$.  If $A^{2(i-j)}=1$, then
$2(i-j)=-r,0,r$ as $-2(r-2)\leq 2(i-j)\leq 2(r-2)$.  It follows
that $i=j$ as $r$ is odd.  If $i=j$, the third and $4$th sum both
are $=-(r-1)$. Put everything together, we have if $i\neq j$ or
$i+j\neq r-2$, then $(\ts^2)_{ij}=0$.  If $i=j$, then $i+j\neq
r-2$ because $r-2$ is odd, and
$(\ts^2)_{ij}=\frac{-2r}{(A^2-A^{-2})^2}$. If $i+j=r-2$, then
$i\neq j$, and $(\ts^2)_{ij}=\frac{-2r}{(A^2-A^{-2})^2}$. Hence
$\ts^2=\frac{-2r}{(A^2-A^{-2})^2} (m_{ij})$, where $m_{ij}=0$
unless $i=j$ or $i+j=k=r-2$.

\end{proof}

We define a colored tangle category $\Delta_A$ based on a label
set $L_A$. Consider $\C\times I$, the product of the plane $\C$ with an interval $I$,
the objects of $\Delta_A$ are
finitely many labelled points on the real axis of $\C$ identified
with $\C \times \{0\}$ or $\C \times \{1\}$. A morphism between
two objects are formal tangles in $\C\times I$ whose arc
components connect the objects in $\C \times \{0\}$ and $\C \times
\{1\}$ transversely with same labels, modulo Kauffman bracket and
Jones-Wenzl projector $p_{r-1}$. Horizontal juxtaposition as a
tensor product makes $\Delta_A$ into a strict monoidal category.

The quantum dimension $d_i$ of a label $i$ is defined to the
Kauffman bracket of the $0$-framed unknot colored by the label
$i$.  So $d_i=\Delta_i(d)$.  The total quantum order of $\Delta_A$
is $D=\sqrt{\sum_i d_i^2}$, so
$D=\sqrt{\frac{-2r}{(A^2-A^{-2})^2}}$. The Kauffman bracket of the
$1$-framed unknot is of the form $\theta_id_i$, where
$\theta_i=A^{-i(i+2)}$ is called the twist of the label $i$.
Define $p_{\pm}=\sum_{i\in I}\theta_i^{\pm 1}d_i^2$, then
$D^2=p_{+}p_{-}$.

A triple $(i,j,k)$ of labels is admissible if $\Hom(p_i\otimes
p_j,p_k)$ is not $0$.  The theta symbol $\theta(i,j,k)$ is the
Kauffman bracket of the theta network, see \cite{kauffmanlins}.

\begin{lemma}

\begin{enumerate}

\item  $\Hom(p_i\otimes p_j,p_k)$ is not $0$ if and only if the
theta symbol $\theta(i,j,k)$ is non-zero, then $\Hom(p_i\otimes
p_j,p_k)\cong \C$.

\item $\theta(i,j,k) \neq 0$ if and only if $i+j+k\leq 2(r-2)$,
$i+j+k$ is even and $i+j\geq k, j+k\geq i, k+i\geq j$.

\end{enumerate}

\end{lemma}

\subsection{Handles and $S$-matrix}

There are various ways to present an $n$-manifold $X$:
 triangulation, surgery, handle decomposition, etc. The
convenient ways for us are the surgery description and handle
decompositions.

Handle decomposition of a manifold $X$ comes from from a Morse
function of $X$.  Fix a dimension$=n$, a $k$-handle is a product
structure $I^k\times I^{n-k}$ on the $n$-ball $B^n$, where the
part of boundary $\p I^k\times I^{n-k}\cong S^{k-1}\times I^{n-k}$
is specified as the attaching region.  The basic operations in
handlebody theory are handle attachment, handle slide,
stabilization, and surgery. They correspond to how Morse functions
pass through singularities in the space of smooth functions on
$X$. Let us discuss handle attachment and surgery here.  Given an
$n$-manifold $X$ with a sub-manifold $S^{k-1}\times I^{n-k}$
specified in its boundary, and an attach map $\phi: \p I^k \times
I^{n-k} \rightarrow S^{k-1}\times I^{n-k}$, we can attach a
$k$-handle to $X$ via $\phi$ to form a new manifold
$X'=X\cup_{\phi} I^k\times I^{n-k}$.  The new manifold $X'$
depends on $\phi$, but only on its isotopy class.  It follows from
Morse theory or triangulation that every smooth manifold $X$ can
be obtained from $0$-handles by attaching handles successively,
i.e., has a handle decomposition.  Moreover, the handles can be
arranged to be attached in the order of their indices, i.e., from
$0$-handles, first all $1$-handles attached, then all
$2$-handles, etc.

Given an $n$-manifold $X$, a sub-manifold $S^{k} \times I^{n-k}$
and a map $\phi: \p I^{k+1} \times S^{n-k-1} \rightarrow
S^{k}\times S^{n-k-1}$, we can change $X$ to a new manifold
$X'$ by doing index $k$ surgery on $S^{k} \times I^{n-k}$ as
follows:  delete the interior of $S^{k} \times I^{n-k}$, and glue
in $I^{k+1} \times S^{n-k-1}$ via $\phi$ along the common boundary
$S^{k}\times S^{n-k-1}$. Of course the resulting manifold $X'$
depends on the map $\phi$, but only on its isotopy class. Handle
decompositions of $n+1$-manifolds are related to surgery of
$n$-manifolds as their boundaries.

It is fundamental theorem that every orientable closed
$3$-manifold can be obtained from surgery on a framed link in
$S^3$; moreover, if two framed links give rise to the same
$3$-manifold, they are related by Kirby moves, which consist of
stabilization and handle slides.  This is extremely convenient for
constructing $3$-manifold invariants from link invariants: it
suffices to write down a magic linear combination of invariants of
the surgery link so that the combination is invariant under Kirby
moves.  The Reshetikhin-Turaev invariants were discovered in this
way.

The magic combination is provided by the projector $\omega_0$ from
the first row of the $S$ matrix: given a surgery link $L$ of a
$3$-manifold, if every component of $L$ is colored by $\omega_0$,
then the resulting link invariant is invariant under handle
slides.  Moreover, a certain normalization using the signature of
the surgery link produces a $3$-manifold invariant as in Theorem
\ref{invariantformula} below.

The projector $\omega_0$ is a ribbon tensor category analogue of
the regular representation of a finite group, and is related to
surgery as below. In general, all projectors $\omega_i$ are
related to surgery in a sense, which is responsible for the gluing
formula for the partition function $Z$ of a TQFT.

\begin{lemma}

\begin{enumerate}

\item Given a $3$ manifold $X$ with a knot $K$ inside, if $K$ is
colored by $\omega_0$, then the invariant of the pair $(X,K)$ is
the same as the invariant of $X'$, which is obtained from $X$ by
$0$-surgery on $K$.

\item  Let $S^2\subset X$ be an embeded $2$-sphere, then any
labeled multicurve $\g$ interests $S^2$ transversely must carry
the trivial label.  In other words, non-trivial particle type
cannot cross an embeded $S^2$.

\end{enumerate}

\end{lemma}

The colored tangle category $\Delta_A$ has natural braidings, and
duality, hence is a ribbon tensor category.  An object $a$ is
simple if $\Hom(a,a)=\C$.  A point marked by a Jones-Wenzl
projector $p_i$ is a simple object of $\Delta_A$.  Therefore, the
label set $L_A$ can be identified with a complete set of simple
object representatives of $\Delta_A$.  A ribbon category is
premodular if the number of simple object classes is finite, and
is called modular if furthermore, the modular $S$-matrix
$S=\frac{1}{D}\ts$ is non-singular.  A non-singular $S$-matrix
$S=(s_{ij})$ can be used to define projectors
$\omega_i=\frac{1}{D}\sum_{j\in I} s_{ij} p_j$, which projects out
the $i$th label.

Given a ribbon link $l$, $<\omega_0*l>$ denotes the Kauffman
bracket of the colored ribbon link $l$ that each component is
colored by $\omega_0$.

\begin{theorem}\label{invariantformula}

\begin{enumerate}

\item The tangle category $\Delta_A$ is a premodular category, and
is modular if and only if $A$ is a primitive $4r$th root of unity.

\item Given a premodular category $\Lambda$, and $X$ an oriented closed
$3$-manifold with an $m$-component surgery link $l$, then
$Z_{JK}(X)=\frac{1}{D^{m+1}}(\frac{p_{-}}{D})^{\sigma(l)}
<\omega_0 *l>$ is a $3$-manifold invariant, where $\sigma(l)$ is
the signature of the framing matrix of $l$.

\end{enumerate}

\end{theorem}

\subsection{Diagram TQFTs for closed
manifolds}\label{closeddiagram}

In this section, fix an integer $r\geq 3$, $A$ as in Lemma
\ref{jwcases}.  For these special values, the picture spaces
$\Pic^A(Y)$ form a modular functor which is part of a TQFT. These
TQFTs will be called diagram TQFTs. In the following, we verify
all the applicable axioms for diagram TQFTs for closed manifolds
after defining the partition function $Z$.

The full axioms of TQFTs are given in Section
\ref{axiomsoftqft}. The applicable axioms for closed manifolds
are:

\begin{enumerate}

\item Empty surface axiom:  $V(\emptyset)=\C$

\item Sphere axiom: $V(S^2)=\C$.  This is a consequence of the
disk axiom and gluing formula.

\item Disjoint union axiom for both $V$ and $Z$:

\item Duality axiom for $V$:

\item Composition axiom for $Z$:  This is a consequence of the
gluing axiom.

\end{enumerate}

These axioms together form exactly a tensor functor as follows:
the category $\X^{2,cld}$ of oriented closed surfaces $Y$ as
objects and oriented bordisms up to diffeomorphisms between
surfaces as morphisms is a strict rigid tensor category if we
define disjoint union as the tensor product; $-Y$ as the dual
object of $Y$; for birth/death, given an oriented closed surface
$Y$, let $Y\times S^1_{-}: \emptyset \la -Y\coprod Y$ be the birth
operator, and $Y\times S^1_{+}: -Y\coprod Y \la \emptyset$ the
death operator, here $S^1_{\mp}$ are the lower/up semi-circles.

\begin{definition}
A $(2+1)$-anomaly free TQFT for closed manifolds is a
\emph{nontrivial} tensor functor $V: \X^{2,cld} \la \V$, where $\V$ is the tensor category of finite dimensional
vector spaces.
\end{definition}

Non-triviality implies $V(\emptyset)=\C$ by the disjoint union
axiom. Since $\emptyset \coprod \emptyset=\emptyset$,
$V(\emptyset)=V(\emptyset)\otimes V(\emptyset)$.  Hence
$V(\emptyset)\cong \C$ because otherwise $V(\emptyset)=0$ the
theory is trivial. The empty set picture $\emptyset$ is the
canonical basis, therefore, $V(\emptyset)=\C$.

The disjoint union axiom and the trace formula for $Z$ in Prop.
\ref{tqftaxiomcorollary} fixes the normalization of $3$-manifold
invariants.  Given an invariant of closed $3$-manifolds, then
multiplication of all invariants by scalars leads to another
invariant. Hence $Z$ on closed $3$-manifolds can be changed by
multiplying any scalar $k$.
 But this freedom is eliminated from TQFTs by the disjoint union
axiom which implies $k=k^2$, hence $k=1$ since otherwise the
theory is trivial. The trace formula implies $Z(S^2\times S^1)=1$.
We set $Z(S^3)=\frac{1}{D}$, and $D$ is the total quantum order of the theory.

Recall that the picture space $\Pic^A(Y)$ is defined even for
unorientable surfaces.  When $Y$ is oriented, $\Pic(Y)$ is
isomorphic to $K_A(Y\times I)$.  Given a bordism $X$ from $Y_1$ to
$Y_2$, we need to define $Z_{D}(X)\in \Pic^A(- Y_1\coprod Y_2)$.
It follows from the disjoint union axiom and the duality axiom,
$Z_{D}(X)$ can be regarded as a linear map $\Pic^A(Y_1)\la
\Pic^A(Y_2)$.

Given a closed surface $Y$, let $V_D(Y)=\Pic^A(Y)$.  For a
diffeomorphism $f: Y\rightarrow Y$, the action of $f$ on pictures
is given by moving them in $Y$.  This action on pictures descends
to an action of mapping classes on $V(Y)$. To define $Z_D(X)$ for
a bordism $X$ from $Y_1$ to $Y_2$, fix a relative handle
decomposition of $X$ from $Y_1$ to $Y_2$.

Suppose that $Y_2$ is obtained from $Y_1$ by attaching a single
handle of indices $=0,1,2,3$.  For indices $=0,3$, the linear map
is just a multiplication by $\frac{1}{D_{JK}}$, where
$D_{JK}=\sqrt{\frac{-2r}{(A^2-A^{-2})^2}}$.  Since $S^3$ is a
$0$-handled followed by a $3$-handle,
$Z_{D}(S^3)=\frac{1}{D^2_{JK}}$.  This is not a coincidence, but a
special case of a theorem of K.~ Walker and V.~ Turaev that
$Z_{D}(X)=|Z_{JK}(X)|^2$ for any oriented closed $3$-manifold.

Given a multicurve $\g$ in $Y_1$,

1): If a $1$-handle $I\times B^2$ is attached to $Y_1$, isotopy
$\g$ so that it is disjoint from the attaching regions $\p I\times
B^2$ of the $1$-handle. Label the co-core circle
$\frac{1}{2}\times \p B^2$ of the $1$-handle by $\omega_0$ to get
a formal multicurve in $Y_2$.  This defines a map from
$\Pic^A(Y_1)$ to $\Pic^A(Y_1)$ by linearly extending to pictures
classes.

2): If a $2$-handle $B^2\times I$ is attached to $Y_1$, isotopy
$\g$ so that it intersects the attaching circle $\p B^2 \times
\frac{1}{2}$ of the $2$-handle transversely. Expand this attaching
circle slightly to become a circle $s$ just outside the $2$ handle
and parallel to the attaching circle $\p B^2 \times \frac{1}{2}$.
Label $s$ by $\omega_0$.  Fuse all strands of $\g$ so that a
single labeled curve intersects the attaching circle $\p B^2
\times \frac{1}{2}$; only the $0$-labeled curves survive the
projector $\omega_0$ on $s$. By drawing all remaining curves on
the $Y_1$ outside the attaching region plus the two disks
$B^2\times \{0\}$ and $B^2\times \{1\}$, we get a formal diagram
 in $Y_2$.

We need to prove that this definition is independent of
handle-slides and cancellation pairs, which is left to the
interested readers.

Now we are ready to verify all the axioms one by one:

The empty surface axiom:  this is true as we have a non-trivial
theory.

The sphere axiom: by the $\lq\lq$d-isotopy" constraint, every
multicurve with $m$ loops $=d^m \emptyset$.  If $\emptyset$
picture on $S^2$ is $=0$ in $\Pic^A(S^2)$, then $Z_{D}(B^3)=0$
which leads to $Z_{D}(S^3)=0$.  But $Z_{D}(S^3) \neq 0$, it
follows that $\Pic^A(S^2)=\C$.

The disjoint union axioms for both $V$ and $Z$ are obvious since
both are defined by pictures in each connected component.

$\Pic(-Y)=\Pic(Y)$ is the identification for the duality axiom. To
define a functorial identification of $\Pic^A(-Y)$ with
$\Pic^A(Y)^*$, we define a Hermitian paring: $\Pic^A(Y)\times
\Pic^A(Y)\la \C$.  Since $\Pic^A(Y)$ is an algebra, and
semi-simple, it is a matrix algebra.  For any $x,y\in \Pic^A(Y)$,
we identify them as matrices, and define
$<x,y>=\tn{Tr}(x^{\dagger}y)$.  This is a non-degenerate inner
product.  The conjugate linear map $x\rightarrow <x,\cdot>$ is the
identification of $\Pic^A(-Y)$ with $\Pic^A(Y)^*$.

Summarizing, we have;

\begin{theorem}

The pair $(V_D,Z_{D})$ is a $(2+1)$-anomaly free TQFT for closed
manifolds.

\end{theorem}

\subsection{Boundary conditions for picture
TQFTs}\label{boundarycondition}

In Section \ref{localrelation},  we consider $\C[\scc]$ for
surfaces $Y$ even with boundaries.  Given a surface $Y$ with $m$
boundary circles with $n_i$ fixed points on the $i$th boundary
circle, by imposing Jones-Wenzl projector $p_{r-1}$ away from the
boundaries, we obtain some pictures spaces, denoted as
$\Pic^A(Y;n_1,\cdots, n_m)$.  To understand the deeper properties
of the picture space $\Pic^A(Y)$, we need to consider the
splitting and gluing of surfaces along circles.  Given a simple
closed curve (scc) $s$ in the interior of $Y$, and a multicurve
$\g$ in $Y$, isotope $s$ and $\g$ to general position.  If $Y$ is
cut along $s$, the resulted surface $Y_{\tn{cut}}$ has two more
boundary circles with $n$ points on each new boundary circle,
where $n$ is the number of intersection points of $s\cap\g$, and
$n\in \{0,1,2,\cdots,\}$. In the gluing formula, we like to have
an identification of $\oplus \Pic^A(Y_{\tn{cut}})$ with all
possible boundary conditions with $\Pic^A(Y)$, but this sum consists of
infinitely many
non-trivial vector space,which contradicts that $\Pic^A(Y)$ is finitely
dimensional. Therefore, we need more refined boundary conditions.
One problem about the crude boundary conditions of finitely many
points is due to bigons resulted from the $\lq\lq$d-isotopy" freedom:
we may introduce a trivial scc intersecting $s$ at many points, or
isotope $\g$ to have more intersection points with $s$.  The most
satisfactory solution is to define a picture category, then the
picture spaces become modules over these categories.  Picture
category serves as crude boundary conditions.  To refine the crude
boundary conditions, we consider the representation category of
the picture category as new boundary conditions.  The
representation category of a picture category is naturally Morita
equivalent to the original picture category.  The gluing formula
can be then formulated as the Morita reduction of picture modules
over the representation category of the picture category. The
labels for the gluing formula are given by the irreps of the
picture categories. This approach will be treated in the next two
sections.  In this section, we content ourselves with the
description of the labels for the diagram TQFTs, and define the
diagram modular functor for all surfaces. In Section
\ref{axiomsoftqft}, we will give the definition of a
TQFT, and later verify all axioms for diagram TQFTs.

The irreps of the non-semi-simple TL annular categories at roots
of unity were contained in \cite{grahamlehrer}, but we need the
irreps of the semi-simple quotients of TL annular categories,
i.e., the TLJ annular categories.

The irreps of the TLJ annular will be analyzed in Sections
\ref{irrepsoftlj4r} \ref{irrepsoftlj2r}.  In the following, we
just state the result.  By Theorem \ref{structuretheorem} in
Appendix \ref{repofcategory}, each irrep can be represented by an
idempotent in a morphism space of some object.  Fix $h$ ($ 0\leq
h\leq k$) many points on $S^1$, and let $\omega_{i,j;h}$ be the
following diagram in the annulus $\Aa$: the two circles in the annulus are
labeled by $\omega_i, \omega_j$, and $h=3$ in the diagram.

\begin{figure}[tbh]\label{annularprojector}
\centering
\includegraphics[width=1.25in]{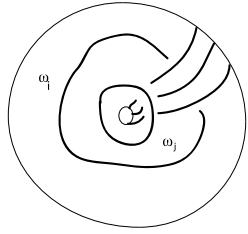}
\caption{Annular projector}\label{annularprojector}
\end{figure}

The labels for diagram TQFTs are the idempotents $\omega_{i,j;h}$
above. Given a surface $Y$ with boundary circles $\g_i,i=1,..,m$.
In the annular neighborhood $\Aa_i$ of $\g_i$, fix an idempotent
$\omega_{i,j;h}$ inside $\Aa_i$.  Let $\Pic_{D}^A(Y;
\omega_{i,j;h})$ be the span of all multicurves that within
$\Aa_i$ agree with $\omega_{i,j;h}$ modulo $p_{r-1}$ .

\begin{theorem}

If $A$ is as in Lemma \ref{jwcases}, then
the pair $(\Pic^A(Y), Z_{D})$ is an anomaly-free TQFT.

\end{theorem}

\subsection{Jones-Kauffman skein spaces}\label{jkskeinspace}

In this section, fix an integer $r\geq 3$, $A$ as in Lemma
\ref{jwcases}, and $d=-A^2-A^{-2}$.

\begin{definition}
Given any closed surface $Y$, let $\Pic^A(Y)$ be the picture space
of pictures modulo $p_{r-1}$.  Given an oriented $3$-manifold $X$,
the skein space of $p_{r-1}$ and the Kauffman bracket is called
the Jones-Kauffman skein space, denoted by $K_A(X)$.
\end{definition}

The following theorem collects the most important properties of the
Jones-Kauffman skein spaces.
The proof of the theorem relies heavily on handlebody
theory of manifolds.

\begin{theorem}\label{jkskeinspace}

a): Let $A$ be a primitive $4r$th root of unity.  Then

\begin{enumerate}

\item $K_A(S^3)=\C$.

\item  There is a canonical isomorphism of $K_A(X_1\amalg
X_2)\cong K_A(X_1)\# K(X_2)$.

\item If $\p X_1=\p X_2$, then $K_A(X_1)\cong K_A(X_2)$, but not
canonically.

\item If the $\emptyset$ link is not $0$ in $K_A(X)$ for a closed
manifold $X$, then $K_A(X)\cong \C$ canonically.  The $\emptyset$
link in $K_A(\#_{r=1}^m S^1\times S^2)$ and $K_A(Y\times S^1)$ is
not $0$ for oriented closed surface $Y$.

\item $K_A(-{X})\times K_A(X) \la K_A(DX)$ is non-degenerate.
Therefore, $K_A(\overline{X})$ is isomorphic to $K_A(X)^*$, but
not canonically.

\item $K_A(Y\times I) \la \textrm{End}(K_A(X))$ is an isomorphism
if $\p X=Y$.

\item $\Pic^A(Y)$ is canonically isomorphic to $K_A(Y\times I)$ if
$Y$ is orientable, hence also isomorphic to
$\textrm{End}(K_A(X))$.

\end{enumerate}

\noindent b): If $A$ is a primitive $2r$th root of unity or $r$th
root of unity, then $(2)$ does not hold, and it follows that the
rest fail for disconnected manifolds.

\end{theorem}

\begin{proof}

(1) Obvious.

(2) The idea here in physical terms is that non-trivial particles cannot cross an $S^2$.

The skein space $K_A(X_1\amalg X_2)$ is a subspace of
$K_A(X_1)\# K(X_2)$ by inclusion.  So it suffices to show this is
onto.  Given any skein class $x$ in $K_A(X_1)\# K(X_2)$, by
isotopy we may assume $x$ intersects the connecting $S^2$
transversely.  Put the projector $\omega_0$ on $S^2$ disjoint from
$x$, then $\omega_0$ encircle $x$ from outside.  Apply $\omega_0$
to $x$ to project out the $0$-label, we split $x$ into two skein
classes in $K_A(X_1\amalg X_2)$, therefore the inclusion is onto.

(3) This is an important fact.  For example, combining with (1), we see that the
Jones-Kauffman skein space of any oriented $3$-manifold is $\cong \C$.

We will show below that any bordism $W^4$ from $X_1$ to $X_2$ induces an isomorphism.  Moreover, the
isomorphism depends only on the signature of the $4$-manifold $W^4$.

Pick a $4$-manifold $W$ such that $\p W=-X_1\cup_Y (Y\times
I)\cup_Y X_2$ ($W$ exists since every orientable $3$ manifold
bounds a $4$ manifold), and fix a handle-decomposition of $W$.  $0$-handles, and dually $4$-handles,
induce a scalar multiplication.  $1$-handles, or dually $3$-handles, also induce a scalar
multiplication by (2).  By (2), we may assume that $X_i,i=1,2$ are connected.
Therefore, we will fix a relative handle decompositions of $W$ with
only $2$-handles, and let $L_{X_i},i=1,2$ be the attaching links
for the $2$-handles in $X_i$, respectively.  Then $X_1\backslash
L_{X_1}\cong X_2\backslash L_{X_2}$.  The links $L_{X_i},i=1,2$ are dual to each other in a sense:
let $L_{X_i}^{\tn{dual}}$ be the link consists
of cocores of the $2$-handles on $X_i$, then
surgery on $L_{X_1}$ in $X_1$ results $X_2$, while surgery on ${L_{X_1}}^{\tn{dual}}$
in $X_1$ results $X_1$, and vice versa.

Define a map $h_1: K_A(X_1)\rightarrow K_A(X_2)$ as follows: for
any skein class representative $\g_1$, isotope $\g_1$ so that it
misses $L_{X_1}$.  Note that in the skein spaces, labeling a component $L_1$ of a link $L$ by
$\omega_0$, denoted as $\omega_0*L_1$, is equivalent to surgering the component;
then $h(\g_1)=\g_1\coprod
\omega_0*L_{X_2}$, where $\g_1$ is now considered as a link in $X_2$.  Formally, we
write this map as:
$$ (X_1;\g_1)\rightarrow (X_1; L_{X_1}\coprod {L_{X_1}}^{\tn{dual}} \coprod \g_1) \rightarrow (X_2;L_{X_2}\coprod \g_2),$$
where $\g_2$ is $\g_1$ regarded as a skein class in $X_2$.  In this map, $L_{X_1}$ is mapped
to the empty skein as it has been surged out, while ${L_{X_1}}^{\tn{dual}}$ is mapped to $L_{X_2}$.
Then define $h_2$ similarly.  The composition of $h_1$ and $h_2$
is the link invariant of the colored link
$L_{X_i}$ union a small linking circle for each component plus a parallel copy of ${L_{X_i}}^{\tn{dual}}$
union its small linking circles as in the Fig.~\ref{jkiso},
which is clearly a scalar, hence an isomorphism.

\begin{figure}[tbh]\label{jkiso}
\centering
\includegraphics[width=3.45in]{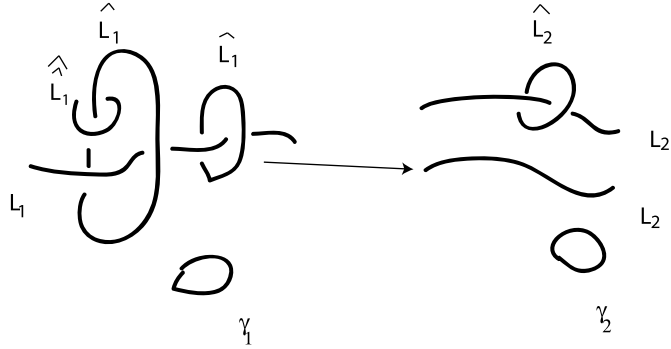}
\caption{Skein space maps}\label{ikiso}
\end{figure}

Now we see that a pair, ($W$, a handle decomposition), induces an isomorphism.
Using Cerf theory, we can show that the isomorphism is first independent of the handle
decomposition; secondly it is a bordism invariant:
if there is a $5$-manifold $N$ which is a relative bordism from $W$ to $W'$, then $W$ and $W'$ induces the same map.
Hence the isomorphism depends only on
the signature of the $4$-manifold $W$.  The detail is a highly non-trivial exercise in Cerf theory.

(4) follows from (1)--(3) easily.

(5) The inner product is given by doubling.  By (3) $K_A(X)$ is
isomorphic to $K_A(H)$, where $H$ is a handlebody with the same
boundary.  By (4), the same inner product is non-singular for
$K_A(H)$.  Chasing through the isomorphism in (3) shows that the
inner product on $K_A(X)$ is also non-singular.   Since
$K_A(DX)\cong \C$, hence $K_A(\overline{X})$ is isomorphic to
$K_A(X)^*$.

(6) $K_A(Y\times I)$ is isomorphic to $K_A(-X\coprod X)$ by (3).
It follows that the action of $K_A(Y\times I)$ on $K_A(X)$:
$K_A(X)\otimes K_A(Y\times I)\rightarrow K_A(X\cup_Y(Y\times I)$
becomes an action of $K_A(-X\coprod X)$ on $K_A(X)$:
$K_A(X)\otimes K_A(-X\coprod X)\rightarrow K_A(DX\coprod X)$.  By
the paring in (5), we identify the action as the action of
$\tn{End}(X)=K_A(-X)\otimes K_A(X)$ on $K_A(X)$.

(7) follows from (6) easily.

\end{proof}

The pairing $K_A(-{X})\times K_A(X) \la K_A(DX)$ allows us to
define a Hermitian product on $K_A(X)$ as follows:

\begin{definition}

Given an oriented closed $3$-manifold $X$, and choose a basis $e$
of $K_A(DX)$.  Then $K_A(DX)=\C e$.  For any multicurves $x,y$ in
$X$, consider $x$ as a multicurve in $-X$, denoted as $\bar{x}$.
Then define $\bar{x}\cup y=<x,y>e$, i.e. the ratio of the skien
$\bar{x}\cup y$ with $e$.  If $\emptyset$ is not $0$ in $K_A(DX)$,
then Hermitian pairing is canonical by choosing $e=\emptyset$.

\end{definition}

Almost all notations are set up to define the Jones-Kauffman
TQFTs. We see in Theorem \ref{jkskeinspace} that if two
$3$-manifolds $X_i,i=1,2$ have the same boundary, then $K_A(X_1)$
and $K_A(X_2)$ are isomorphic, but not canonically.  We like to
define the modular functor space $V(Y)$ to be a Jones-Kauffman
skein space. The dependence on $X_i$ is due to a framing-anomaly,
which also appears in Witten-Restikhin-Turaev $SU(2)$ TQFTs.  To
resolve this anomaly, we introduce an extension of surfaces.
Recall by Poincare duality, the kernel $\lambda_X$ of $H_1(\p
X;\R)\rightarrow H_1(X;\R)$ is a Lagrangian subspace $\lambda
\subset H_1(Y;\R)$. This Lagrangian subspace contains sufficient
information to resolve the framing dependence.  Therefore, we
define an extended surface as a pair $(Y;\lambda)$, where
$\lambda$ is a Lagrangian subspace of $H_1(Y;\R)$.  The
orientation, homology and many other topological property of an
extended surface $(Y;\lambda)$ mean that of the underlying surface
$Y$.

The labels for the Jones-Kauffman TQFTs are the Jones-Wenzl
projectors $\{p_i\}$. Given an extended surface $(Y;\lambda)$ with
boundary circles $\g_i,i=1,..,m$.  Glue $m$ disks $B^2$ to the
boundaries to get a closed surface $\widehat{Y}$ and choose a
handlebody $H$ such that $\p H=\widehat{Y}$, and the kernel
$\lambda_H$ of $H_1(Y;\R)\rightarrow H_1(H;\R)$ is $\lambda$. In a
small solid cylinder neighborhood $B_i^2\times [0,\e]$ of each
boundary circle $\g_i$, fix a Jones-Wenzl projector $p_{i_j}$
inside some arc$\times [0,\e]$, where the arc is any fixed
diagonal of $B_i^2$. Let $V_{JK}^A(Y; \lambda, \{p_{i_j}\})$ be
the Jones-Kauffman skein space of $H$ of all pictures within the
solid cylinders $B_i^2\times [0,\e]$ agree with $\{p_{i_j}\}$.

\begin{lemma}\label{fusionrule}

Let $B^2$ be a $2$-disk, $\Aa$ an annulus and $P$ a pair of pants,
and $(Y,\lambda)$ an extended surface with $m$ punctures labelled
by $p_{i_j},j=1,2,\cdots, m$, then

\begin{enumerate}

\item $V_{JK}^A(B^2; p_i)=0$ unless $i=0$, and
$V_{JK}^A(B^2;p_0)=\C$.

\item $V_{JK}^A(\Aa;p_i,p_j)=0$ unless $i=j$, and
$V_{JK}^A(\Aa;p_i,p_i)=\C$

\item $V_{JK}^A(P;p_i,p_j,p_k)=0$ unless $i,j,k$ is admissible,
and $V_{JK}^A(P;p_i,p_j,p_k)\cong \C$ if $i,j,k$ is admissible.

\item Given an extended surface $(Y;\lambda)$, and let $H$ be a
genus=$g$ handlebody such that $\p H=(\widehat{Y};\lambda)$ as
extended manifolds. Then admissible labelings of any framed
trivalent spine dual to a pants decomposition of $Y$ with all
external edges labelled by the corresponding boundary label
$p_{i_j}$ is a basis of $V_{JK}^A(Y;\lambda, \{p_{i_j})\}$.

\item $V^A_{JK}(Y)$ is generated by bordisms $\{X|\p X=Y\}$ if $Y$
is closed and oriented.

\end{enumerate}

\end{lemma}

Given an extended surface $(Y;\lambda)$, to define the partition
function $Z_{JK}(X)$ for any $X$ such that $\p X=Y$, let us first
assume that $\p X=(Y;\lambda)$ as an extended surface.  Find a
handlebody $H$ such that $\lambda_H=\lambda$, and a link $L$ in
$H$ such that surgery on $L$ yields $X$. Then we define $Z_X$ as
the skein in $V_D(H)$ given by the $L$ labeled by $\omega_0$ on
each component of $L$.  If $\lambda_X$ is not $\lambda$, then
choose a $4$-manifold $W$ such that $\p W=-X\coprod X$ and the
Lagrange space $\lambda$ and $\lambda_X$ extended through $W$. $W$
defines an isomorphism between $K_A(X)$ and itself.  The image of
the empty skein in $K_A(X)$ is $Z(X)$.  Given $f:
(Y_1;\lambda_1)\rightarrow (Y_2;\lambda_2)$, the mapping cylinder
$I_f$ defines an element in $V(Y_1\coprod Y_2)\cong V(Y_1)\otimes
V(Y_2)$ by the disjoint union axiom.  This defines a
representation of the mapping class group $\Mcg(Y)$, which might
be a projective representation.

\begin{theorem}

If $A$ is a primitive $4r$th root of unity for $r\geq 3$, then
the pair $(V_{JK},Z_{JK})$ is a TQFT.

\end{theorem}

This theorem will be proved in Section \ref{joneskauffman}.

There is a second way to define the projective representation of
$\Mcg(Y)$.  Given an oriented surface $Y$, the mapping class group
$\Mcg(Y)$ acts on $\Pic(Y)$ by moving pictures in $Y$. This action
preserves the algebra structure of $\Pic(Y)$ in Prop.
\ref{algebrastructure}. The algebra $\Pic(Y)\cong
\tn{End}(K_A(Y))$ is a simple matrix algebra, therefore any
automorphism $\rho$ is given by a conjugation of an invertible
matrix $M_{\rho}$, where $M_{\rho}$ is only defined up to a
non-zero scalar. It follows that for each $f\in \Mcg(Y)$, we have
an invertible matrix $V_f=M_f$, which forms a projective
representation of the mapping class group $\Mcg(Y)$.

\section{Morita equivalence and cut-paste
topology}\label{moritacutpaste}

Temperley-Lieb-Jones algebras can be generalized naturally to
categories by allowing different numbers of boundary points at the
top and bottom of the rectangle $\Rr$.  Another interesting
generalization is to replace the rectangle by an annulus $\Aa$.
Those categories provide crude boundary conditions for $V(Y)$ when
$Y$ has boundary, and serve as $\lq\lq$scalars" for a
$\lq\lq$higher" tensor product structure which provides the formal
framework to discuss relations among $V(Y)$'s under cut-paste of
surfaces. The vector spaces $V(Y)$ of a modular functor $V$ can be
formulated as bimodules over those picture categories. An
important axiom of a modular functor is the gluing formula which
encodes locality of a TQFT, and describes how a modular functor
$V(Y)$ behaves under splitting and gluing of surfaces along
boundaries.  The gluing formula is best understood as a Morita
reduction of the crude picture categories to their representation
categories, which provides refined boundary conditions for
surfaces with boundaries.
 Therefore, the Morita reduction of a picture category amounts to
 the computation of all its irreps.  The use of bimodules and
 their tensor products over linear
categories to realize gluing formulas appeared in [BHMV]. In this
section, we will set up the formalism. The irreducible
representations of our examples will be computed in the next
section.

We work with the complex numbers $\C$ as the ground ring.  Let
$\Lambda$ denote a linear category over $\C$ meaning that the morphisms
set of $\Lambda$ are vector spaces over $\C$ and composition is
bilinear.  We consider two kinds of examples: \lq\lq rectangular"
and \lq\lq annular" $\Lambda$'s.  (The adjectives refer to methods for
building examples rather than additional axioms.)  We think of
rectangles$(\mathcal R)$ as oriented vertically with a \lq\lq top"
and \lq\lq bottom" and annuli$(\mathcal A)$ has an \lq\lq inside"
and an \lq\lq outside".  Sometimes, we draw an annulus as a
rectangle, and interpret the rectangles as having their left and
right sides glued.  The objects in our examples are finite
collections of points, or perhaps points marked by signs, arrows,
colors, etc., on \lq\lq top" or \lq\lq bottom" in the
rectangular case, and on \lq\lq inside" or \lq\lq outside" in the
annular case. The morphisms are formal linear combinations of
\lq\lq pictures" in $\Rr$ or $\Aa$ satisfying some linear
relations.  The most important examples are the Jones-Wenzl
projectors.  Pictures will variously be unoriented submanifolds
(i.e. multicurves), 1-submanifolds with various decorations such
as orienting arrows, reversal points, transverse flags, etc.,
and trivalent graphs.  Even though the pictures are drawn in two
dimensions they may in some theories be allowed to indicate
over-crossings in a formal way.  A morphism is sometimes called an
\lq\lq element" as if $\Lambda$ had a single object and were an
algebra.

Our $\Rr$'s and $\Aa$'s are parameterized, i.e. not treated merely
up to diffeomorphism.  One crucial part of the parameterization is
that a base point arc $*\times I \subset S^1\times I=\Aa$ be
marked.  The $*$ marks the base point on $S^1$ and brings us to a
 technical point.  Are the objects of $\Lambda$
 the continuously many collections of finitely many points in $I$
 (or $S^1$) or are they to be simply one representative example
 for each non-negative integer $m$.  The second approach
 makes the category feel bit more like an algebra (which has only
 one object) and the linear representations have a simpler object
 grading.  One problem with this approach is that if an annulus
 $\Aa$ factored as a composition of two by drawing a degree=$1$ scc
 $\gamma \subset \Aa$ (and parameterizing both halves), even if
 $\gamma$ is transverse to an element $x$ in $\Aa$  $\gamma\cap x$
 may not be the representative set of its cardinality.  This
 problem can be overcome by picking a base point preserving
 re-parameterization of $\gamma$.  This amounts to \lq\lq
 skeletonizing" the larger category and replacing some \lq\lq
 strict" associations by \lq\lq weak" ones.  Apparently a theorem
 of S.~MacLane guarantees that no harm follows, so either viewpoint
 can be adopted \cite{maclane}.  We will work with the continuously
 many objects version.

Recall the following definition from Appendix \ref{repofcategory}:

\begin{definition} A representation of a linear category $\Lambda$ is a functor $\rho: \Lambda
\rightarrow \V$, where $\V$ is the category of finite dimensional
vector spaces.  The action is written on the right: $\rho(a)=V_a$
and given $m\in \textrm{Mor}(a,b), \rho(m): V_a\rightarrow V_b$.
We write on the left to denote a representation of $\Lambda^{op}$.
\end{definition}

Let us track the definitions with the simplest pair of examples,
temporarily denote $\Lambda^{\Rr}$ and $\Lambda^{\Aa}$ the $\Rr$ and
$\Aa$-categories with objects finite collections of points and
morphisms transversely embedded un-oriented 1-submanifolds with
the marked points as boundary data.  Let us say that 1) we may
vary multicurves by $\lq\lq$d-isotopy" for $d=1$, and 2) to place
ourselves in the simplest case let us enforce the skein relation:
$p_2=0$ for $d=1$. This means that we allow arbitrary recoupling
of curves.  This is the Kauffman bracket relation associated to
$A=e^{\frac{2\pi i}{6}}, d=-A^2-A^{-2}=1$. The admissible pictures
may be extended to over-crosssings by the local Kauffman bracket
rule in Figure \ref{kauffmanbracket} in Section
\ref{kauffmanbrackets}.

In these theories, which we call the rectangular and annular
Temperley-Lieb-Jones categories $\TLJ$ at level $=1,d=1$,
over-crossings are quite trivial, but at higher roots of unity
they are more interesting. In schematic Figures
\ref{rectangleaction}, \ref{annularaction} give examples of
$\Lambda^{\Rr}$ and $\Lambda^{\Aa}$ representations.

\begin{figure}
\centering
\includegraphics[scale=1]{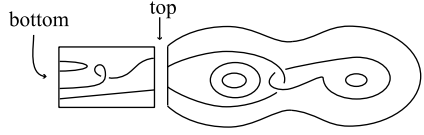}
\caption{$\Lambda^{\Rr}$ acts (represents) on vector space of pictures
in twice punctured disk (or genus=2
handlebody)}\label{rectangleaction}
\end{figure}

It is actually the morphism between objects index by 4 to 2
points, resp. ${_4}\Lambda^R_2$, which is acting in Figure
\ref{rectangleaction}, Figure \ref{annularaction}.

\begin{figure}
\centering
\includegraphics[scale=1]{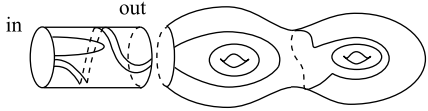}
\caption{$\Lambda^{\Aa}$ acts on vector space of pictures in punctured
genus=2 surface}\label{annularaction}
\end{figure}

The actions above are by regluing and then re-parameterizing to
absorb the collar. For each object $a$ in $\TLJ_{d=1}$, the
functor assigns the vector space of pictures
lying in a given fixed space with boundary data equal the object
$a$.  Given a fixed picture in $\Rr$ and $\Aa$, i.e. an element
$e$ of $\Lambda$, gluing and absorbing the collar defines a restriction
map: $f(e): V_4\rightarrow V_2$ (in the case illustrated) between
the vector spaces with the bottom (in) and the top (out) boundary
conditions.  To summarize the annular categories act on vector
spaces which are pictures on a surface by gluing on an annulus.
The rectangular categories, in practice, act on handlebodies or
other 3-manifolds with boundary by gluing on a solid cylinder;
Figure \ref{rectangleaction} is intentionally ambiguous and may be seen as a diagram or
$3-$manifolds. Because we can use framing and overcrossing
notations in the rectangle we are free to think of $\Rr$ either as
2-dimensional, $I\times I$, or $3-$dimensional $I\times B^2$.

\subsection{Bimodules over picture category}

Because a rectangle or annulus can be glued along two sides, we
need to consider $\Lambda^{\textrm{op}} \times \Lambda$ actions:
$\Lambda^{\textrm{op}} \times \Lambda\stackrel{\rho}{\rightarrow} \V$.  The
composition $\Lambda \underset{m \, \longmapsto
(m^{\tn{op},m})}{\stackrel{\tr}{\rightarrow}} \Lambda^{\textrm{op}}
\times \Lambda\stackrel{\rho}{\rightarrow} \V$ describes the action of
gluing an $\Rr$ or $\Aa$ on two sides (Figure
\ref{rectanglebiaction}, Figure \ref{annularbiaction}).

\begin{figure}
\centering
\includegraphics[scale=.8]{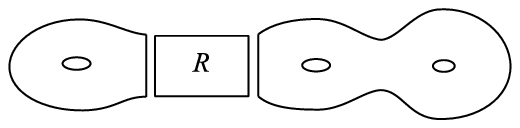}
\caption{$\Lambda^{\Rr}$ acts on both sides}\label{rectanglebiaction}
\end{figure}

\begin{figure}
\centering
\includegraphics[scale=.5]{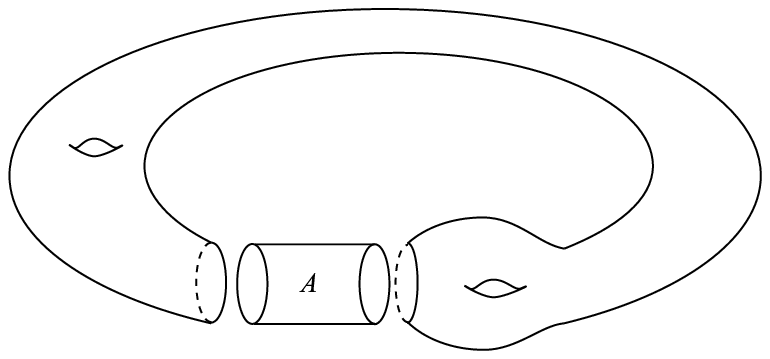}
\caption{$\Lambda^{\Aa}$ acts on both sides}\label{annularbiaction}
\end{figure}

We refer to the action of $\Lambda^{\textrm{op}}$ as \lq\lq left" and
the action of $\Lambda$ as \lq\lq right".

\begin{definition}

1) Let $M$ be a right $\Lambda$-representation (or \lq\lq module") and
$N$ a left $\Lambda$-module.  Denote by $M\bigotimes_{\Lambda}N$ the
$\C$-module quotient of $\underset{a}{\bigoplus}
M_a\bigotimes_{\C} {}_aN$ by all relations of the form
$u\alpha\bigotimes v=u\bigotimes \alpha v$, where \lq\lq a" and
\lq\lq b" are general objects of $\Lambda$, $u\in M_a$, $v\in {}_bN$
and $\alpha \in {}_a\Lambda_b=\textrm{Mor}(a,b).$

2)  A $\Cc \times \mathcal{D}$ bimodule is a functor $\rho:
\Cc^{\tn{op}}\times \mathcal{D} \la \V$.  Note that $\Cc$ is
naturally a $\Cc\times \Cc$ bimodule, which will be called the
regular representation of $\Cc$.

\end{definition}

Suppose now that $\Lambda$ is semi-simple.  This means that there is a
set $I$ of isomorphism classes of (finite dimensional) irreducible
representations $\rho_i, i\in I$ of $\Lambda$ and every (finite
dimensional) representation of $\rho$ may be decomposed $\rho\cong
\bigoplus V_i\bigotimes \rho_i$, where $V_i$ is a $\C$-vector
space with no $\Lambda$ action; $\textrm{dim}(V_i)$ is the multiplicity
of $\rho_i$. (If $\rho_1(a)=M_a^1$ and $\rho_2(a)=M_a^2$, then
$\rho_1\bigoplus \rho_2=M_a^1\bigoplus M_a^2$, and similarly for
morphisms.)

The following example is contained in the section in the general
discussion, but it is instructive to see how things work in
$\TLJ_{d=1}^{\Rr}$ and $\TLJ_{d=1}^{\Aa}$, the TLJ rectangular and
picture categories for $d=1$. These simple examples include the
celebrated toric codes TQFT in \cite{kitaev1} or $\Z_2$ gauge
theory, and illustrate the general techniques. Since it is almost
no extra work, we will include the corresponding calculation for
$\TLJ_{d=-1}^{\Rr}$ and $\TLJ_{d=-1}^{\Aa}$ where $A=e^{\frac{2\pi
i}{12}}, d=-1$ and $p_2=0$ for $d=-1$.
%${\includegraphics[width=.27cm,height=.25cm]{drawingd.eps}} = -{\cup\atop\cap}$.

A general element $x\in {}_a\Lambda_{\pm 1 b}^{\Rr}$ is determined by
its coefficients of \lq\lq squeezed" diagrams where only $0$ and
$1$ arcs cross the midlevel of the rectangle such diagrams look
like:

\begin{figure}[tbh]\label{squeezes}
\centering
\includegraphics[width=2.45in]{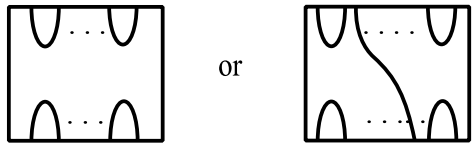}
\caption{Squeezed morphisms}\label{squeeze}
\end{figure}

%\vskip.2in \epsfxsize=2.5in \centerline{\epsfbox{fig53AB.eps}}{\centerline{Figure 5.3AB}} \vskip.2in

Similarly $x\in {}_a\Lambda_{\pm 1, b}^{A}$ are determined by the
coefficients of the diagrams in an annulus made by gluing the left
and right sides of FIGS??.  To each $a\in \Lambda^0, \Lambda=\Lambda_{\pm 1
}^{\Rr\; \T{or}\; \Aa}$, let $V_a$  be the vector space spanned by
diagrams, with $a$ end points on the top (outside) and zero ($a$
even) or one ($a$ odd) end point on the bottom (inside), thus
$V_a={}_{0\; \T{or}\; 1}\Lambda_a$.  The gluing map ${}_{0\; \T{or}\;
1}\Lambda^{\Rr}_a \bigotimes
{}_a\Lambda_b^{\Rr}\stackrel{\rho^{\Rr}}{\rightarrow} {}_{0 \;\T{or}\;
1}\Lambda^{\Rr}_b$ provides the two representations $\rho_0^{\Rr}$ and
$\rho_1^{\Rr}$ of $\Lambda_{\pm 1}^{\Rr}$ ($\rho_{0\;(\tn{or} \; 1)}$
sends $a$ odd (or even) to the 0-dimensional vector space.)

\begin{lemma}
The representation $\rho_0^{\Rr}$ and $\rho_1^{\Rr}$ are
irreducible.
\end{lemma}

\begin{proof} Proof:  Consider $\rho_0^{\Rr}$, the morphism vector space
${}_{2k}\Lambda_{2k}$ has dimension=$1$ (spanned by the empty diagram
in a rectangle) so that in \lq\lq grade", $2k$, $\rho_0^{\Rr}$ is
automatically irreducible. There is a morphism $m\in
{}_{2k}\Lambda_{2n}$, $m^{\dag}\in {}_{2n}\Lambda_{2k}$:

\begin{figure}[tbh]\label{factormorphisms}
\centering
\includegraphics[width=3.45in]{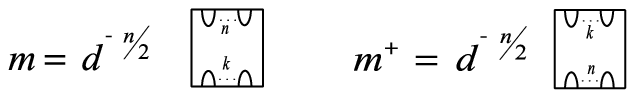}
\caption{Factored morphisms}\label{factormorphism}
\end{figure}

%\vskip.2in \epsfxsize=2.5in \centerline{\epsfbox{fig54.eps}}{\centerline{Figure 5.4B}} \vskip.2in

and $m^{\dag}m=\T{id}\in {}_{2k}\Lambda_{2k}$.  Thus any representation
$\{V\}$ on the even grades of the categories must have equal
dimension in all (even) grades since $\rho_0^{\Rr}(m)$ and
$\rho_0^{\Rr}(m^{\dag})$ are inverse to each other.  It follows
that any proper subrepresentation of $\rho_0^{\Rr}$ must be zero
dimensional in all grades.  Thus $\rho_0^{\Rr}$ is irreducible.

The argument for $\rho_1^{\Rr}$ is similar, simply add a vertical
line near the right margin of the rectangles in Fig.~\ref{factormorphism} to obtain
the corresponding $m, m^{\dag}$ in the odd grades. \end{proof}

\begin{lemma}  Any irreducible representation of $\Lambda_{\pm 1}^{\Rr}$ is
isomorphic to $\rho_0^{\Rr}$ or $\rho_1^{\Rr}$.
\end{lemma}

\begin{proof} The proof is based on \lq\lq resolutions of the
identity".
 In this case that means:

 \begin{figure}[tbh]\label{resolutionofids}
 \centering
\includegraphics[width=3.45in]{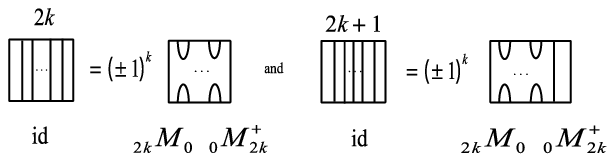}
\caption{Resolution of identity}\label{resolutionofid}
\end{figure}
%\vskip.2in \epsfxsize=3in \centerline{\epsfbox{fig55.eps}}{\centerline{Figure 5.5}} \vskip.2in

 Acting by $\rho$ on $\{V\}$ may be factored schematically as
 shown in Figure \ref{pictureaction}.

\begin{figure}[tbh]\label{pictureactions}
\centering
\includegraphics[width=3.45in]{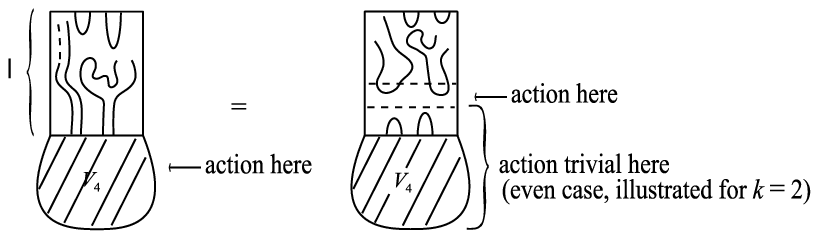}
\caption{Picture action}\label{pictureaction}
\end{figure}

%\vskip.2in \epsfxsize=3in \centerline{\epsfbox{fig56.eps}}{\centerline{Figure 5.6}} \vskip.2in

By Theorem \ref{structuretheorem}, for every $a\in \Lambda^0$, $V_a$ is
a subspace of ${}_a\Lambda_b$ for sone $b\in \Lambda^0$.
 In formulas, let $l\in {}_{2n}\Lambda_{2k}$ (for the even case)
 $\rho(l)=\rho\left(l\cdot {}_{2k}m^{\dag}_0 \cdot {}_0 m_{2k}\right)$, so the action factors
 through ${}_{0}\Lambda_{2k}$.  On the even (odd) grades the action is
 isomorphic to $\rho_0^{\Rr} \;\left(\rho_1^{\Rr}\right)$ tensor the subspace of ${}_{2n}\Lambda_0$ generated by
 elements of the form $l\cdot {}_{2k}m^{\dag}_0$ with trivial
 action.  So the general representation is isomorphic to a direct
 sum of irreducibles.  In this simple case it was not necessary
 (as it will be in other cases) to construct the Hermitian
 structure on $\Lambda$ to derive semi-simplicity.

 \end{proof}

 Now consider representations of $\Lambda_{\pm 1}^{\Aa}$.  Again $x\in
 {}_a\Lambda_{\pm 1, b}^{\Rr}$ is determined by diagrams with a
 \lq\lq weight" of $0$ or $1$.

%\vskip.2in \epsfxsize=2in \centerline{\epsfbox{fig57.eps}}{\centerline{Figure 5.7}} \vskip.2in

 In the special (\lq\lq principle graph") cases: ${}_0\Lambda_0$ and
 ${}_1\Lambda_1$ there are four diagrams (Figure \ref{idempotentd=1}) up to isotopy in the presence of
 relations $p_2=0$ for $d=\pm 1$.

\begin{figure}[tbh]\label{idempotentd=1s}
\centering
\includegraphics[width=3.45in]{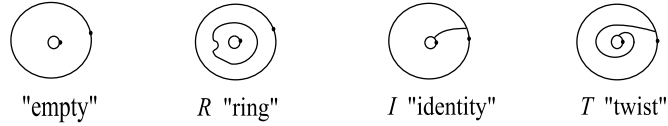}
\caption{Idempotents for annular $d=\pm 1$}\label{idempotentd=1}
\end{figure}
%${\includegraphics[width=.27cm,height=.25cm]{drawingd.eps}}\pm{\cup\atop\cap} =0$, and trivial loop $= d= \pm 1$.

 %\vskip.2in \epsfxsize=3in \centerline{\epsfbox{fig58.eps }}{\centerline{Figure 5.8}}\vskip.2in

 The reader should observe that if pictures are glued to be
 outside of $\emptyset$, ring $R$, straight arc $I$, or twist $T$ they may be transformed to
 another picture:
\begin{equation*}
    \emptyset \bigotimes R=R, \; R\bigotimes R=\pm \emptyset, \; I\bigotimes T=T, \;
    \T{and} \; T\bigotimes T=\pm T.
\end{equation*}
 (The signs are for $d=\pm 1$).  Let us call the object (i.e.
 number of end points) a \lq\lq crude label".  We have two crude
 labels \lq\lq 0" and \lq\lq 1" in this example.  For each crude
 label the symmetric $\left(\frac{\emptyset +R}{2}, \T{and} \;
 \frac{I+T}{2}\right)$ and anti-symmetric $(\frac{\emptyset -R}{2}, \T{and}
 \;\frac{I-T}{2})$ averages are in fact $(+1,-1)$ eigenvectors
 under gluing on a ring $R$ in ${}_{0}\Lambda_{1, 0}$ and gluing on $T$ in
 ${}_1\Lambda_{1,1}$.  The combinations $(\emptyset -iR)$ and $(\emptyset
 +iR)$ are $\pm 1$-eigenvectors for the action of $R$ in
 ${}_0\Lambda_{-1,0}$ and $(T -iT)$ and $(T
 +iT)$ are $\pm 1$-eigenvectors for the action of $T$ in
 ${}_0\Lambda_{-1,0}$.  In all cases these vectors span a 1-dimensional
 representation of four algebras
 ${}_0\Lambda_{1,0}, {}_1\Lambda_{-1,0}, {}_1\Lambda_{1, 1}$, ${}{_1} \Lambda_{-1,1}$ in which they lie.
  That is, ${}_0\Lambda_0$ and ${}_1\Lambda_1$ have the structure of commutative
  rings under gluing ($\cdot$) and formal sum $(+)$.  They
  satisfy ${}_0\Lambda_{\pm 1,0}\cong \C[R]/(R^2=\pm\emptyset)$ and
  ${}_1\Lambda_{\pm 1,1}=\C (T)/(T^2=\pm I)$ with $\emptyset$ and $I$
  serving as respective identities.

  What is more important than the representations of these
  algebras is the representations of the entire category ${}_a\Lambda_{\pm
 1, b}$ in which they lie.  Similar to the rectangular case, these
 four representatives together form the \lq\lq principle graph"
 from which the rest of the \lq\lq Bratteli diagram" for full
 category representations follows formally.

 \begin{lemma} These 4 representations of $\Lambda_{\pm 1}^{\Aa}$ are a
 complete set of irreducibles.
 \end{lemma}

 The Bratteli diagram in Figure \ref{bratteli} explains how to extend the
 algebra representations to the linear category (\lq\lq
 algebroid") in the rectangle $\Rr$ case.

\begin{figure}[tbh]\label{bratteli}
\centering
\includegraphics[width=3.45in]{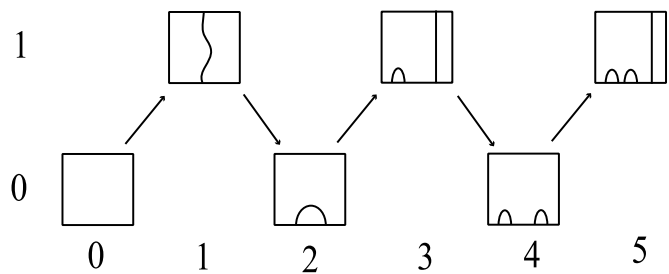}
\caption{Bratteli diagram}\label{bratteli}
\end{figure}

 %\vskip.2in \epsfxsize=2.7in \centerline{\epsfbox{fig59.eps }}{\centerline{Figure 5.9}} \vskip.1in

 All vector spaces above are 1-dimensional and spanned by the
 indicated picture in $\Rr$ and the $\nearrow$ is \lq\lq add line on
 right", the $\searrow$ \lq\lq bend right".  The annular case
  is similar and is shown in Figures \ref{annular510}, \ref{annular511}.

\begin{figure}[tbh]\label{annular510}
\centering
\includegraphics[width=3.45in]{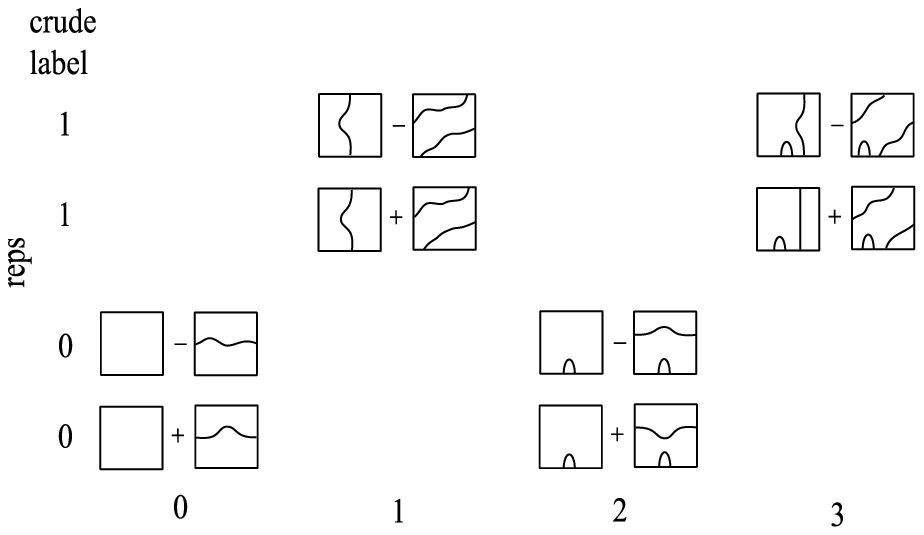}
\caption{$\Lambda^A_{+1}$}\label{annular510}
\end{figure}

\begin{figure}[tbh]\label{annular511}
\centering
\includegraphics[width=3.45in]{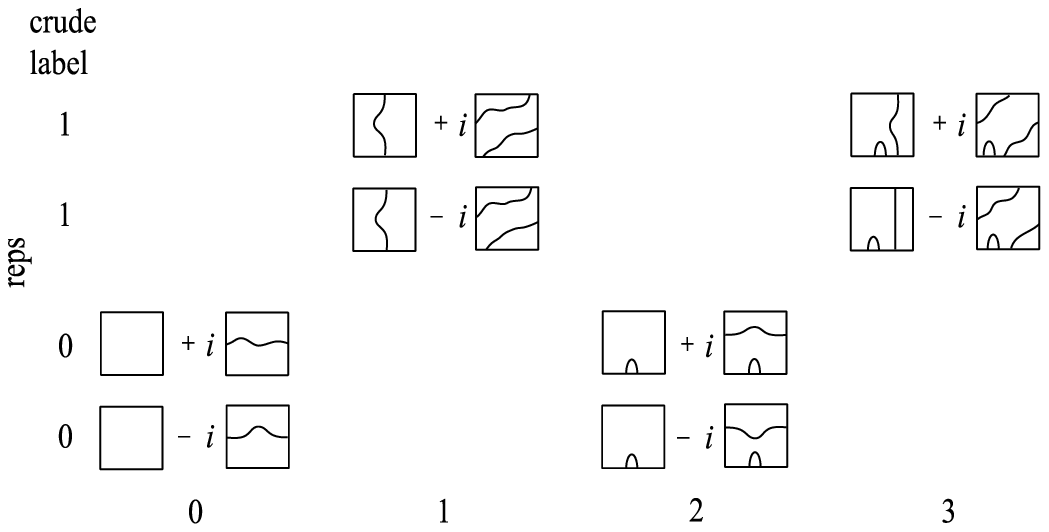}
\caption{$\Lambda^A_{-1}$}\label{annular511}
\end{figure}
 %\vskip.2in \epsfxsize=2.7in \centerline{\epsfbox{fig510.eps }}{\centerline{Figure 5.10}} \vskip.2in

%\vskip.2in \epsfxsize=2.7in \centerline{\epsfbox{fig511.eps}}{\centerline{Figure 5.11}} \vskip.2in

\noindent Note that we interpret the rectangles as having their
left and right sides glued.

 In the case of annular categories there is no tensor
 structure (horizontal stacking) so in general the arrows present
 in the $\Rr$-case seems more difficult to define in
 generality, but should be clear in these examples.  In the
 annular diagrams above all vector spaces of morphisms ${}_a\Lambda_b$
 have dimension=2 if $a=b\; mod \;2$ and zero otherwise.  As in the
 rectangular case, resolutions of the identity morphisms of $\Lambda_{\pm
 1}^{\Aa}$ into morphism which factor through $0$ or $1$-strand show
 that all representations are sums of the four.  One
 dimensionality and the existence of invertible morphisms between
 grades (exactly those shown in Fig.~\ref{factormorphism}, but now with the
 convention that the vertical sides of rectangles are glued to
 form an annulus) again show that the four are irreducible.

 By the corollary \ref{uniqueness} to Schur's lemma, the above
 decompositions into irreducibles are all unique.
There are direct generalizations of the categories so far
 considered to Temperley-Lieb-Jones categories in the next section.

\subsection{Cutting and paste as Morita equivalence}

Crude labels for picture categories are given as finitely many
points of the boundary.  In the gluing formulas for TQFTs, labels
are irreps of the picture categories.  The passage from the crude
labels of points to the refined labels of irreps is Morita
equivalence.

\begin{definition} Two linear categories $\Cc$ and $\mathcal{D}$ are
Morita equivalent if there are $\Cc \times \mathcal{D}$ bimodule
$M$ and $\mathcal{D}\times \Cc$ bimodule $N$ such that $M\otimes
N\cong \Cc$ and $N\otimes M\cong \mathcal{D}$ as bimodules.
\end{definition}

Let $\Lambda$ be a linear category, $\{a_i\}_{i\in I}$ a family of
objects of $\Lambda$.  For each $i\in I$, let $e_i$ be an idempotent in
the algebra ${}_{a_i}\Lambda_{a_i}$.  Define a new linear category
$\Delta$ as follows:  the objects of $\Delta$ is the index set
$I$, and the morphism set ${}_i\Delta_j=e_i{}_{a_i}\Lambda_{a_j}$.

Given an object $a$ in $\Lambda$, define the $\Delta\times \Lambda$ bimodule
$M$ as ${}_iM_a=e_i{}_{a_i}\Lambda_a$, and the  $\Lambda\times \Delta$
bimodule $N$ as ${}_aN_i={}_{a}\Lambda_{a_i}e_i$.

A key lemma is the following theorem in Appendix A of \cite{BHMV}:

\begin{theorem}

Suppose the idempotents $e_i$ generate $\Lambda$ as a two-sided ideal.
Then the bimodule $M\otimes_{\Lambda}N\cong \Delta$ and
$N\otimes_{\Delta}M\cong \Lambda$, i.e., $\Lambda$ and $\Delta$ are Morita
equivalence.

Consequently, tensoring (on the left or right), by the modules $N$
and $M$, gives rise to the Morita equivalence of $\Lambda$ and
$\Delta$. Moreover, these equivalences preserves tensor product of
bimodules.

\end{theorem}

Given two surfaces $Y_1, Y_2$ such that $\p Y_1 =\bar{\g_1}\prod
\g, \p Y_2 =\bar{\g}\prod \g_2$, and the picture spaces
$\Pic(Y_1), \Pic(Y_2)$ are bimodules over the picture category
$\Lambda$, then the picture space $\Pic(Y_1\cup_{\g}Y_2)$ is the tensor
product of $\Pic(Y_1)$ and $\Pic(Y_2)$ over $\Lambda$.  Morita
equivalence, applied to the picture category $\Lambda$, sends the
bimodule ${}_{-} {\Pic(Y_1\cup_{\g} Y_2)}_{-}$ to the bimodule
${}_{-} {\Pic(Y_1)\otimes \Pic(Y_2)}_{-}$ (over the representation
category $\Delta$ of the picture category $\Lambda$) because tensor
products are preserved under Morita equivalence.  Now the general
gluing formula can be stated as a consequence of Morita
equivalence:

\begin{theorem}

Let $Y_1, Y_2$ are two oriented surfaces such that $\p Y_1
=\bar{\g_1}\prod \g$ and $\p Y_2=\bar{\g}\prod \g_2$.  Then the
picture bimodule ${}_{-} {V(Y_1\cup_{\g} Y_2)}_{-}$ is isomorphic
to ${}_{-} {V(Y_1)\otimes_{\Delta} V(Y_2)}_{-}$ as bimodules.

\end{theorem}

As explained in Appendix \ref{repofcategory}, the idempotents
$e_i$ label a complete set of irreps of the linear category $\Lambda$.
Therefore, gluing formulas for picture TQFTs need the
representation categories of the picture categories.  In the
axioms of TQFTs the label set was a mysterious feature, now we
will see its origins in picture TQFTs.

Now let us write the Morita equivalence more explicitly.  Let $\Lambda$
be some $\Lambda^{\Rr}$ (or $\Lambda^{\Aa}$) and suppose $\Lambda$ is semi-simple
with index set $I$. The pictures in a fixed 3-manifold (surface)
with a \lq\lq left" and \lq\lq right" gluing region provide a
bimodule ${}_aB_b$ on for $\Lambda$.
 If the gluing region is not connected within the 3-manifold
 (surface) then $B\cong B^{\textrm{left}}\bigotimes_{\Lambda}
 B^{\textrm{right}}$.  We treat this case first.

\begin{lemma}
\label{clm1}
 \tn{$B^l\bigotimes_{\Lambda}B^r\cong_{i\in I}(V_i\bigotimes W_i)$,
 where ${}_aB^{\textrm{right}}\cong \underset{i}{\bigoplus}V_i\otimes \rho_i$ and
 ${}_aB^{\textrm{left}}\cong \bigoplus W_j \otimes \rho_j^{\textrm{op}}$,
 $\left(\rho_j^{\textrm{op}}(m)=(\rho_j(m))^{\dag}\right)$.}
\end{lemma}

\begin{proof}  Note that $\rho_i^{\T{op}}\bigotimes \rho_j\cong
\T{Hom}_{\Lambda}(\rho_i, \rho_j)\cong \left\{
\begin{smallmatrix}
\C & \T{if} & i=j \\
0 & \T{if} & i\neq j
\end{smallmatrix} \right.$.
As usual $\bigoplus_{\Lambda}$ distributes over $\bigotimes_{\Lambda}$, the
unusual feature is that the coefficients are vector spaces $V_i$
and $W_j$, not complex numbers.  They are \lq\lq multiplied" by
(ordinary) tensor product $\bigotimes$.
\end{proof}

The manipulations above are standard in the context of 2-vector
spaces \cite{freed}\cite{walker06}, and in fact a representation
is a 2-vector in the 2-vector space of all formal representations.

Now suppose the regions to be glued to the opposite ends of $\Rr$
$(\Aa)$ are part of a connected component of a 3-manifold
(surface), then write the bi-module ${}_aB_b\cong
\underset{(i,j)\in I^{\T{op}}\times I}{\bigoplus} W_{ij}\otimes
(\rho_i^{\T{op}}\bigotimes_{\Lambda}\rho_j)$ as a bimodule. Define a
$2$-trace,

\begin{equation*}
    \T{tr}B=\underset{a\in \T{obj}(\Lambda)}{\bigoplus}
    {}_aB_a/u\alpha\cong \alpha^{\T{op}} u,
\end{equation*}
where $\alpha \in {}_a\Lambda_b, u\in {}_bB_a$ are arbitrary.  Again,
\lq\lq linear algebra" yields:

\begin{lemma}\label{clmtr}
 $\T{tr}(B)\cong \underset{i\in I}{\bigoplus} W_{ii}.$
\end{lemma}

\begin{proof}  Schur's lemma implies
$\rho_i^{\T{opp}}\bigotimes_{\C}\rho_j\cong \C \;\;\T{iff}\;\;
i=j$. \end{proof}

Note that disjoint union of the spaces carries over to tensor
product, $\bigotimes_{\C}$, of the modules of pictures on the
space.  This makes lemma \ref{clm1} a special case of lemma
\ref{clmtr}. And further observe that both lemmas match the form
of the \lq\lq gluing formula" as expected, with $I={\mathcal L},$
the label set, and adjoint ($\dag$) is the involution $\widehat{}:
{\mathcal L}\rightarrow \mathcal L.$

\subsection{Annualization and quantum double}

Annular categories are closely related to the corresponding
rectangle categories.  In particular,  there is  an interesting
general principle:

\underline{\bf Conjecture:} {\it If $\Lambda^{\Rr}$ and $\Lambda^{\Aa}$ are
rectangular and annular versions of locally defined
picture/relation categories, then ${(\mathcal
D}(Rep(\Lambda^{\Rr}))\cong Rep (\Lambda^{\Aa})$, the Drinfeld center or
quantum double of the representation category of the rectangular
picture category is isomorphic to the representation category of
the corresponding annular category.}

The conjecture and its higher category generalizations are proved
in \cite{walker06}.

\section{Temperley-Lieb-Jones categories}\label{repoftlj}

To obtain the full strucure of the picture TQFTs, we need to
consider surfaces with boundaries, and boundary conditions for the
corresponding vector spaces $V(Y)$. The crude boundary conditions
using objects in TLJ categories are not suitable for the gluing
formulas. As shown in Section \ref{boundarycondition}, Section
\ref{moritacutpaste}, we need to find the irrpes of the TLJ
categories.  Two important properties of boundary condition
categories needed for TQFTs are semi-simplicity and the finiteness
of irreps.  For TLJ categories, both properties follow from a
resolution of the identity in the Jones-Wenzl projectors.

Let $X$ be a compact {\it parameterized} $n$-manifold.  The
interesting cases in this paper are the unit interval $I=[0,1]$ or
the unit circle $S^1$.  Define a category $\Cc(X)$ as follows: an
object $a$ of $\Cc(X)$ consists of finitely many points in the
interior of $X$, and given two objects $a,b$, a morphism in
$\M(a,b)$ is an $(n+1)$-manifold, not necessarily connected, in
the interior of $X\times [0,1]$ whose boundaries are $a\times 0,
b\times 1$, and intersects the boundary of $X\times [0,1]$
transversely. Given two morphisms $f\in {}_a\Cc_b, g\in
{}_b\Cc_c$, the composition of $f,g$ will be just the vertical
stacking from bottom to top followed by the rescaling of the
height to unit length $1$.  When $X$ is a circle, we will also
draw the vertical stacking of two cylinders as the gluing of two
annuli in the plane from inside to outside.  More often, we will
draw the stacking of cylinders as vertical stacking of rectangles
one on top of the other with periodic boundary conditions
horizontally.  Note the two boundary circles of a cylinder are
parameterized, so they have base points and are oriented.  The
gluing respects both the base-point and orientation.

\begin{figure}[tbh]\label{annularcomp}
\centering
\includegraphics[width=3.45in]{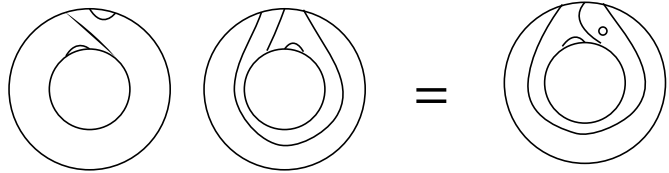}
\caption{Composition of annular morphisms}\label{annularcomp}
\end{figure}

Given a non-zero number $d\in \C$, the {\it Temperley-Lieb
category} $\textrm{TL}_d$ is the linear category obtained from
$\Cc([0,1])$ by first imposing $d$-isotopy in each morphism set,
and then taking formal finite sums of morphisms as follows: the
objects of $\TL_d$ are the same as that of $\Cc([0,1])$, and for
any two objects $a,b$, the vector space $\M_{\TL}(a,b)$ is spanned
by the set $\M(a,b)$ modulo $d$-isotopy.

The structure of the Temperley-Lieb categories $\TL_d$ depends
strongly on the values of $d$ as we have seen in the {\it
Temperley-Lieb algebras} $\TL_n(d)=\M(a,a)$ for any object $a\in
\TL^0_d$ consisting of $n$ points. When $A$ is as in Lemma
\ref{jwcases}, the semi-simple quotient of the Temperley-Lieb
category $\TL_d$ by the Jones-Wenzl idempotent $p_{r-1}$ is a
semi-simple category. The associated semi-simple algebras
$\TL_{n}(d)$ were first discovered by Jones in \cite{jones83rep}.
Therefore the semi-simple quotient categories of $\TL_{d}$ for a
particular $d$ will be called the {\it rectangular
Temperley-Lieb-Jones category} $\TLJ^{\Rr}_d$, where
$d=-A^2-A^{-2}$. Note that there will be several different $A$'s
which result in the same TLJ category as the coefficients of the
Jones-Wenzl idempotents are rational functions of $d$.  If we
replace the interval $[0,1]$ in the definition of the
Temperley-Lieb categories by the unit circle $S^1$, we get the
{\it annular Temperley-Lieb categories} $\TL^{\Aa}_d$, and their
semi-simple quotients {\it the annular Temperley-Lieb-Jones
categories} $\TLJ^{\Aa}_d$.

\subsection{Annular Markov trace}\label{annulusmarkov}

In the analysis of the structure of the TL algebras, the Markov
trace defined by Figure \ref{markovtrace} in Section
\ref{jonesrepandjoneswenzl} plays an important rule.  In order to
analyze the annular TLJ categories, we introduce an annular
version of the Markov trace and $2$-category generalizations.

Recall that $\Delta_{n}(x)$ is the Chebyshev polynomial. Let
$C_{n}(x)$ be the algebra $\C[x]/(\Delta_n(x))$.  Inductively, we
can check that the constant term of $\Delta_{n}$ is not 0 if $n$
is odd, and is $0$ if $n$ is even.  For $n$ even, the coefficient
of $x$ is $(-1)^{\frac{n}{2}} n$.  Let $n=2m$ and $q_{2m}(x)$ be
the element of $C_n(x)$ represented by $\frac{\Delta_{2m}(x)}{x}$.

Define the annular Markov trace $Tr^{\Aa}$ as follows:  $Tr^{\Aa}:
\TL_{n,d}\rightarrow C_{n}(x)$ is defined exactly the same as in
Figure \ref{markovtrace} in Section \ref{jonesrepandjoneswenzl}
except instead of counting the number of simple loops in the
plane, the image becoming elements in the annular algebra, where
$x$ is represented by the center circle(=called a ring sometimes).

\begin{prop}
$Tr^{\Aa}(p_n)=\Delta_{n}(x)$.
\end{prop}

It follows that the algebra $C_n(x)$ can be identified as the
annular algebra when $d$ is a  simple root of $\Delta_n(x)$.

If the inside and outside of the annulus $\Aa$ are identified, we
have a torus $T^2$.  The annular Markov trace followed by this
identification leads to a $2$-trace from $\TL_{n,d}$ to the vector
space of pictures in $T^2$.

\subsection{Representation of Temperley-Lieb-Jones categories}

Our goal is to find the representations of a TLJ category
$\TLJ^{\Rr}_d$ or $\TLJ^{\Aa}_d$. The objects consisting of the
same number of points in such categories are isomorphic, therefore
the set of natural numbers $\{0,1,2,\cdots\}$ can be identified
with a skeleton of the category (a complete set of representatives
of the isomorphism classes of objects). Each morphism set
$\M(i,j)$ is spanned by pictures in a rectangle or an annulus.

To find all the irreps of a TLJ category, we use Theorem
\ref{structuretheorem} in Appendix \ref{repofcategory} to
introduce a table notation as follows: we list a skeleton
$\{0,1,\cdots,\}$ in the bottom row. Each {\it isomorphism} class
$\rho_j$ of irreps of the category is represented by a row of
vector spaces $\{V_{j,i}\}=\{\rho_j(i)\}$. Each column of vector
spaces $\{V_{i,j} \}$ determines an isomorphism class of objects
of the category.  The graded morphism linear maps of any two
columns will be ${}_i\TLJ_{j}$, in particular the graded linear
maps of any column to itself give rise to the decomposition of
${}_i\TLJ_{i}$ into matrix algebras, i.e.,
${}_i\TLJ_{i}=\oplus_{j}\Hom(V_{j,i},V_{j,i})$. To find all irreps
of $\TLJ$, we look for minimal idempotents of ${}_i\TLJ_{i}$
starting from $i=0$. Suppose there exists an $m_0$ such that the
irreps $\{e_{j_i}\}$ of ${}_j\TLJ_{j}, j\leq m_0$ are sufficient
to decompose every ${}_m\TLJ_{m}$ as
$\oplus_{j_i}\textrm{Hom}(V_{j_i,m},V_{j_i,m})$ for all $m\geq
m_0$, then it follows that all irreps of $\TLJ$ are found;
otherwise, a non-zero new representation space $V_{k,a}$ from some
new irrep $\rho_k$ and $a\in \TLJ^0$ implying
$\textrm{Hom}(V_{k,a},V_{k,a})\subset {}_a\TLJ_a$ will contradict
the fact that ${}_a\TLJ_a\cong \oplus_{j\neq k}
\textrm{Hom}(V_{j,a},V_{j,a})$.

{\bf Remark:} For the annulus categories, we can identify one
irrep as the trivial label using the disk axiom of a TQFT. Given a
particular formal picture $x$ in an annulus, we define the disk
consequences of $x$ as all the formal pictures obtained by gluing
$x$ to a collar of the disk: given a picture $y$ on the disk,
composition $x$ and $y$ is a new picture in the disk. By
convention, pictures with mismatched boundary conditions are $0$.
Then the trivial label is the one whose disk consequences form the
vector space $\C$, while all others would result in 0.

For an object $m\in \TLJ^0$ if
$\textrm{id}_m=\oplus_{j<m}(\oplus_{l}f^l_{m,j}\cdot g^l_{j,m})$
for $f^l_{m,j}\in {}_m\TLJ_j, g^l_{j,m}\in {}_j\TLJ_m$, where $l$
is a finite number depending on $j$, then we have a {\it
resolution of the identity of $m$} into lower orders.

\begin{lemma}\label{Resolution of Identity}
If for some object $m$ of a TLJ category, we have a resolution of
its identity $\textrm{id}_m$ into lower orders, then every irrep
of the category $\TLJ$ is given by a minimal idempotent in
${}_j\TLJ_j$ for some $j<m$.
\end{lemma}

Given a TLJ category and two objects $a,c\in \TLJ^0$, there is a
subalgebra, denoted by $A_{cc}^a$, of the algebra
$A_{cc}={}_c\TLJ_c$ consisting all morphisms generated by those
factoring through the object $a$: $f\cdot g, f\in {}_c\TLJ_a, g\in
{}_a\TLJ_c$. If $e_a$ is an idempotent of ${}_a\TLJ_a$, then
$A_{cc}^{e_a}$ denotes the subalgebra of $A_{cc}^a$ consisting all
morphisms generated by those factoring through $e_a$, i.e., those
of the form $f\cdot e_a \cdot g$.

\begin{lemma}\label{tljequivalence}
Given two objects $a,b$ of a TLJ category, and two minimal
idempotents $e_a\in {}_a\TLJ_a, e_b\in {}_b\TLJ_b$, then

1): $A_{cc}^{e_a}$ is the simple matrix algebra over the vector
space ${}_c\TLJ_a e_a$.

2): If the two representations $e_a\TLJ,e_b\TLJ$ are isomorphic,
then for any $c\in \TLJ^0$, which is neither $a$ nor $b$, the
subalgebras $A_{cc}^{e_a}, A_{cc}^{e_b}$ of $A_{cc}$ are equal.
\end{lemma}

We will use these lemmas to analyze representations of TLJ
categories, but first we consider only the low levels.

\subsection{Rectangular Tempeley-Lieb-Jones categories for low levels}

Denote $A_{ij}={}_i\Lambda_j$.  Note that $A_{ii}$ is an algebra, and
$A_{ij}=0$ if $i\neq j$ mod 2.  The Markov trace induces an inner
product  $<,>: A_{ij}\times A_{ij}\rightarrow \C$ on all $A_{ij}$
given by $<x,y>=\tn{Tr}(\bar{x}y)$.

\subsubsection{Level=1, $d^2=1$}

 Using $p_2=0$,
we can \lq\lq squeeze" a general element $x\in A_{ij}$ so that
there are only $0$ or $1$ arcs cross the mid-level of the
rectangle. Such diagrams in Figure \ref{squeeze}) in Section
\ref{squeezes}.

The algebra $A_{00}=\C$, and the empty diagram is the generator.
The first irrep $\rho_0$ of $\TLJ^{\Rr}_{d=\pm 1}$ is given the
idempotent $p_0$, which is just the identity $\tn{id}_{\emptyset}$
on the empty diagram: if $j$ is odd, $\rho_0(j)=0$; if $j$ is
even, $\rho_0(j)=A_{0j}\cong \C$.

The algebra $A_{11}=\C$, generated by a single vertical line. The
identity does not factor through the $0$-object, so we have a new
idempotent $p_1$ (=idenity on the vertical line). The resulting
irrep $\rho_1$ sends even $j$ to $0$, and odd $j$ to $A_{1j}\cong
\C$.

Continuing to $A_{22}$, we see that the identity on two strands
does factor through $p_0$ given by the Jones-Wenzl idempotent
$p_2$.  By Lemma \ref{Resolution of Identity}, we have found all
the irreps of $\TLJ^{\Rr}_{\pm 1}$, which are summarized into
Table \ref{reclevel1}.

\begin{table}\caption{Irreps of rectangular level=1}\label{reclevel1}
\centering
\begin{tabular}{*{3}{|c}|}%{c|c|c|}
\hline
 $\rho_1$ & 0 & 1 \\
\hline
 $\rho_0$ & 1 & 0 \\
\hline
 &  0& 1  \\
%\hline
\end{tabular}%\caption{Irreps of rectangular level=1}\label{reclevel1}
\end{table}

$\TLJ^{\Rr}_{d=1}$ does not lead to a TQFT since the resulting
$S$-matrix $\left( \begin{array}{cc} \frac{1}{2} &\frac{1}{2} \\
\frac{1}{2} & \frac{1}{2} \end{array} \right)$ is singular.
Although $\TLJ^{\Rr}_{d=-1}$ does give rise to a TQFT, the
resulting theory  with $S$-matrix
$=\left( \begin{array}{cc} \frac{1}{2} &-\frac{1}{2} \\
-\frac{1}{2} & -\frac{1}{2} \end{array} \right)$ is not unitary.
The semion theory with $S$-matrix
$=\left( \begin{array}{cc} \frac{1}{2} &\frac{1}{2} \\
\frac{1}{2} & -\frac{1}{2} \end{array} \right)$ can be realized
only by the representation category of the quantum group $SU(2)$
at level=1. This subtlety comes from the Frobenius-Schur indicator
of the non-trivial label, which is $1$ for TLJ and -1 for quantum
group.

\subsubsection{Level=2, $d^2=2$}

Since $p_3$ is a resolution of the identity of $\textrm{id}_3$
into lower orders, it suffices to analyze $A_{ii}$ for $i\leq 2$.
The cases of $A_{00}, A_{11}$ are the same as level=1. Since
dim$A_{20}$=1, dim$A_{21}$=0 and dim$A_{22}$=2, id$_2$ does not
factor through lower orders, so there is a new idempotent in
$A_{22}$. The 1-dimensional subalgebra $A_{22}^0$ is generated by
$e_2$, which is the following diagram:
\[ \xy (10,0)*{}="A"; (20,0)*{}="B"; (10,12)*{}="C";
(20,12)*{}="D"; "A";"B" **\crv{(15,7)}; "C";"D" **\crv{(15,5)};
(6,7)*{\frac{1}{d}};
\endxy  \]  It is easy to check $e_2$ is the identity of $A_{22}^0$. Since
the identity of $A_{22}$ is the sum of the two central idempotents
(the two identities of each 1-dimensional subalgebra), the new
idempotent $p_2$ is id$_2-e_2$.  The irrep corresponding to $p_2$
sends each odd $j$ to $0$, and each even $j$ to $p_2A_{2j}$.

Therefore, the irreps of the level=2 $\TLJ^{\Rr}$ are given by
$p_0\TLJ$, $p_1\TLJ$, $p_2\TLJ$, which are summarized into Table \ref{reclevel2}.

\vspace{.1in}

\begin{table}
\centering
\begin{tabular}{c|c|c|c|}
\hline $\rho_2$ &0&0& 1\\
 \hline
 $\rho_1$ & 0 & 1 & 0\\
\hline
 $\rho_0$ & 1 & 0 & 1\\
\hline
 &  0& 1  & 2\\
%\hline
\end{tabular}
\caption{Irreps of rectangular level=2}\label{reclevel2}
\end{table}

\subsubsection{Level=3, $d^2=1+d$ or $d^2=1-d$}

The same analysis for objects $0,1,2$ yields three idempotents
$p_0,p_1,p_2$.  Direct computation shows $\Hom(3,3)\cong \C^5$,
$\Hom(3,0)=\Hom(3,2)=0$ and $\Hom(3,1)\cong \C^2$. By Lemma
\ref{tljequivalence}, $A_{33}^{p_1}=A_{33}^1$ is the
$4$-dimensional algebra of $2\times 2$ matrices over the vector
space $A_{13}$.  Let $v_1, v_2$ be the two vectors of $A_{31}$
represented by diagrams such that $<v_1,v_2>=<v_2,v_1>=d^2$, and
$<v_1,v_1>=<v_2,v_2>=d$.  Using Gram-Schmidt on the vectors
$v_1,v_2$, we get an orthonormal basis $e_1=\frac{v_1}{d},
e_2=\frac{v_2-e_1}{d^2-1}$ of $A_{31}$. Hence the identity of the
algebra $A_{33}^1$ is $|e_1><e_1|+|e_2><e_2|$. Therefore, the
remaining idempotent of $A_{33}$ is
$\tn{id}_3-|e_1><e_1|-|e_2><e_2|$, which is just $p_3$.  It
follows that the irreps of $\TLJ$ are given by $p_0\TLJ,
p_1\TLJ,p_2\TLJ, p_3\TLJ.$

\subsection{Annular Temperley-Lieb-Jones theories for low levels}

First we have the following notations for the pictures in the
annular morphism sets $A_{00}, A_{11}, A_{02}, A_{22}$, where
${\bf{1}_0},R,B, {\bf{1}_1},T_1, {\bf{1}_2}, T_2$ are annular diagrams:  $\bf{1}_0, \bf{1}_1, \bf{1}_2$
are identities with $0,1,2$ strands, $R$ is the ring, $B$ is the birth, and $T_1$ is the Dehn twisted
curve, and $T_2$ is the
fractional Dehn twisted curve.  We also use $B'$ to denote the diagram of $RB$ after $Z_2$ homology surgery.
A diagram with a $\bar{}$ is the one obtained from a reflection through a horizontal
line.

\subsubsection{Level=1, $d^2=1$}

The Jones-Wenzl idempotent $p_2$ is a resolution of $\tn{id}_2$
into the lower orders, so we need only to find the minimal
idempotents of $\Hom(0,0)$ and $\Hom(1,1)$.  Since any two
parallel lines can be replaced by a turn-back, the algebra
$A_{00}$ is generated by the empty picture $\emptyset$ and the
ring circle $R$.  Stacking two rings $R$ together and resolving
the two parallel lines give $R^2=1$, hence $A_{00}$ is the algebra
$\C[R]/(R^2-1)$. By Lemma \ref{decomposition}, the two minimal
idempotents of $A_{00}$ are $e_1=\frac{\emptyset+R}{2},
e_2=\frac{\emptyset-R}{2}$.  To test which idempotent is of the
trivial type, we apply $e_1,e_2$ to the empty diagram on the disk
and obtain $e_1 \emptyset=(\frac{d+1}{2})\emptyset, e_2\emptyset
=(\frac{1-d}{2})\emptyset$.  Hence if $d=1$, then $e_1$ is of the
trivial type, and if $d=-1$, then $e_2$ is of the trivial type.

The algebra $A_{11}$ is generated by the straight arc $I$ and the
tiwst $T$.  By stacking two rings $R$ together and resolving the
two parallel lines, we see that $A_{11}$ is the algebra
$\C[T]/(T^2-dI)$.  By Lemma \ref{decomposition}, for $d=1$, we
have two minimal idempotents
$e_{3,1}=\frac{I+T}{2},e_{4,1}=\frac{I-T}{2}$.  For $d=-1$, we
have two minimal idempotents
$e_{3,-1}=\frac{1-iT}{2},e_{4,-1}=\frac{1+iT}{2}.$  Note that
$\Hom(0,1)=\Hom(1,0)=0$.  Therefore, the annular TLJ categories
for $d=\pm 1$  have $4$ irreps $e_i,i=1,2,3,4$.

\subsubsection{Level=2, $d^2=2$}

For the TLJ categories at level=2, $d^2=2$, $p_3$ is a resolution
of the identity of $id_3$ into lower orders, so we need to analyze
the algebras $A_{00},A_{11},A_{22}$.  The algebra $A_{00}$ is generated by the empty picture $\emptyset$
and the ring $R$. Since $R^3=2R$, $A_{00}=\C[R]/(R^3-2R)$.  By
Lemma \ref{decomposition}, the three minimal idempotents are
$e_1=\emptyset-\frac{R^2}{2}, e_2=\frac{R^2+dR}{4},
e_3=\frac{R^2-dR}{4}$.  Testing on the disk, we know that $e_2$ is
of the trivial type.

For $A_{11}$, we apply the Jones-Wenzl idempotent $p_3$ to the
stacking of two twists $T^2$. After simplifying, we get
$T^4-dT^2+1=0$. Again by Lemma \ref{decomposition}, we have 4
minimal idempotents: $\frac{1}{2d}(\alpha^2 I+\alpha T -\alpha^4
T^2-\alpha^3 T^3)$, where $\alpha^4-d\alpha^2+1=0$.

A new phenomenon arises in the algebra $A_{22}$, which is
generated by $8$ diagrams: $1_2, \bar{B}B, T_2, \bar{B'}B,
\bar{B}B', \bar{B'}B',\bar{B}RB, \bar{B'}RB$. Computing their
inner products shows that $A_{22}\cong \C^8$. $A_{02}$ is spanned
by $B,B', RB, RB'$.  Using the three minimal idempotents in
$A_{00}$, we see that $e_0A_{02}$ is spanned by $RB+RB'=f_0$,
$e_1A_{02}$ is spanned by $B-\frac{d}{2}RB'=f_1,
B'-\frac{d}{2}RB=f_1'$, and $e_1A_{02}$ is spanned by
$RB-RB'=f_2$.  Hence $A_{22}^0\cong \C^6$ as the direct sum of $2$
$1\times 1$ matrix algebras generated by $f_0,f_2$ and a $2\times
2$ algebra generated by $f_1,f_1'$. Therefore there are two more
idempotents in $A_{22}$. Applying $p_3$ to the action of the
1/2-Dehn twist $F$ on $A_{22}$, we get $F^2=1$ modulo lower order
terms, hence the last two idempotents are of the form
$\frac{1}{2}(I_2 \pm iF)$ plus lower order terms in $A_{22}^0$.
Since $A_{22}^0=BA_{02}+B'A_{20}$, we need to find an $x$ such
that $e=-\frac{1}{2}1_2\pm \frac{i}{2}T_2+x$ is a projector and
$eB=eB'=0$.  Solve the equations, we find
$$e_{\pm}=\frac{1}{2}1_2\pm \frac{i}{2}T\mp
\frac{i}{2d}\bar{B'}B-\frac{1}{2d} \bar{B}B\mp \frac{i}{2d}
BB'-\frac{1}{2d}\bar{B'}B'\pm
\frac{i}{2d^2}\bar{B}RB+\frac{1}{2d^2} \bar{B'}RB.$$

\subsubsection{Level=3, $d^2=1+d$ or $d^2=1-d$}

The algebra $A_{00}$ is the algebra $\C[R]/(R^4-3R^2+1)$, so we
have $4$ minimal idempotents.

The algebra $A_{11}$ is generated by the twist $T$, so $A_{11}$ is
the algebra $\C[T]/(T^6-dT^4-dT^2+1)$, so we have $6$ minimal
idempotents.

Let $F$ be the fractional Dehn twist on $A_{22}$, then $p_4$
results in a dependence among $F^{-2},F^{-1}, I_2, F, F^2$:
$F^4-dF^2+1$ modulo lower order terms. So we have $4$ minimal
idempotents.

Let $F$ be the fractional Dehn twist on $A_{33}$, then $p_4$
results in a relation between $F^{-1}, I_3, F$. So we have $2$
minimal idempotents.

We leave the exact formula for the idempotents to interested
readers.  Note that the number of irreps of the annular TLJ
categories is the square of the corresponding TLJ rectangular
categories.

\subsection{Temperley-Lieb-Jones categories for primitive $4r$th roots of
unity}\label{irrepsoftlj4r}

 Let $A$ be a primitive $4r$-th root of unity, and $d=-A^2-A^{-2}$.
 $\TLJ^{\Rr,k,A}$ is just the $TL_d$ modulo its annihilator $p_{r-1}$.
 We found that it has minimal
 idempotents $p_0,p_1,p_2,\cdots, p_k$, $k=r-2$ and
 with image $(p_{k+1})$ being the annihilator of the Hermitian
 paring $\langle\, ,\, \rangle$.

 The case $A$ a primitive $2r$-th or $r$th root of unity, $r$ odd, e.g. $A=e^{2\pi
 i/6}, k=1,d=1;$ is identical as far as the rectangle categories
 go, but for the annular categories is more complicated; it is
 analyzed in the next section.

\begin{theorem}
 \tn{Rectangle diagrams with $p_i, 0\leq i\leq k$, near the
 bottom and object $t$ at top span spaces $\{W^t_{A,i}\}
 :=W_{A,i}$ on which $\Lambda :=\Lambda^{\Rr,k,A}$ acts from above.
 The families $\{W_{A,i}\}$ (as $i$ varies ) are the $k+1$
 (isomorphism classes of) irreducible representations of $\Lambda$.
 The involution $\hat{}$ is the identity.}
\end{theorem}

\begin{proof}
Most of the argument is by now familiar.  Resolving
the
 identity
 %{\includegraphics[width=1.5in,height=.95cm]{drawing5.15.eps }}
%{\centerline{Figure 5.15}}
%\vskip.1in
shows that any representation is a direct sum of $\{W_{A,i}\},
0\leq
 i\leq k$.

 For the first time $\T{dim}(W^t_{A,i})$ may be $>1$ and there
 will not be invertible morphism $t\rightarrow t'$ but
 irreducibility can still be proved as follows: for all $m=
 p_i\cdot m_0$ and $m'=p_i \cdot m_0'$ one may construct morphism
 $x$ and $y$ so that $m'=mx$ and $m=m'y$, where $p_i\in {}_i\Lambda_i,
 m, m_0\in {}_i\Lambda_a,m', m'_0\in {}_i\Lambda_b,x\in {}_a\Lambda_b,y\in {}_a\Lambda_b.$
\end{proof}

 It is a bit harder to find the irreps of $\Lambda^{\Aa,k,A} :=\Lambda$, but we will do this now.
 Similar irreps for $\TL^{\Aa}$ categories were previously found by Graham-Lehner \cite{grahamlehrer}, in a different
 context.

 We do not know how to proceed in a purely combinatorial fashion
 but must invoke the action of the doubled theory on the
 undoubled.  Topologically this amounts to the action on pictures
 in the solid cylinder $(B^2\times I, B^2\times \partial I)$
 under the addition of additional strands in a shell
 $(B^2_2\backslash B^2_1\times I;B^2_2\backslash B^2_1\times \partial
 I)$.  Logically our calculation should be done until we have
 established\footnote{It has been shown in Lemma \ref{smatrix} that the
 Temperley-Lieb-Jones theories
 $\Lambda^{\Rr,k,A}$ violate an important TQFT-axiom when
 $A^{2r}=1$.  The $S$-matrix is singular, half the expected rank, so the action
 of the mapping class group is not completely defined.  In this case irreps of
 $\Lambda^{\Aa,k,A}$ have a more complicated structure.} the undoubled
 TQFT based on $\Lambda^{\Rr,k,A}$ where the hypothesis $A$ a primitive $4r$th
 root is used.  This can be done in
 Section \ref{joneswenzlclosed} already or from here by going directly to
 Section \ref{joneskauffman} which does not depend on this section.  Therefore, we will
 freely invoke this material.

In the low level cases we found that $\#\T{irreps}
\Lambda^{\Aa}=(\#\T{irreps} \Lambda^{\Rr})^2$.  This is not an accident but
comes from identifying $\Lambda^{\Aa}$ with $\T{End}(\Lambda^{\Rr})$.
$\Lambda^{\Aa,k,A}$ is too complicated to ''guess'' the irreps so we
compute them from the endomorphism view point.

Recall from Section \ref{recouplingtheory} the projectors
$\omega_a=\overset{k}{\underset{c=0}{\sum}}\frac{\Delta_{(a+1)(c+1)}[c]}{D}$
 onto the $a$-label, and
$D^2=\overset{k}{\underset{c=0}{\sum}}\Delta_{c+1}^2$.

%\vskip.2in \epsfxsize=4.5in \centerline{\epsfbox{drawing5.16.eps}}{%\centerline{Figure 5.15}}
\vskip.2in

Also recall from Section \ref{joneskauffman} if $Y=\partial X$ and
$\gamma \subset \T{interior} X$ is a family of sccs labelled by
$\tilde{\omega_a}$ and $\gamma$ cobounds a family of imbedded
annuli $\Aa\subset X$ with $\gamma'\subset Y,$, i.e. $\partial
\Aa=\gamma\cup \gamma'$, then $Z(X,\gamma_{\omega_a})\in
V(Y\backslash \gamma; a,\hat{a})\subset
\bigoplus_{\T{admissible}}V(Y\backslash \gamma; l,\hat{l})=V(Y).$

Consider the $4-$component formal tangle in annulus cross
interval, $ -A\times I$, where $h=|i-j|$:

\begin{figure}[tbh]\label{4component}
\centering
\includegraphics[width=2.45in]{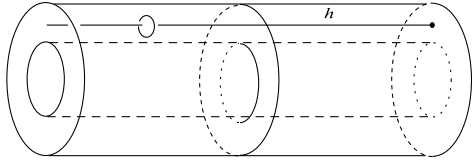}
\caption{4-component formal tangle}\label{4component}
\end{figure}

%\vskip.2in \epsfxsize=3.5in \centerline{\epsfbox{fig515.eps}}{\centerline{Figure 5.15}} \vskip.2in

Let $X$ be the 3-manifold made by removing small tubular
neighborhoods of the $h$-labeled arc, and write $X\cong Y\times
I$, where $Y$ is the annulus with a new puncture, and $\partial
X=D Y$ the double of $Y$.  Let $(Y,l_0)$ be $Y$ with $\partial Y$
labeled as follows: $\T{outer boundary}\rightarrow j, \T{inner
boundary}\rightarrow i, \T{new boundary}\rightarrow h.$ From Lemma
\ref{fusionrule} we know $V(Y,l_0)\cong V_{i,j,h}\cong \C$.

Another useful decomposition of $\partial X$ results from
expanding the inner and outer boundary components of $Y$ to
annuli, $A_i$ and $A_0$: $\partial X=-Y\cup +Y\cup -A_i\cup
+A_0\cup A_h.$ Applying $V$ we have: $V(\partial
X)=$
$$\bigoplus_{\T{admissible labels}} V^*(Y,l)\bigotimes
V(Y,l)\bigotimes V^*(A_i,l)\bigotimes V(A_0,l)\bigotimes V(A_h,l).
 (*)$$

Let us restrict to label: $l_0$.  By lemma, $\T{dim}V(Y,l_0)=1$
and let $x$ be the unit normalized vector $\kappa\in V(Y,l_0)$,

%\vskip.2in \epsfxsize=2in \centerline{\epsfbox{fig516.eps}}{\centerline{Figure 5.16}} \vskip.2in

The Jones-Wenzl projectors $p_i, p_j$ and $p_h$ are inserted as
shown.

\begin{figure}[tbh]\label{insertjones}
\centering
\includegraphics[width=1.25in]{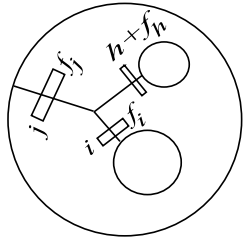}
\caption{Insert Jones-Wenzl}\label{insertjones}
\end{figure}

The arc diagram should be pushed into a ball $B^+$ bounding the
2-sphere $S^2$ made by capping $\partial Y$, to define an element
of $V(Y,l_0)$.  The root $\theta-$symbol normalizes $||x'||^2$ to
the invariant of $D(B^3)=S^3$ and the $s_{00}^{\frac{1}{2}}$ kills
this factor so as defined $||x'||^2=1$.

Let $V_{l_0}(\partial X)$ denote the $l_0$ summand of the rhs (*).
 Fixing $x'$, and therefore its dual $\hat{x}'$, give an
 isomorphism $\epsilon_{ij}': V_{l_0} \rightarrow
 \T{Hom}(V(A_i;i,\hat{i}), V(A_0;j,\hat{j})=:\T{Hom}(V_i,V_j)$.

 Consider the partition function $Z(X,L)$ of $(X,L)$, where $L$ is
 the 3-component link in $X$ labelled by ${\omega_i},
 {\omega_j}$ and ${\omega_h}$ in FIG??, and $Z(X,L)\subset V_{l_0}(X).$

 We now check that $\epsilon_{i,j}'(Z(X,L))$ is a non-zero vector
 in $\T{Hom}(V_i,V_j)$ whose definition is independent of phase(X).

 The pairing axiom can be used to analyze the result of gluing $X$
 to the genus two handlebody $(H,\overline{\theta}_{i,j,h})$ which is a
 thickening of the $i,j,h$ labelled $\theta$-graph (with the graph
 inside), we get:

 $$S_{00}\theta_{ijh}=S_{0i}^2S_{0j}^2S_{0h}\langle x,\hat{x}\rangle
 \langle \beta^*_{ii},\beta_{ii}\rangle
 \langle \beta_{jj},\beta^*_{jj}\rangle.\q\q (**)$$

 The two factors of $S_{0i}$ and $S_{0j}$ come from gluing
 along the seems separating off the inner and outer annuli (respectively);
 the factor $S_{0h}$ derives from gluing across the \lq\lq new
 component" of $\partial Y$, $S_{00}$ is the 3-sphere
 normalization constant, making the lhs (**) a \lq\lq spherical
 $\theta$-symbol".  Previously we arranged
 $\langle x^\prime , \widehat{x^\prime} \rangle =1$ and
 $\langle \beta^{\ast}_{aa}, \beta_{a0} \rangle= S^{-1}_{0, a}$
 so if we define
   $x=\frac{\sqrt{S_{0i}}\sqrt{S_{0j}}\sqrt{S_{0h}}}{\sqrt{S_{00}}\sqrt{\theta_{ijk}}}\cdot
   x'=\frac{\sqrt{S_{0i}}\sqrt{S_{0j}}\sqrt{S_{0h}}}{S_{00}}\kappa,$
   and redefine $\epsilon_{ij}'$ to $\epsilon_{ij}$ by replacing
   $x'$ with $x$ in its definition, we obtain:
   $$\epsilon_{ij}(Z(X,L))=\beta^*_{ii}\beta_{jj},$$ i.e. the
   canonical element of $\T{Hom}(V_i,V_j)$.

 Now attach a 2-handle to $X$ along the \lq\lq new"
 component of $\partial Y$ to reverse our original construction:
 $X\cup \T{2-handle}=A\times I$.  The co-core of the 2-handle
 should now be labeled by $h$ and the $\omega_h$-labeled
 component can be dispensed with (it is now irrelevant ).  Call
 this  new idempotent $3-$component formal tangle $\bar{L}_{ij}$.
 Fix $x$, as above, a map $\overline{\epsilon_{ij}}$ closely related to
 $\epsilon_{ij}$ is now defined: $\bar{\epsilon_{ij}}: V(A\times I,
 \bar{L})\rightarrow \T{Hom}(V_i,V_j)$ and as before we have

\begin{lemma}
$\bar{\epsilon_{ij}}(Z(A\times
I,\bar{L}))=\beta^*_{ii}\beta_{jj}.$
\end{lemma}

 Using the geometric interpretation of links in a product as
 operators,  and using the product structure from the middle factor in
$A\times I=S^1\times I\times I$, we see $\beta^*_{i,i}\beta_{jj}$
realized by a formal knot projection $(\overline{L})\subset
S^1\times \frac{1}{2}\times I.$  This projection can be
Kauffman-resolved to a formal $1-$submanifold $=:L_{ij}\subset
S^1\times I$ (ignoring the constant $\frac{1}{2}$).  This is the
minimal, in fact 1-dimensional idempotent of $\Lambda^{\Aa,k,A}$.  In
fact by counting we see that we have achieved a complete
resolution of the identity, and in the annular algebras
${}_i\Lambda^{\Aa}_{i}$ a complete list of isomorphisms classes of
irreducible representations of the full annular category
$\Lambda^{\Aa,k,A}$.

Our assumption in Section \ref{joneskauffman} is that $A$ is a
primitive $4r$-th root of unity, $r=k+2$.  Section
\ref{joneskauffman} constructs a TQFT with
$\{\tn{labels}\}=\{\T{irreps}\Lambda^{\Rr,k,A}\}$. Using the $s$-matrix
of this TQFT, we have just constructed a basis $\{\beta^*_{ii}
\beta_{jj}\}$ of $(k+1)^2$ operators for $\T{Hom}\left(
\overset{k}{\underset{i=0}{\bigoplus}} V_i,
\overset{k}{\underset{j=0}{\bigoplus}} V_j \right)$ which are
geometrically represented as formal submanifolds $\{L_{ij}\subset
A\}$, also an idempotent in ${}_h\Lambda_{h}^{\Aa,k,A}$.

The counting argument below holds for A a primitive $4r$th or
$2r$th root of unity $r$ odd or $r$th root of unity $r$ odd, and
so applies in the next section as well.

By a direct count of classical (not formal) pictures up to the
projector relation $p_{k+1}=p_{r-1}$ we find:

\[
\T{dim}\left({}_h\Lambda_h^{\Aa,k,A}\right)\leq \left\{
\begin{array}{ccc}
k, &\T{for} & i=0\\
2k+2-2h, &\T{for} & 1\leq h\leq k \\
\end{array}\right. \q(***) .
\]
Summing over $h$,
$\T{dim}\left(\overset{k}{\underset{h=0}{\bigoplus}}\,
{}_h\Lambda_{h}^{\Aa,k,A}\right)\leq (k+1)^2.$

Since $\{L_{ij}, 0\leq i,j\leq k\}$ represent as $(k+1)^2$
linearly independent operators, the above inequalities must, in
fact, be equalities.

${_0}\Lambda_{0}^{\Aa,k,A}$ is spanned by the empty picture
$\emptyset$, the ring circle $R$, and its powers up to $R^k$.  The
projector decomposes $R^{k+1}$ into a linear combination of lower
terms.

For $h=1$, let $I$ denote the straight arc picture and $T$ the
counter clockwise Dehn twist.  The pictures: $\bar{T}^{(k-1)},
\bar{T}^{k-2)}, \cdots, I, T, \cdots T^{k-1}$ appear (and are)
independent but there is an obvious dependency if
the list is expanded to $T^{-k}\cdots, T^k$.  This dependency
leads quickly to the claimed bound for $h=1$.

%\vskip.2in \epsfxsize=3in \centerline{\epsfbox{fig517.eps}}{\centerline{Figure 5.17}} \vskip.2in

For $h>1$, the argument is similar to the above, except a
fractional Dehn twist $F$ replaces $T$.
%For example, for $k=3$ and $h=2,3$ the first dependencies are illustrated in Fig??.

%\vskip.2in \epsfxsize=3.5in \centerline{\epsfbox{fig517-34.eps }}{\centerline{Figure 5.17 3/4}} \vskip.2in

\begin{lemma}
\tn{$\{L_{ij}, 0\leq i,j\leq k\}$ is a complete set of minimal
idempotents for $\{{}_h\Lambda_h^{\Aa,k,A} \}, 0\leq h\leq k$.}
\end{lemma}

Proof: From Fig.~\ref{4component}, each $\bar{L}$ is a minimal idempotent and
$L_{ij}$ represents the same operator.

Fixing $h>0$ now consider the action of $\lq\lq$fractional Dehn
twist", $F$ on $L_{ij}$.

%\vskip.2in \epsfxsize=3.5in \centerline{\epsfbox{fig518.eps}}{\centerline{Figure 5.18}} \vskip.2in

\begin{lemma}
\tn{For $i<j$, $F(L_{ij})=-A^{i+j+2}L_{ij},$ and for $i>j$,
$F(L_{ij})=-A^{i+j+2}L_{ij}$.}
\end{lemma}

Proof:  Use the Kauffman relation to resolve the diagram below,
noting the left kink is equal to a factor of $-A^3$ and that only
the resolution indicated by arrows gives a term not killed by the
projectors; its coefficient is $A^{i+j-1}$.

\begin{figure}[tbh]\label{lemma55}
\centering
\includegraphics[width=2.25in]{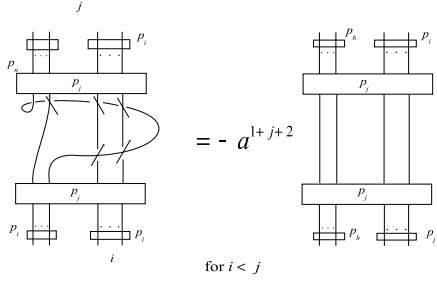}
\caption{Annular idempotent}\label{lemma55}
\end{figure}

For $i>j$ one considers the mirror image of the above, interchanging
$A$ and $A^{-1}$.

%\vskip.2in \epsfxsize=3.5in \centerline{\epsfbox{fig519.eps}}{\centerline{Figure 5.19}} \vskip.2in

If $h=0,i=j$, then $L_{ij}=\tilde{\omega_i}$ and we may consider
the action $R$ of ring addition to $\omega_i$.  Since $\omega_i$
is the projector to the i-th label, we have
$R(\tilde{\omega_i})=-(A^{2i+2}+A^{-2i-2})\tilde{\omega_i},$.

%\vskip.2in \epsfxsize=3in \centerline{\epsfbox{fig520.eps}}{\centerline{Figure 5.20}} \vskip.2in

  This establishes:

\begin{lemma}
$R(L_{ii})=-(A^{2i+2}+A^{-2i-2})L_{ii}.$
\end{lemma}

Let $V_{ij}$ be the vector space spanned by formal 1-submanifolds
in the annulus which near the inner boundary agree with $L_{ij}$.

Essentially the same argument employed for the rectangle
categories: resolution of the identities but now for $\tn{id}\in
{}_i\Lambda^{\Aa,k,A}_i, 0\leq i\leq k$ shows:

\begin{theorem} The spaces $\{V_{ij}\}$ form a complete set of irreps
for $\T{Rep}(\Lambda^{\Aa,k,A})$.  Direct sum decompositions into these
irreducibles are unique.
\end{theorem}

The irreps $V_{ij}$ have another \lq\lq diagonal" indexing by
$(h=|i-j|, \T{eigenvalue}(i,j))$.  We think of $h$ as the \lq\lq
crude label" it specifies the boundary condition (object); it is
refined into a true label by the additional information of
eigenvalue under fractional Dehn twist.

\iffalse Exercise: Draw all $L_{ij}, 0\leq i,j\leq k$, for $a=e^{
2\pi i/12}$ and $a=e^{2\pi i/16}.$\fi

Note that for ring addition: $\T{eigenvalue(i,j)}=-A^{i+j+2}$ for
$i\neq j$.

\subsection{Temperley-Lieb-Jones categories for primitive $2r$th root or $r$th root of unity, r
odd}\label{irrepsoftlj2r}

In Lemma \ref{smatrix} we compute the $S$ matrix associated to the
rectangle category $\Lambda^{\Rr,k,A}$, $A$ a primitive $2r^{\T{th}}$
root or $r$th root of unity, $r$ odd and find
 it is singular (Note that the last theorem of Ch XI \cite{turaev} holds only for even $r$ because the
 $S$ matrix is singular for odd $r$).  We find
 there is an involution on the label set $\bar{}: \{0,\cdots,
 k\}\rightarrow \{0, \cdots, k\}$ defined by $\bar{a}=k-a$ so that
 $S=S_{\T{even}}\bigotimes
 \big( \begin{smallmatrix}
 1&1 \\ 1&1
  \end{smallmatrix}\big)$.)
  We use the notation $i_{\T{even}}$ or
 just $i_e, 0\leq i\leq k$, to denote the even number $i$ or
 $\bar{i}$.  (Note that $\;\bar{} \;$ is not the usual duality $\; \hat{} \;$
 on labels which is trivial in the TLJ theory.  Also note
 that since $k$ is odd exactly one of $i$ and $\bar{i}$ is even,)
  The $\frac{k+1}{2}$ by $\frac{k+1}{2}$ matrix $S_{\T{even}}$ is
  nonsingular and defines an $SU(2)_k^{\textrm{even}}$-TQFT\footnote{
  These TQFTs are called $SO(3)$-TQFTs by many authors.  As noted in \cite{RSW},
  there is some mystery about those TQFTs as $SO(3)$-Witten-Chern-Simons TQFTs.
  Since they are the same TQFTs
  as $SU(2)$-Witten-Reshetikhin-Turaev TQFTs restricted to integral
  spins, therefore we adopt this notation.  Their corresponding MTCs are
  denoted by $(A_1,k)_{\frac{1}{2}}$ in \cite{RSW}.} on the even labels  at
  level $k$, explicitly:

\begin{equation*}
 S_{i_ej_e}=\sqrt{\frac{2}{r}}(-1)^{i+j}
 \big( [i+1][j+1] \big).\hspace{1in}{(5c.1)}
\end{equation*}

  The formal $1-$submanifolds $L_{i_ej_e}, 0\leq i,j\leq k$ can be
  defined just as in last section.  As operators on the
  $SU(2)_k^{\textrm{even}}$
  TQFT they are $\beta^{*}_{i_ei_e}\beta_{j_e j_e}$.  Also each
  $L_{i,j}$ has an interpretation as a formal 1-manifold in the
  category $\Lambda^{\Aa,k,e^{2\pi i/6}}$. (This is the \lq\lq $d=1$" category
  ($\lq\lq \Z_2"$-gauge theory) that we have been developing as a simple
  example.)

  Letting $i_0$ denote the odd index, $i$ or $\bar{i}$, the tensor
  decomposition of the $S$-matrix implies
  $\omega_{i_e}=\omega_{i_0}$.  It follows that
  $L_{i_e,j_e}=L_{i_0,j_0}$ and $L_{i_e,j_0}=L_{i_0,j_e}$ so we
  have found only half of the expected number of minimal
  idempotents.  Let $R_{i_e,j_e}$ be $\lq\lq$reverse"$(L_{i_e,j_e})$,
  $L_{i_e,j_e}$ with certain $(-1)$ phase factors.  That is, if
  $L_{i_e,j_e}=\sum_{n}a_n\alpha_n$ where $\alpha_i$ is a
  classical tangle then $R_{i_e,j_e}=\sum(-1)^{k_n}a_n\alpha_n$,
  where $k_n\stackrel{\T{def}}{=}k(\alpha_n))$ is the transverse
  intersection number with a radial segment, $s\times I\subset
  S^1\times I=\Aa$ in the annulus.  Similarly define $R_{i_e,j_0}$.

  Recall the four irreps of $\Z_2$-gauge theory, $0=\emptyset +R,
  e=\emptyset -R, m=I+T, $ and $em=I-T$, and consider the
  following bijection:
\begin{equation*}
\begin{split}
 & \{ L_{i_e,j_e}, R_{i_e,j_e}, L_{i_e,j_0}, R_{i_e,j_0}\}\stackrel{\beta}{\rightarrow} \\
 & \{ \beta^*_{i_e i_e}\bigotimes \beta_{j_e j_e}\bigotimes 0,
   \beta^*_{i_e i_e}\bigotimes \beta_{j_e j_e}\bigotimes e,
   \beta^*_{i_e i_e}\bigotimes \beta_{j_e j_e}\bigotimes m,
   \beta^*_{i_e i_e} \\
 & \bigotimes \beta_{j_e j_e}\bigotimes em \}
   \hspace{2in}
   (5c.2)
\end{split}
\end{equation*}

  \begin{theorem} $\beta$ is a bijection between the minimal
  idempotents of two graded algebras:
  \[
  {}_h\Lambda_h^{\Aa,k,A}\stackrel{\beta}{\rightarrow} {}_h\Lambda_h^{\Aa,k,e^{2\pi i/6}}
  \bigotimes \T{End}\left(\bigoplus_{j=\T{even}}\,\, {}_j\Lambda_j^{\Rr,k,A}\right),
  \]
  where $A$ is a primitive
  $2r^{\T{th}}$ root or $r$th root of unity, $r$ odd.  The bijection $\beta$ induces a
  bijection between the isomorphism classes of irreps of
  categories:

  \[
  \T{irreps.}
  \left(\Lambda^{\Aa,k,A}\right)\stackrel{\bar{\beta}}{\rightarrow}\T{irreps.}\left(\Lambda^{\Aa,k,e^{2\pi
  i/6}}\bigotimes \T{End}\left(\Lambda_{\T{even}}^{\Rr,k,A}\right)\right).
  \]

  \end{theorem}

  Proof: The second statement is by now the familiar consequence
  of the first and a \lq\lq resolution of the identity."

  The dimension count (upper bound) of last section applies equally for
  primitive $2r^{\T{th}}$ roots or $r$th roots, so it suffices to check that $R_{i_e,j_e}$
  and $R_{i_e,j_o}$ are idempotents.  Writing either as
  $\T{reverse}(L)=\sum_{n}(-1)^{k_n}a_n\alpha_n$ we square:
 \[
\begin{split}
{\left(\T{reverse($L$)}\right)}^2&=(\sum_{n}(-1)^{k_n}a_n\alpha_n)^2=
 \sum_{n,m}(-1)^{k_n+k_m}a_na_m\alpha_n\alpha_m=\\
 & \tn{reverse}\left(\sum_{n,m}a_na_m\alpha_n\alpha_m\right)=\tn{reverse}(L^2)=\tn{reverse}(L).
  \hspace{.1in}
  (5c.3)
\end{split}
\]

  In the third equality holds since intersection number with a
  product ray in $\Aa$ is additive under stacking annuli:

  $$k_n+k_m=k(\alpha_n)+k(\alpha_m)=k(\alpha_n\alpha_m)=k_{n,m}.$$

\section{The definition of a TQFT}

There are two subtle ingredients in the definition of a TQFT: the
framing anomaly and the Frobenius-Schur (FS) indicator.  For the
TQFTs in this paper, the diagram and black-white TQFTs have
neither anomaly nor non-trivial FS indicators, therefore, they are
the easiest in this sense.  The Jones-Kauffman TQFTs have anomaly,
but no non-trivial FS indicators.  Our version of the Turaev-Viro
$SU(2)$-TQFTs have non-trivial FS indicators, but no anomaly;
while the WRT TQFTs have both anomaly, and non-trivial FS
indicators.

Our treatment essentially follows \cite{walker91} with two
variations: first the axioms in \cite{walker91} apply only to
TQFTs with trivial FS indicators, so we extend the label set to
cover the non-trivial FS indicators; secondly we choose to resolve
the anomaly for $3$-manifolds only half way in the sense that we
endow every $3$-manifold with its canonical extension, so the
modular functors lead to only projective representations of the
mapping class groups. One reason for our choices is to minimize the topological
prerequisite, and the other is that for application to quantum
physics projective representations are adequate.

\subsection{Refined labels for TQFTs}

A TQFT assigns a vector space $V(Y)$ to a surface $Y$.  If $Y$ has
boundaries, then certain conditions for $\p Y$ have to be
specified for the vector space $V(Y)$ to satisfy desired
properties for a TQFT. In Section \ref{moritacutpaste}, we see
that crude boundary conditions need to be refined to the irreps of
the picture categories, which are the labels.  But for more
complicated theories such as Witten-Reshetikhin-Turaev TQFTs,
labels are not sufficient to encode the FS indicators. Therefore,
we will introduce a boundary condition category to formalize
boundary conditions.  More precisely, boundary conditions are for
small annular neighborhoods of the boundary circles.  Our
boundary condition category will be a strict weak fusion category
$\Cc$, which enables us to encode the FS indicator for a label by
marking boundaries with $\pm U$, where $U\in \Cc^0$. In our
examples, the strict weak fusion categories are the representation
categories of the TLJ categories. Then the labels are
irreps of $\TLJ$ categories. In
anyonic theory, labels are called superselection sectors,
topological charges, or anyon types, etc. Boundary conditions which are
labels are preferred because anyonic systems with such boundary
conditions are more stable, while general boundary conditions such
as superpositions of labels are difficult to maintain.

A fusion category is a finitely dominated semi-simple rigid linear
monoidal category with finite dimensional morphism spaces and simple unit.
A weak fusion category is like a fusion
category except that rigidity is relaxed to weak rigidity as
follows.  A monoidal category $\Cc$ is weakly rigid if every
object $U$ has a weak dual: an object $U^*$ such that $Hom(1,
U\otimes W)\cong Hom(U^*,W)$ for any object $W$ of $\Cc$.

A refined label set for a TQFT with a boundary condition category
$\Cc$ is a finite set $L^e=\{\pm V_i\}_{i\in I}$, where the label
set $L=\{V_i\}_{i\in I}$ is a set of representatives of
isomorphism classes of simple objects of $\Cc$, and $I$ a finite
index set with a distinguished element $0$ and $V_0=1$.  An
involution $\hat{}: L^e \la L^e$ is defined on refined labels
$l=\pm V_i $ by $\hat{l}=-l$ formally.  There is also an involution on the
index set $I$ of the label set: $\hat{i}=j$ if $V_j\cong V_i^*$.
 A label $V_i\in L$ is self-dual if $\hat{i}=i$, and a refined label is self dual if
 the corresponding label is self dual.
 A (refined) label set is self-dual if every (refined) label is self-dual. Each label $V_i$
 has an FS indicator $\nu_l$: $0$ if not
self-dual, and $\pm 1$ if self-dual. A self-dual label $V_i$ is
symmetrically self-dual or real in conformal field theory language
if $V_i^*=V_i$ in $\Cc$, then we say $\nu_i=1$, and
anti-symmetrically self-dual or pseudo-real if otherwise, then we
say $\nu_i=-1$, i.e., $V_i^*$ is not the same object $V_i$ in
$\Cc$, though they are isomorphic.  Secretly $-V_i$ is $V_i^*$,
and we will identify the label $-V_i$ with  $V_i$ if the label is
symmetrically self-dual; but we cannot do so if the label is
anti-symmetrically self-dual, e.g., in the
Witten-Reshetikhin-Turaev $SU(2)$ TQFTs. Frobenius-Schur
indicators are determined by the modular $S$ and $T$ matrices \cite{RSW}. Note that the
trivial label $1$ is always symmetrically self-dual.

\subsection{Anomaly of TQFTs and extended manifolds}

In diagram TQFTs in Section \ref{joneswenzlclosed}, we see that
$Z(X_1\cup_{Y_2} X_2)=Z(X_1)\cdot Z(X_2)$ as composition of linear
maps. For general TQFTs, this identity only holds up to a phase
factor depending on $X_1,X_2$ and the gluing map. Moreover, for
general TQFTs, the vector spaces $V(Y)$ for oriented surface $Y$
are not defined canonically, but depend on extra structures under
the names of $2$-framing, Lagrange subspace, or $p_1$ structure, etc.
A Lagrangian subspace of a surface $Y$ is a maximal isotropic
subspace of $H_1(Y;\R)$ with respect to the intersection pairing of
$H_1(Y;\R)$.  We choose to work with Lagrange subspaces to resolve
the anomaly of a TQFT.

An extended surface $Y$ is a pair $(Y,\lambda)$, where $\lambda$
is a Lagrangian subspace of $H_1(Y;\R)$.  Note that if $\p X=Y$, then $Y$
has a canonical Lagrange subspace $\lambda_X=ker(H_1(Y;\R)\la
H_1(X;\R))$.  In the following, the boundary $Y$
of a $3$-manifold $X$ is always extended by the canonical
Lagrangian subspace $\lambda_X$ unless stated otherwise.  For any
planar surface $Y$, $H_1(Y;\R)=0$, so the extension is unique.
Therefore, extended planar surfaces are just regular surfaces.

To resolve the anomaly for surfaces, we define a category of
labeled extended surfaces.  Given a boundary condition category
$\Cc$, and a surface $Y$, a labeled extended surface is a triple
$(Y;\lambda,l)$, where $\lambda$ is a Lagrnagian subspace of
$H_1(Y;\R)$, and $l$ is an assignment of a signed object $\pm
U\in \Cc^0$ to each boundary circle.  Moreover each boundary
circle is oriented by the induced orientation from $Y$, and
parameterized by an orientation preserving map from the standard
circle $S^1$ in the plane.

Given two labeled extended surfaces $(Y_i;\lambda_i,l_i), i=1,2$,
their disjoint union is the labeled extended surface $(Y_1\coprod
Y_2; \lambda_1\oplus \lambda_2, l_1\cup l_2)$.  Gluing of surfaces
has to be carefully defined to be compatible with the boundary
structures and Lagrangian subspaces.  Given two components $\g_1$
and $\g_2$ of $\partial Y$ parameterized by $\phi_i$ and labeled by signed
objects $\pm U$, and let $\tn{gl}$ be a diffeomorphism $\phi_2 \cdot
r\cdot \phi_1^{-1}$, where $r$ is the standard involution of the
circle $S^1$.  Then the glued surface $Y_{gl}$ is the quotient
space of $q: Y\rightarrow Y_{gl}$ given by $x\sim x'$ if
$\tn{gl}(x)=\tn{gl}(x')$.  If $Y$ is extended by $\lambda$, then
$Y_{gl}$ is extended by $q_{*}(\lambda)$.  The boundary surface
$\p M_f$ of the mapping cylinder $M_f$ of a diffeomorphism $f:
Y\rightarrow Y$ of an extended surface $(Y;\lambda)$ has a
canonical extension by the inclusions of $\lambda$.

Labeled diffeomorphisms between two labeled extended surfaces are
orientation, boundary parameterization, and label preserving
diffeomorphisms between the underlying surfaces. Note that we do
not require the diffeomorphisms to preserve the Lagrangian
subspaces.

\subsection{Axioms for TQFTs}\label{axiomsoftqft}

The category $\X^{2,e, l}$ of labeled extended surfaces is the
category whose objects are labeled extended surfaces, and the
morphism set of two labeled extended surfaces
$(Y_1,\lambda_1,l_1)$ and $(Y_2,\lambda_2,l_2)$ are labeled
diffeomorphisms.

The anomaly of a TQFT is a root of unity $\kappa$, and to match
physical convention, we write $\kappa=e^{\pi i c/4}$, and $c\in
\Q$ is well-defined mod $8$, and called the central charge
of a TQFT. Therefore, a TQFT is anomaly free if and only if the
central charge $c$ is $0$ mod $8$.

%\newpage

\begin{definition}

A $(2+1)$-TQFT with a boundary condition category $\Cc$, a refined
label set $L^e$, and anomaly $\kappa$ consists of a pair $(V,Z)$,
where $V$ is a functor from the category $\X^{2,e,l}$ of oriented
labeled extended surfaces to the category $\mV$ of finitely
dimensional vector spaces and linear isomorphisms composed up to powers of $\kappa$,
and $Z$ is an assignment for each oriented $3$-manifold $X$ with extended boundary,
$Z(X, \lambda) \in V(\p X;\lambda)$, where $\p X$ is extended by a Lagrangian subspace
$\lambda$.  We will use the notation $Z(X), V(\p X)$ if  $\p X$ is extended
by the canonical Lagrangian subspace $\lambda_X$.  $V$ is called a modular functor.  $Z$ is the partition
function if $X$ is closed in physical language, and we will call
$Z$ the partition function even when $X$ is not closed.

\vspace{.1in}

Furthermore, $V$ and $Z$ satisfy the following axioms.

\vspace{.2in}

{\bf{Axioms for $V$:}}

\begin{enumerate}

\item \tn{Empty surface axiom:}

\vspace{.1in}

$V(\emptyset)=\C$

\vspace{.1in}

\item \tn{Disk axiom:}

\vspace{.1in}

$V(B^2;l)\cong \begin{cases}
\C & \textrm{if $l$
is the trivial label}\\
0 & \textrm{otherwise} \end{cases}$,
where $B^2$ is a $2$-disk.

\vspace{.1in}

\item \tn{Annular axiom:}

\vspace{.1in}

$V(\Aa;a,b)\cong \begin{cases} \C &
\textrm{if $a=\hat{b}$}\\ 0 & \textrm{otherwise} \end{cases}$,
where $\Aa$ is an annulus, and $a,b\in L^e$ are refined labels.

\vspace{.1in}

\item \tn{Disjoint union axiom:}

\vspace{.1in}

$V(Y_1 \amalg Y_2;\lambda_1\oplus
\lambda_2, l_1\coprod l_2)\cong V(Y_1;\lambda_1, l_1)\otimes
V(Y_2;\lambda_2, l_2)$. The isomorphisms are associative, and
compatible with the mapping class group actions.

\vspace{.1in}

\item \tn{Duality axiom:}

\vspace{.1in}

$V(-Y;l)\cong V(Y;\hat{l})^*$.

\vspace{.1in}

The isomorphisms are compatible with mapping class group actions, with
orientation reversal and disjoint union axiom as follows:

a)  The isomorphisms $V(Y) \rightarrow V(-Y)^*$ and
$V(-Y)\rightarrow V(Y)^*$ are mutually adjoint.

b) Given $f: (Y_1;l_1)\rightarrow (Y_2;l_2)$ and let $\bar{f}:
(-Y_1; \hat{l_1})\rightarrow (-Y_2;\hat{l_2})$, then
$<x,y>=<V(f)x,V(\bar{f})y>$, where $x\in V(Y_1;l_1), y\in
V(-Y_1;\hat{l_1})$.

c) Let $\alpha_1\otimes \alpha_2\in V(Y_1\coprod
Y_2)=V(Y_1)\otimes V(Y_2)$, and $\beta_1\otimes \beta_2\in
V(-Y_1\coprod -Y_2)=V(-Y_1)\otimes V(-Y_2)$, then
$$<\alpha_1\otimes \alpha_2, \beta_1\otimes \beta_2>=<\alpha_1, \beta_1><\alpha_2,
\beta_2>.$$

\vspace{.1in}

\item \tn{Gluing Axiom:}

\vspace{.1in}

Let $Y_{gl}$ be the surface obtained from
gluing two boundary components of $Y$, then $V(Y_{gl})\cong
\oplus_{l\in L} V(Y;(l,\hat{l}))$, where $l,\hat{l}$ label the two
glued boundary components.  The isomorphism is associative and
compatible with mapping class group actions.

\vspace{.1in}

Moreover, the isomorphism is compatible with duality as follows:
let $\oplus_{i\in L}\alpha_i\in V(Y_{gl};l)=\oplus_{i\in L}
V(Y;l,(i,\hat{i}))$ and $\oplus_{i}\beta_i\in
V(-Y_{gl};\hat{l})=\oplus_{i\in L} V(-Y;\hat{l},(i,\hat{i}))$, then there are
non-zero real numbers $s_{i}$ for each label $V_i$ such that $$
<\oplus_{i}\alpha_i, \oplus_{i}\beta_i>=\sum_i
s_{i}<\alpha_i,\beta_i>.$$

\end{enumerate}

\vspace{.15in}

{\bf{Axioms for $Z$:}}

\vspace{.1in}

\begin{enumerate}

\item \tn{Disjoint axiom:}

\vspace{.1in}

If $X=X_1\amalg X_2$, then
$Z(X)=Z(X_1)\otimes Z(X_2)$.

\vspace{.1in}

\item \tn{Naturality axiom:}

\vspace{.1in}

If $f: (X_1, (\p X_1, \lambda_1))
\la (X_2, (\p X_2, \lambda_2))$ is a diffeomorphism, then $V(f):
V(\p X_1)\la V(\p X_2)$ sends $Z(X_1,\lambda_1)$ to $Z(X_2,\lambda_2)$.

\vspace{.1in}

\item \tn{Gluing axiom:}

\vspace{.1in}

If  $\p X_1={-Y_1}\amalg Y_2, \p
X_2={-Y_2}\amalg Y_3$, then $Z(X_1\cup_{Y_2} X_2)=\kappa^n
Z(X_1)Z(X_2)$, where
$n=\mu((\lambda_{-}X_1),\lambda_2,(\lambda_{+}X_2))$ is the Maslov index (see Appendix \ref{topology}).

\vspace{.1in}

More generally, if $X$ is an oriented $3$-manifold and let
 $Y_i,i=1,2$ be disjoint surfaces in $\p X$, extended by
$\lambda_i\subset \lambda_X, i=1,2$, and $f: Y_1\rightarrow Y_2$
be an orientation reversing dffeomorphism sending $\lambda_1$ to
$\lambda_2$.

\vspace{.1in}

Then $V(\p X)$ is isomorphic to $\sum_{l_1,l_2}V(Y_1;l_1)\otimes
V(Y_2;l_2)\otimes V(\p X \backslash (Y_1\cup Y_2);
(\hat{l_1},\hat{l_2}))$ by multiplying $\kappa^m$, where $l_i$ runs through all
labelings of $Y_i$, and $m=\mu(K,\lambda_1\oplus \lambda_2,
\Delta)$ (see Appendix \ref{topology}). Hence $Z(X)=
\oplus_{l_1,l_2} \kappa^m \sum_j\alpha^j_{l_1}\otimes\beta^j_{l_2}\otimes
\g^j_{\hat{l_1},\hat{l_2}}$.

\vspace{.1in}

If gluing $Y_1$ to $Y_2$ by $f$
results in the manifold $X_f$, then
$$Z(X_f)=\kappa^m \sum_{j,l}<V(f)\alpha^j_{l}, \beta^j_l>\g^j_{\hat{l},l}.$$

\vspace{.1in}

\item \tn{Mapping cylinder axiom:}

\vspace{.1in}

If $Y$ is closed and extended by $\lambda$, and $Y\times I$ is
extended canonically by $\lambda \oplus (-\lambda)$.
Then $Z(Y \times I, \lambda \oplus (-\lambda))=\textrm{id}_{V(Y)}$.

\vspace{.1in}

More generally, let $\tn{I}_{\tn{id}}$ be the mapping cylinder of
$\tn{id}: Y\rightarrow Y$, and $\tn{id}_l$ be the identity in
$V(Y;l)\otimes V(Y;l)^*$, then
$$ Z(\tn{I}_{\tn{id}}, \lambda \oplus (-\lambda))=\oplus_{l\in L(Y)}\tn{id}_l.$$

\end{enumerate}

\end{definition}

\hfill $\square$
\vspace{.5in}

First we derive some easy consequences of the axioms:

\begin{prop}\label{tqftaxiomcorollary}

\begin{enumerate}

\item  $V(S^2)\cong \mathbb{C}$

\item $Z(X_1\sharp X_2)=\frac{Z(X_1)\otimes Z(X_2)}{Z(S^3)}.$

\item {\it{Trace formula:}}  Let $X$ be a bordism from closed surfaces
$Y$, extended by $\lambda$, to itself, and $X_f$ be the closed $3$-manifold obtained by gluing $Y$
to itself with a diffeomorphism $f$.

Then $Z(X_f)=\kappa^{m}Tr_{V(Y)}(V(f))$, where
$m=\mu(\lambda(f),\lambda_Y\oplus f_*(\lambda),\Delta_Y)$ and $\lambda(f)$ is the graph of $f_{*}$, $\Delta_Y$ is the
diagonal of $H_1(-Y;\R)\oplus H_1(Y;\R)$.  In
particular, $Z(Y\times S^1)=\dim(V(Y))$.

\item The dimension of $V(T^2)$ is the number of particle types.

\end{enumerate}

\end{prop}

For a TQFT with anomaly, the representations of the mapping class
groups are projective in a very special way.  From the axioms, we
deduce:

\begin{prop}

The representations of the mapping class groups are given by the
mapping cylinder construction: given a diffeomorphsim $f: Y\la Y$ and $Y$ extended by $\lambda$,
the mapping cylinder $Y_f$ induces a map $V(f)=Z(Y_f):
V(Y)\la V(Y)$.  We have
$V(fg)=\kappa^{\mu(g_{*}(\lambda),\lambda,f^{-1}_{*}(\lambda))}V(f)V(g)$.

\end{prop}

It follows from this proposition that the anomaly can be
incorporated by an extension of the bordisms $X$, in particular,
modular functors yield linear representations of certain central
extensions of the mapping class groups.

\subsection{More consequences of the axioms}

For refined labels $a,b,c$, we have vector spaces $V_a=V(B^2;a),
V_{a,b}=V(\Aa_{ab}), V_{a,b,c}=V(P_{abc})$, where $P$ is a pair of
pants or three-punctured sphere.  Denote the standard orientation
reversing maps on $B^2,\Aa_{ab},P_{abc}$ by $\psi$.  Then
$\psi^2=id$, therefore $\psi$ induces identifications
$V_{abc}=V^*_{\hat{a}\hat{b}\hat{c}}$,
$V_{a\hat{a}}=V^*_{a\hat{a}}$, and $V_1=V_1^*$.  Choose basis
$\beta_1\in V_1, \beta_{a\hat{a}}\in V_{a\hat{a}}$ such that
$<\beta_a,\beta_a>=\frac{1}{d_a}$.

\begin{prop}

\begin{enumerate}

\item $Z(B^2\times I)=\beta_1\otimes \beta_1$

\item $Z(S^1\times B^2)=\beta_{11}$

\item $Z(X\backslash B^3)=\frac{1}{D} Z(X)\otimes \beta_1 \otimes
\beta_1$.

\end{enumerate}

\end{prop}

\begin{proof}

Let $B^3$ be a $3$-ball regarded as the mapping cylinder as the
identity map $id: B^2\la B^2$.  By the mapping cylinder axiom,
$Z(B^3)=\beta_1\otimes \beta_1$.  Gluing two copies of $B^3$
together yields $S^3$.  By the gluing axiom
$Z(S^3)=s_{00}=\frac{1}{D}$.  It follows that $Z(X\backslash
B^3)=\frac{1}{D} Z(X)\otimes \beta_1 \otimes \beta_1$.

\end{proof}

\begin{prop}\label{twist}

The action of the left-handed Dehn twist along a boundary
component labeled by $a$ of $B^2,\Aa_{ab},P_{abc}$ on $V_1,
V_{a,\hat{a}}$ or $V_{abc}$ is a multiplication by a scalar
$\theta_a$. Furthermore, $\theta_1=1,\theta_a=\theta_{\hat{a}}$,
and $\theta_a$ is a root of unity for each refined label $a$.

\end{prop}

\subsection{Framed link invariants and modular representation}

Let $K$ be a framed link in a $3$-manifold $X$.  The framing of
$K$ determines a decomposition of the boundary tori of the link
compliment $X\backslash \tn{nbd}(K)$ into annuli.  With respect to
this decomposition,
$$Z(X\backslash \tn{nbd}(K))=\oplus_l J(K;l)\beta_{a_1\hat{a_1}}\otimes \cdots \otimes \beta_{a_n\hat{a_n}},$$
where $J(k;l)\in \C$ and $l=(a_1,\cdots, a_n)$ ranges over all
labelings of the components of $K$.  $J(K;l)$ is an invariant of
the framed, labeled link $(K;l)$.  When $(V,Z)$ is a
Jones-Kauffman or WRT TQFT, and $X=S^3$, the resulting link
invariant is a version of the celebrated colored Jones polynomial
evaluated at a root of unity.  This invariant can be extended to
an invariant of labeled, framed graphs.

A framed link $K$ represents a $3$-manifold $\chi(K)$ via surgery.
Using the gluing formula for $Z$, we can express $Z(\chi(K))$ as a
linear combination of $J(K;l)$:
$$Z(\chi(K))=\sum_l c_lJ(K;l).$$

Consider the Hopf link $H_{ij}$ labeled by $i,j\in L$.  Let
$\ts_{ij}$ be the link invariant of $H_{ij}$.  Note that when a
component is labeled by the trivial label, then we may drop the
component from the link when we compute link invariant.
Therefore, the first row of $\ts$ consists of
invariants of the unknot labeled by $i\in L$.  Denote ${\tilde{s}}_{i0}$ as
$d_i$, and $d_i$ is called the quantum dimension of label $i$.  In
Prop. \ref{twist}, each label is associated with a root of unity
$\theta_i$, which will be called the twist of label $i$.  Define
$D=\sqrt{\sum_{i\in L} d_i^2}$, and $S=\frac{1}{D}\ts,
T=(\delta_{ij}\theta_i)$, then $S,T$ give rise to a
representation of $SL(2,\Z)$, the mapping class group of $T^2$.

\subsection{Verlinde algebras and Verlinde formulas}

Let $T^2=S^1\times S^1=\p D^2 \times S^1$ be the standard torus.
Define the meridian to be the curve $\mu=S^1\times {1}$ and the
longitude to be the curve $\lambda={1}\times S^1$.

Let $(V,Z)$ be a TQFT, then the Verlinde algebra of $(V,Z)$ is the
vector space $V(T^2)$ with a multiplication defined as follows:  consider the two
decompositions of $T^2$ into annuli by splitting along $\mu$ and
$\lambda$, respectively.  These two decompositions determine two
bases of $V(T^2)$ denoted as $m_a=\beta_{a\hat{a}}$, and
$l_a=\beta_{\hat{a}a}$.  These two bases are related by the
modular $S$-matrix as follows:

\begin{equation}
l_a=\sum_{b}s_{ab}m_b, m_a=\sum_b s_{\hat{a}b}l_b.
\end{equation}

  Define $N_{abc}=\dim V(P_{abc})$, then we
have

\begin{equation}
m_bm_c=\sum_{a}N_{a\hat{b}\hat{c}}m_a.
\end{equation}

The multiplication makes $V(T^2)$ into an algebra, which is called the Verlinde
algebra of $(V,Z)$.

In the longitude bases $l_a$, the multiplication becomes

\begin{equation}
l_al_b=\delta_{ab}s^{-1}_{0a}l_a.
\end{equation}

This multiplication also has an intrinsic topological definition:
$Z(P\times S^1)$ gives rise to a linear map from $V(T^2)\times
V(T^2)\rightarrow V(T^2)$ by regarding $P\times S^1$ as a bordism
from $T^2 \coprod T^2$ to $T^2$.

The fusion coefficient $N_{abc}$ can be expressed in terms of $s_{ab}$, we have
\begin{equation}
N_{abc}=\sum_{x\in L} \frac{s_{ax}s_{bx}s_{cx}}{s_{0x}}.
\end{equation}

More generally, for a genus=$g$ surface $Y$ with $m$ boundaries
labeled by $l=(a_1\cdots a_m)$,

\begin{equation}\label{verlindeformula}
 \dim V(Y)=\sum_{x\in L}
s_{0x}^{2-2g-n}(\prod s_{a_ix}).
\end{equation}

\section{Diagram and Jones-Kauffman TQFTs}\label{joneskauffman}

For the remaining part of the paper, we will construct picture
TQFTs and verify the axioms for those TQFTs.  Our approach is as
follows:  start with a local relation and a skein relation, we
first define a picture category $\Lambda$ whose objects are points with
decorations in a $1$-manifold $X$ which is either an interval $I$
or a circle $S^1$, and morphisms are unoriented sub-$1$-manifolds
in $X\times I$ with certain structures connecting objects (=points
in $X\times \{0\}$ or $X\times \{1\}$). More generally, the
morphisms can be labeled trivalent graphs with coupons. Those picture
categories serve as crude boundary conditions for defining picture
spaces for surfaces with boundaries. Secondly, we find the
representation category $\Cc$ of $\Lambda$, which is a spherical tensor
category.  The irreps will be the labels. In the cases that we are
interested, the resulting spherical categories are all ribbon
tensor categories. Thirdly, we define colored framed link
invariants with the resulting ribbon tensor category in the second
step. Invariants of the colored Hopf links with labels form the
so-called modular $S$-matrix.  Each row of the $S$-matrix can be
used to define a projector $\omega_i$ which projects out the
$i$-th label if a labeled strand goes through a trivial
circle labeled by $\omega_i$.

\begin{figure}[tbh]\label{projectionid}
\centering
\includegraphics[width=2.45in]{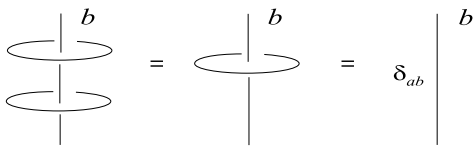}
\caption{Projectors}\label{projectionid}
\end{figure}

The projector $\omega_0$ is used to
construct the resulting $3$-manifold invariant. Finally, we define
the partition function $Z$ for a bordism $X$ using a handle
decomposition.  This construction will yield a TQFT if the
$S$-matrix is nonsingular, which is always true for the annular
TLJ cases.  If the $S$-matrix is singular, we still have a
$3$-manifold invariant, but we cannot define the representations
of the mapping class groups for high genus surfaces, though
representations of the braid groups are still well defined.

\subsection{Diagram TQFTs}\label{joneswenzlfull}

In this section, we
outline the proof that for some $r\geq 3$, $A$ a primitive $4r$th root of unity, or a primitive
$2r$th root of unity and $r$ odd, or a primitive
$r$th root of unity and $r$ odd, the diagram theories $\Pic^A(Y), Z_D$ defined in
Section \ref{joneswenzlclosed} indeed satisfy the axioms of TQFTs.

The diagram TQFTs are constructed based on the TLJ annular
categories.  The boundary condition categories $\Cc$ are the
representation categories of the TLJ annular categories $\Lambda$. A
nice feature of those TQFTs is that we can identify the objects of
the TLJ annular categories $\Lambda$ as boundary conditions using
Theorem \ref{structuretheorem}: each object in $\TLJ$ gives rise
to a representation of $\Lambda$ and therefore becomes an object of
$\Cc$, which is in general not simple, i.e., not a label. Hence
picture vector spaces are naturally vector spaces for the diagram
TQFTs.

For the diagram TQFTs, all labels are self-dual with trivial FS indicators.  Therefore,
it suffices to use only the label set.
The label sets of the diagram TQFTs are given by the idempotents
$L=\{\omega_{i,j,h}\}$ in Fig.~\ref{annularprojector}.  Given a surface $Y$ with
$\p Y=\g_1,\cdots,\g_m$, and each boundary circle $\g_i$ labeled by an
idempotent $e_i \in L$.  Then the picture space
$\Pic_D^A(Y;e_1,\cdots,e_m)$ consists of all formal pictures that
agree with $e_i$ inside a small annular neighborhood $\Aa_i$ of
the boundary $\g_i$ modulo the Jones-Wenzl projector $p_{r-1}$
outside all $\Aa_i$'s in $Y$. Given a bordism $X$ from $Y_1$ to
$Y_2$, the partition function $Z_D(X)$ is defined in Section
\ref{closeddiagram}.  Now we verify that $(\Pic^A,Z_D)$ is indeed
a TQFT.

For the axioms for modular functor $V$:

(1) is obvious.

(2) Since Jones-Wenzl projectors kill any turn-backs, then
$\Pic^A(B^2; \omega_{i,j,h})=0$ unless $h=0$.  For $h=0$, all
pictures are multiples of the empty diagram.

(3) Since $\Hom(p_i,p_j)=0$ unless $i=j$, so
$\Pic^A(\Aa;\omega_{i',j',h'},\omega_{i,j,h})=0$ unless $h=h'$. If
$h=h'$, then we have $\omega_i \cdot \omega_{i'}$ and $\omega_j
\cdot \omega_{j'}$, respectively in the annulus.  Recall that
$\omega_a\omega_b=\delta_{ab}\omega_a$, it follows that unless
$i=i',j=j'$, $\Pic^A(\Aa;\omega_{i',j',h'},\omega_{i,j,h})=0$.

(4) Obvious

(5) $\Pic^A(-Y)=\Pic^A(Y)$, hence duality is obvious.

(6) Gluing follows from Morita equivalence.

The axioms of partition function $Z$ follow from handle-body theory and properties of the
$S$ matrix.

The action of the mapping class groups is easy to see: a diffeomorphism maps
one multicurve to another.  Since a diffeomorphism preserves the local relation and skein
relation, this action sends skein classes to skein classes.  The compatibility of the action
with the axioms for vectors spaces is easy to check.

\subsection{Jones-Kauffman TQFTs}

In this section, we
outline the proof that for $r\geq 3$, $A$ a primitive $4r$th root of unity, the Jones-Kauffman
skein theories $V_{JK}^A(Y), Z_{JK}$ defined in
Section \ref{joneswenzlclosed} indeed satisfy the axioms of TQFTs.

The boundary condition category for a Jones-Kauffman TQFT is the
representation category of a TLJ rectangular category.  The label
set is $L=\{p_i\}_{i\in I}$, and $I=\{0,1,\cdots, r-2\}$.  Same reason as for the
diagram TQFTs, we need only the label set.

The new feature of the Jones-Kauffman TQFTs is the framing anomaly.  If $A$ and $r$ as in
Lemma \ref{jwcases}, then the central charge is $\frac{3(r-2)}{r}$.

Given an extended surface $(Y;\lambda)$, the modular functor $V(Y;\lambda)$ is defined in
Section \ref{jkskeinspace}.  If $\partial X=Y$, then we define $Z(X)$ as the skein class in $K_A(\partial X)$
represented by the empty skein.  TQFT axioms for $V$ and $Z$ follow from theorems in Section \ref{jkskeinspace}.
The non-trivial part is the mapping class group action.  This is explained at the end of Section \ref{jkskeinspace}.

\iffalse
\subsection{Walker-Turaev theorem}

 As an application, we prove the Walker-Turaev theorem
that for any oriented closed $3$-manifold $X$,
$Z_{\tn{D}}(X)=|Z_{\tn{JK}}(X)|^2$.
\fi

\section{WRT and Turaev-Viro $SU(2)$-TQFTs}\label{wrttqfts}

The pictorial approach to the Witten-Reshetikhin-Turaev $SU(2)$
TQFTs was based on \cite{kirbymelvin}.  The paper \cite{kirbymelvin}
finished with $3$-manifold invariants, just as \cite{kauffmanlins}
for the Jones-Kauffman theories.  The paper \cite{BHMV} took the picture
approach in \cite{kauffmanlins} one step further to TQFTs, but the
same for WRT TQFTs has not been done using a pictorial approach.
The reasons might be either people believe that this has been done
by \cite{BHMV} or realize that the Frobennius-Scur indicators make
a picture approach more involved. It is also widely believed that
the two approaches resulted in the same theories. But they are
 different. The spin $1/2$ representation of quantum group
$SU(2)_q$ for $q=e^{\pm 2\pi i/r}$ has a Frobenius-Schur
indicator=$-1$, whereas the corresponding label $1$ in
Temperley-Lieb-Jones theories has Frobenius-Schur indicator $=1$.
The Frobenus-Schur indicators $-1$ in the Witten-Reshetikhin-Turaev
theories introduce some $-1$'s into the $S$-matrix, hence for the
odd levels $k$, these $-1$'s change the $S$-matrix from singular in the Jones-Kauffman theories
when $A=\pm ie^{\pm \frac{2\pi i}{4r}}$ to non-singular. For even
levels $k$, the $S$-matrices are the same as those of the Jones-Kauffman
TQFTs, even though the TQFTs are different theories (see \cite{RSW} for the level=$2$ case).

In the pictorial TLJ approach to TQFTs, there is no room to encode
the Frobenius-Schur indicators $-1$.  In this section, we
introduce $\lq\lq$flag" decorations on each component of a framed multicurve
which can point to either side of
 the component. These flags allow us to
encode the FS indicator $-1$, hence reproduce the
Witten-Reshetikhin-Turaev $SU(2)$ TQFTs exactly. The doubled
theories of WRT TQFTs are not the diagram TQFTs, and will be
called the Turaev-Viro $SU(2)$-TQFTs.  They are direct products of
WRT theories with their mirror theories.

\subsection{Flagged TLJ categories}

In flagged TLJ categories, the local relation is still the
Jones-Wenzl projectors, but the skein relation is not the Kauffman
bracket exactly, but a slight variation discovered by R.~Kirby and
P.~ Melvin in \cite{kirbymelvin}.

The skein relation for resolving a crossing $p$ is given in
\cite{kirbymelvin} is as follows: if the two strands of the
crossing belongs to two different components of the link, then the
resolution is the Kauffman bracket in Figure
\ref{kauffmanbracket}; but if the two strands of the crossing $p$
are from the same component, then a sign $\e(p)=\pm 1$ is
well-defined, and the skein relation is:

\begin{figure}[tbh]\label{kirbymelvins}
\centering
\includegraphics[width=3.45in]{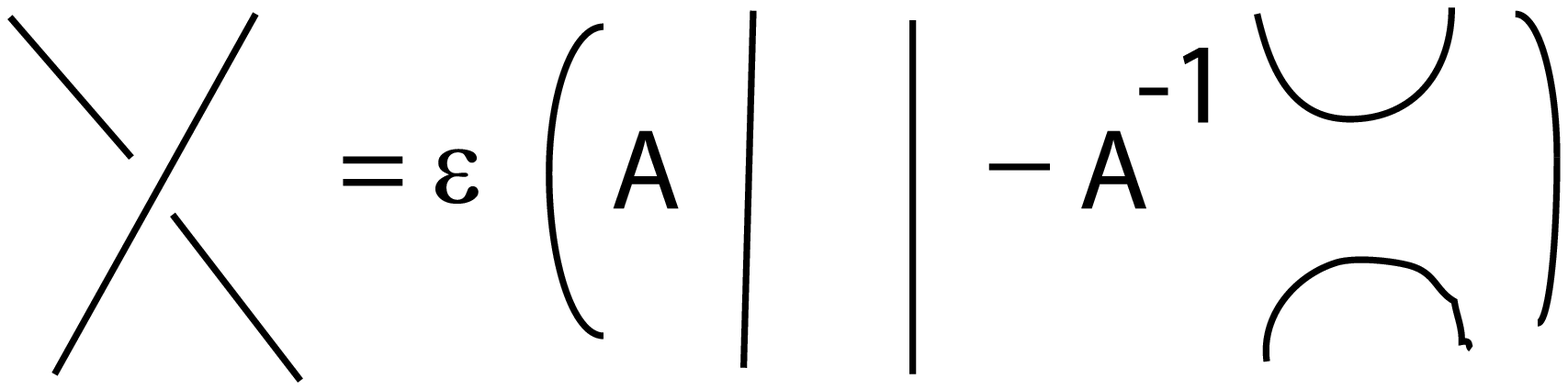}
\caption{Kirby-Melvin skein relation}\label{kirbymelvin}
\end{figure}

The flagged TLJ categories have objects signed points in the
interval and morphism flagged multicurves as follows: given an
oriented surface $Y$, and a multicurve $\g$ in the interior of
$Y$, and no critical points of $\g$ are within small neighborhoods
of $\p Y$. Let $\g \times [-\e,\e]$ be a small annulus
neighborhood of $\g$. A flag of $\g$ at $p\in \g$ is an arc
$p\times [0,\e]$ or $p\times [-\e, 0]$.  A flag is admissible if
$p$ is not a critical point of $\g$. A multicurve $\g$ is flagged
if all flags on $\g$ are admissible and the number of flags has
the same parity as the number of critical points of $\g$.  An
admissible flag on $\g$ can be parallel transported on $\g$ so
that when the flag passes through a critical point, it flips to
the other side.  In the plane, this is the same as parallel
transport by keeping the flag parallel at all times in the plane.
A multicurve is flagged if all its components are flagged.

Given a surface $Y$ with signed points on the boundary.  Each
signed point is flagged so that if the sign is $+$, the flag
agrees with the induced orientation of the boundary; if the sign
is $-$, the flag is opposite to the induced orientation.   Let
$\C[\scc]$ be the space of all formal flagged multicurves in $\Rr$
with signed points at the bottom and top, then the morphism set
between the bottom signed point and the top signed point of
$\TL^{\tn{flag}}$ is the quotient space of $\C[\scc]$ such that

\begin{enumerate}

\item Flags can be parallel transported

\item Flipping a flag to the other side results in a minus sign

\item Two neighboring flags can be cancelled if there are no
critical points between them and they are on the opposite sides.

\item Apply Jones-Wenzl projector to any part of an multicurve
with no flags.

\end{enumerate}

Then all discussions for TL apply to $\TL^{\tn{flag}}$.  The
representation category is similarly given by the same Jones-Wenzl
projectors.  The biggest difference from TL is the resulting
framed link invariant.

\begin{lemma}

Given a framed link diagram $D$, then the WRT invariant $<D>_{KM}$
of $D$ using Kirby-Melvin skein relation and the Jones-Kauffman
invariant $<D>_K$ using Kauffman bracket is related by:
\begin{equation}
<D>_{KM}(A)=(-i)^{D\cdot D} <D>_K (iA).
\end{equation}

\end{lemma}

\subsection{Turaev-Viro Unitary TQFTs}

Fix $A=\pm e^{\pm \frac{2\pi i}{4r}}$ for some $r\geq 3$.

The label set is the same as that of the corresponding diagram TQFT,
but for the first time we need to
work with the refined label set.

Given a surface $Y$ with boundaries labeled by
refined labels $\e_iV_i$.  If $\e_i=1$, we flag the point to the
orientation of $\p Y$; if $\e_i=-1$, we flag the points opposite
to the orientation of $\p Y$.  Then define the modular functor space
analogous to the skein space replacing multicurves with flagged multicurves.
The theories are similar enough so we will leave the details to interested readers.
The difference is that when the level=$k$ is odd, our version of the Turaev-Viro theory
is a direct product-a trivial quantum double,
while the corresponding diagram TQFT is a non-trivial
quantum double.

\subsection{WRT Unitary TQFTs}

Fix $A=\pm e^{\pm \frac{2\pi i}{4r}}$ for some $r\geq 3$.

The label set of a WRT TQFT is the same as that of the corresponding Jones-Kauffman TQFT, but it needs to
be extended to the refined label set.  The central charge of a level=$k$ theory is $\frac{3k}{k+2}$.
The discussion together with the Turaev-Viro theories is completely parallel to the Jones-Kauffman TQFTs
with diagram TQFTs.

\section{Black-White TQFTs}\label{blackwhite}

Interesting variations of the TLJ categories can also be obtained
by $2$-colorings: the black-white annular categories
$\TLJ^{\tn{BW}}_d$. The objects of the category are the objects of
the corresponding annular TLJ category enhanced by two colorings
of the complements of the points. In particular there are two circles:
black and white.
Morphisms between two objects are enhanced by colorings of the
regions.  A priori there are two enhancements of each Jones-Wenzl
idempotent, but it has been proved in \cite{freedmanmag} that the
two versions are equivalent.

\subsection{Black-white TLJ categories}

Fix some $r\geq 3$ and $A$, where $A$ is a primitve $4r$th root of unity, or a primitive
$2r$th root of unity and $r$ odd, or a primitive
$r$th root of unity and $r$ odd

The objects of black-white TLJ categories are points in the
interval or $S^1$ with a particular $2$-coloring of the complementary
intervals so that adjacent intervals having different colors. Given
two objects, morphisms are multicurves from the bottom to top
whose complement regions have black-white colors that are
compatible with the objects, and any two neighboring regions
receive different colors.  The local relation is the $2$-color
enhanced Jones-Wenzl projector and the skein relation is the
$2$-color enhancement of the Kauffman bracket.  The
representation theories of the black-white categories are
considerably harder to analyze.

The object with no points in the circle has two versions
$0_B,0_W$, which might be isomorphic.  Indeed sometimes they are
isomorphic and sometimes not. Therefore a skeleton of a
black-white TLJ category can be identified with
$\{0_B,0_W,2,4,\cdots \}$ with the possibility that $0_B=0_W$.  We
will draw the black object $0_b$ as a bold solid circle, and the
white object as a dotted circle. Interface circles between black
and white regions will be drawn as regular solid circles.
Morphisms will be drawn inside annuli, directed and composed from
inside-out.  There are two color changing morphisms $r_{bw} \in
\Hom(0_b,0_w), r_{wb} \in \Hom(0_w,0_b)$.

Let us denote the two compositions $r_{bw}\cdot r_{wb}=x_b\in
\Hom(0_b,0_b), r_{wb}\cdot r_{wb} =x_w\in \Hom(0_w,0_w)$, which
are just rings in the annulus.

Given an oriented closed surface $Y$, a $2$-colored multicurve in
$Y$ is a pair $(\g, c)$, where $\g$ is a multicurve, and $c$ is an
assignment of black or white to all regions of $Y\backslash \g$ so
that any two neighboring regions have opposite colors.  Let
$\C[\scc]$ be the vector space of formal $2$-colored multicurves,
and $\Pic^{\tn{BW}}(Y)$ be the quotient space of $\C[\scc]$ modulo
the BW-enhancement of JW projectors.

\subsection{Labels for black-white theories}

Recall in Section \ref{annulusmarkov}, we define the element
$q_{2m}\in C_{2m}(x)$.

\begin{lemma}

The element $q_{2m}$ is a minimal idempotent of $C_{2m}(x)$.

\end{lemma}

\begin{prop}

(1): If $r$ is even, then $x_b, x_w$ are not invertible, hence
$0_b$ is not isomorphic to $0_w$.

(2): If $r$ is odd, then $x_b, x_w$ are invertible, hence $0_b$ and
$0_w$ are isomorphic.

(3): The color swap involution is the identity on the TQFT vector spaces.

\end{prop}

\subsubsection{Level=2, $d^2=2$}

The algebra $\textrm{A}_{0_b0_b}\cong \C^2$, and so is the algebra
$\tn{A}_{0_w0_w}$.  Both are generated by $x$, so $\cong
\C[x]/(x^2=2x)$.

$\Hom(0_b,0_w)\cong \C$ is generated by $r_{bw}$. Similarly,
$\Hom(0_w,0_b)\cong \C$ is generated by $r_{wb}$.

$A_{2,2}\cong \C^4$.  Following the same analysis as in Section \ref{repoftlj}, we get
the irreps denoted by the following:

\begin{table}\label{bwlevel2}
\centering
\begin{tabular}{c|c|c|c|}
\hline $\rho_3$ &0&0& 1\\
\hline $\rho_2$ &1&0& 1\\
 \hline
 $\rho_1$ & 0 & 1 & 1\\
\hline
 $\rho_0$ & 1 & 1 & 1\\
\hline
 &  $0_b$ & $0_w$  & 2\\
%\hline
\end{tabular}\caption{Irreps of black-white level=2}
\end{table}

\subsubsection{Level=3}

The algebra $A_{0_B,0_B}\cong \C^2$ is generated by $x$ so $\cong
\C[x]/(x^2-3x+1)$.

The algebra $A_{22}\cong \C^7$.  Similar analysis as above leads to:

\begin{table}\label{bwlevel3}
\centering
\begin{tabular}{c|c|c|c|}
\hline $\rho_3$ &0&0& 1\\
\hline $\rho_2$ &0&0& 1\\
 \hline
 $\rho_1$ & 1 & 1 & 2\\
\hline
 $\rho_0$ & 1 & 1 & 1\\
\hline
 &  $0_b$ & $0_w$  & 2\\
%\hline
\end{tabular}\caption{Irreps of black-white level=3}
\end{table}

\subsection{BW TQFTs}

\begin{theorem}\label{BWTQFT}

(1): If $r\geq 3$, and $A$ a primitive $4r$th root of unity, or a primitive
$2r$th root of unity and $r$ odd, or a primitive
$r$th root of unity and $r$ odd, then
$(V^A_{BW},Z^A_{BW})$ is a TQFT.

(2): If $r$ odd, then $(V^A_{BW},Z^A_{BW})$ is isomorphic to the doubled
even TLJ sub-category TQFT, i.e., the TQFT from the quantum double of the even TLJ subcategory at the
corresponding $A$.

\end{theorem}

The proof of this theorem and the irreps for all $r$ are left to a future publication.

 We have not be able to identify the BW TQFTs with known ones when $r$ is even, and $A$ is a primitive $4r$th root of unity.
 If $r=4$, then  $(V^A_{BW},Z^A_{BW})$ is isomorphic to the toric code TQFT.  We conjecture Theorem \ref{BWTQFT} (2) still holds for these cases.  Furthermore, each $(V^A_{BW},Z^A_{BW})$ decomposes into a direct product of the toric code TQFT with another TQFT.

\section{Classification and Unitarity}\label{classificationunitary}

In this section, we classify all TQFTs based on Jones-Wenzl
projectors and Kauffman brackets.  Then we decide when the
resulting TQFT is unitary.  In literature $A$ has been chosen to
be either as a primitive $4r$-th root of unity or as a primitive
$2r$th root of unity.  We notice that for $r$ odd, when $A$ is a
primitive $r$th root of unity, the resulting $TLJ$ rectangular
categories give rise to ribbon tensor categories with singular
$S$-matrices, but their annular versions lead to TQFTs which are
potentially new.  Also when $A$ is a primitive $4r$th root of unity and $r$ even,
the BW TQFTs seem to contain new theories.

\subsection{Classification of diagram local
relations}\label{classifylocal}

By $d$ generic we mean that $d$ is not a
root of some Chebyshev polynomial $\tr_i$. Equivalently $d\neq B +
\bar{B}$ for some $B$ such that $B^\ell =1$.

Let us consider $d$-isotopy classes of multicurves on a closed
surface $Y$.  Call this vector space TL$_{d} (Y)$.  This vector
space has the  subtle structure of gluing formula associated to
cutting into subsurfaces (and then regluing); there is a product
analogous to both times and tensor products in TLJ$_d$.  Also for
special values of $d$ TL$_{d}(Y)$ has a natural singular Hermitian
structure.

\begin{theorem}\label{thm1}
If $d$ has the form: $d=-A^2 -A^{-2}$, $A$ a root of unity.
 Then there is a (single) local relation $R(d)$ so that
 ${\TL}_d (Y)$ modulo $R(d)$, denoted by $V_d(Y)$,  have finite nonzero dimension. If $d$ is not
of the above form then $V_d (Y) =0$ or $=\TL_d (Y)$ for any given
$R(d)$.  Furthermore the quotient space $V_d(Y)$ of  $\TL_d(Y)$
when it is neither $\{ 0\}$ nor TL$_d (Y)$ is uniquely determined,
and when $A$ is a primitive $4r$th root of unity, then $V_{d} (Y)$
is  the $\lq\lq$Drinfeld double" of a Jones-Kauffman TQFT at level
$k=r-2$.
\end{theorem}

\begin{proof} Consider a local relation $R_0 (d)$ of smallest
degree, say $2n$, which holds in $\TL_d$ (i.e. is a consequence of
$R(d)$). Arbitrarily draw $R_0 (d)$ in a rectangle with $n$
endpoints assigned to the top and $n$ endpoints assigned to the
bottom, to place $R_0 (d)$ in the algebra TL$_n(d)$. Adding any
cup or cap to $R_0$ gives a consequent relation of degree $=2n-2$;
this relation, by minimality, must be zero.  This implies that
$e_i R_0 (d) = R_0 (d) e_i =0$ for $1 \leq i \leq n-1$.  So by
Proposition \ref{joneswenzlprojectors}, $R_0 (d) = c p_{n,d}$, $c$
a nonzero scalar.

The trace, tr$(p_{n,d})\in \C$ is a degree $=0$ consequence of
$p_{n,d}$ so unless $d$ is a root of $\tr_n$, tr$(p_{n,d}) \neq 0$
and so generates all relations: $p_{n,d} (Y) =0$.

Now suppose $d$ is the root of two Chebyshev polynomials $\tr_m$
and $\tr_\ell$, $m < \ell$.  This happens exactly when $(m+1)$
divides $(\ell +1)$.  In fact to understand the roots of $\tr_n
(d)$ introduces a change of variables $d = B+B^{-1}$, then: $\tr_n
(d) = \break (B^{n+1} - B^{-n-1})/ (B-B^{-1})$.  The r.h.s.
vanishes (simply) when (and only when) $B$ is a $2n+2-$ root of
unity $\neq \pm 1$. In particular if $d$ is a root of $\tr_m$ and
$\tr_\ell$ then $p_{m,d}$ is a consequence of $p_{m, \ell}$ by
$\lq\lq$partial trace" as shown in Figure \ref{partialtrace}.

\begin{figure}\label{partialtraces}
\centering
\includegraphics[scale=1]{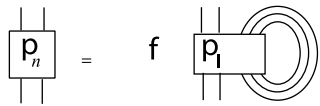}
\caption{Partial trace}\label{partialtrace}
\end{figure}

%\vskip.2in \epsfxsize=2in \centerline{\epsfbox{fig12.eps}}{\centerline{Figure 1.12}} \vskip.2in

Trace both sides to verify the coefficient $f=\f{\tr_m}{\tr_\ell}$
and note that since both numerator and denominator have simple
roots at $d$ the coefficient at $d$ is well defined and nonzero.

Thus for a diagram local relation (or set thereof) to yield a
nontrivial set of quotient space $\neq 0$, $d$ must be a root of
lowest degree $\tr_n$ and the relation(s) are equivalent to the
single relation $p_{n, d}$.

Geometrically, $0 =p_n = 1 +U$ implies $1=-U$ means that the
multicurves whose \underline{multiplicity} is less than $n$ along
the $1-$ cells, $SK^1(Y)$, ($\lq\lq$bonds" in physical language)
of any fixed triangulation of $Y$ determine ${p_n}(Y)$.  Here
multiplicity $\leq n$ for a multicurve $\g$ means that $\g$ runs
near $SK^1 (Y)$ and with fewer than $n$ parallel copies of a $1-$
cell (bond) of $SK^1(Y)$. Finite dimensionality of ${p_n}(Y)$ is
an immediate consequence. ($SK^1$ stands for $\lq\lq 1-$skeleton.")

The quotient space $ \neq 0$, for
$Y=S^2$ this follows from the nonvanishing of certain
$\theta$-symbols; for $Y$ of higher genus the Verlinde
formulas. \end{proof}

\subsection{Unitary TQFTs}

A unitary modular functor is a modular functor such that each
$V(Y)$ is endowed with a non-degenerate Hermitian pairing:
$$<,>: \overline{V(Y)}\times V(Y)\la \C,$$
and each morphism is unitary.  The Hermitian structures are
required to satisfy compatibility conditions as in the naturality
axiom of a modular functor. In particular,
$$<\oplus_{i}v_i, \oplus_j w_j>=\sum_i s_{i0}<v_i,w_j>.$$
Note that this implies that all quantum dimensions of particles
are positive reals.  It might be true that any theory with all
quantum dimensions positive is actually unitary.  Moreover, the
following diagram commutes for all $Y$:

$$
\begin{CD}
V(Y) @>\cong>> V(-Y)^* \\
@V{\cong}VV             @VV{\cong}V\\
\overline{V(Y)^*} @>\cong>> \overline{V(-Y)}
\end{CD}
$$
A unitary TQFT is a TQFT whose modular functor is unitary and
whose partition function satisfies $Z(-M)=\overline{Z(M)}$.

\subsection{Classification and unitarity}

There are two kinds of TQFTs that we studied in this paper:
undoubled and doubled, which are indexed by the Kauffman variable
$A$.

When $A$ is a primitive $4r$th root of unity for $r\geq 3$, we have the
Jones-Kauffman TQFTs.  The even sub-categories of TLJs yield TQFTs for $r$ odd, but have singular
$S$-matrix if $r$ even.  If $r$ is even, and $A=\pm i e^{\pm
\frac{2\pi i}{4r}}$, then the Jones-Kauffman TQFTs are unitary.

We also have the WRT $SU(2)$-TQFTs for $q=e^{\pm \frac{2\pi i}{r}}$,
which are unitary.  WRT TQFTs were believed to be the same as the
Jones-Kauffman TQFTs with $q=A^{\pm 4}$, but they are actually
different.  Jones-Kauffman TQFTs and WRT TQFTs are related by a
version of Schur-Weyl duality as alluded in Section
\ref{jonesrepandjoneswenzl} for the braid group representations.

All above theories can be doubled to get picture TQFTs: the doubled
Jones-Kauffman TQFTs are the diagram TQFTs, while the doubled WRT
TQFTs are the Turaev-Viro TQFTs \cite{turaevviro}.
The doubles of even sub-categories for $r$ odd form
part of the Black-White TQFTs in Theorem \ref{BWTQFT}, while for $r$ even this is still a conjecture.

When $A$ is a primitive $2r$th or $r$th root of unity and $r$ odd,
the TLJ categories do not yield TQFTs.  But the restrictions to the
even labels lead to TQFTs.  When $A=\pm i e^{\pm \frac{2\pi
i}{4r}}$, the resulting TQFTs are unitary. Those unitary TQFTs are
the same as those obtained from the restrictions of WRT TQFTs to
integral spins.  All can be doubled to picture TQFTs.
Note that for these cases when $A$ is
a primitive $r$th root of unity, then $-A$ is a primitive $2r$th root of unity.  For the
even sub-categories, they lead to the same TQFTs, which form part of
the Black-White TQFTs in Theorem \ref{BWTQFT}.

\hfill $\square$

\newpage

\appendix

\section{Topological phases of matter}

Fractional quantum Hall liquids are new phases of matter
exhibiting topological orders, and Chern-Simons theories are
proposed as effective theories to describe the universal
properties of such quantum liquids. Quantum Chern-Simons theories
are (2+1)-dimensional topological quantum field theories (TQFTs),
so we define a topological phase of matter as a quantum system
with a TQFT effective theory.

\subsection{Ground states manifolds as modular functors}

While in real experiments, we will prefer to work with quantum
systems in the plane, it is useful in theory to consider quantum
systems on 2-dimensional surfaces such as the torus. Given a
quantum system on a 2-dimensional oriented closed surface
$\Sigma$, there associates a Hilbert space $\mathbb{H}$ consisting
of all states of the system.  The lowest energy states form the
ground states manifold $V(\Sigma)$, which is a subspace of
$\mathbb{H}$. For a given theory, the local physics of the quantum
systems on different surfaces are the same, so there are relations
among the ground states manifolds $V(\Sigma)$ for different
$\Sigma$'s dictated by the local physics. In a topological quantum
system, the ground states manifolds form a modular functor---the
2-dimensional part of a TQFT. In particular, $V(\Sigma)$ depends
only on the topological type of $\Sigma$.

\subsection{Elementary excitations as particles}

A topological quantum system has many salient features including
an energy gap in the thermodynamic limit, ground states degeneracy
and the lack of continuous evolutions for the ground states
manifolds.  The energy gap implies that elementary excitations are
particle-like and particle statistics is well-defined.  These
quasi-particles are anyons, whose statistics are described by
representations of the braid groups rather than representations of
the permutation groups.

The mathematical model for an anyonic system is a ribbon category.
In this model an anyon is pictured as a framed point in the plane:
a small interval.  Given a collection of $n$ anyons in the plane,
we will arbitrarily order them and place them onto the real axis,
so we can represent them by intervals  $[i-\epsilon,
i+\epsilon],i=1,2,\cdots, n$ on the real axis for some small
$\epsilon$. The worldlines of any $n$ anyons from time $t=0$ to
the same set of $n$ anyons at time $t=1$ form a {\it framed} braid
in $\mathbb{R}^2 \times [0,1]$.  We will represent worldlines of
anyons by diagrams of ribbons in the plane which are projections
from $\mathbb{R}^2 \times [0,1]$ to the real axis $\times [0,1]$
with crossings. (Technically we need to perturb the worldlines in
order to avoid more singular projections.)  A further convention
is the so-called blackboard framing: we will draw only single
lines to represent ribbons with the understanding the ribbon is
the parallel thickening of the lines in the plane.

Suppose $n$ elementary excitations of a topological quantum system
on a surface $\Sigma$ are localized at points $p_1,p_2,\cdots,
p_n$, by excising the particles from $\Sigma$, we have a
topological quantum system on a punctured surface $\Sigma'$
obtained from $\Sigma$ by deleting a small disk around each point
$p_i$. Then the ground states manifold of the quantum system on
$\Sigma'$ form a Hilbert space $V(\Sigma')$. The resulting Hilbert
space should depend only on the topological properties of the
particles---particle types that will be referred to also as
labels. In this way we assign Hilbert spaces $V(\Sigma, a_1, a_2,
\cdots, a_n)$ to surfaces with boundary components $\{1, 2,\cdots,
n\}$ labelled by $\{a_1,a_2,\cdots, a_n\}$.

\subsection{Braid statistics}

The energy gap protects the ground states manifold, and when two
particles are exchanged adiabatically within the ground states
manifolds, the wavefunctions are changed by a unitary
transformation.  Hence particle statistics can be defined as the
resulting unitary representations of the braid groups.

\section{Representation of linear category}\label{repofcategory}

%\section{Category and its representation}

Category theory is one of the most abstract branch of mathematics.
It is extremely convenient to use category language to describe
topological phases of matters.  It remains to be seen whether or
not this attempt will lead to useful physics.  But tensor category
theory might prove to be the right generalization of group theory
for physics. On a superficial level, the two layers of structures
in a category fit well with physics: objects in a category
represent states, and morphisms between objects possible
$\lq\lq$physical processes" from one state to another.  For quantum
physics, the category will be linear so each morphism set is a
vector space. Functors might be useful for the description of
topological phase transitions and condensations of particles or
string-nets.  For more detailed introduction, consult the book
\cite{maclane}.

A {\it category} $\Cc$ consists of a collection of objects,
denoted by $a,b,c,\cdots$,  and a morphism set ${}_a\Cc_{b}$ (also
denoted by $\textrm{Mor}(a,b)$) for each ordered pair $(a,b)$ of
objects which satisfy the following axioms:

Given $f\in {}_a\Cc_{b}$ and $g\in {}_b\Cc_{c}$, then there is a
morphism $f\cdot g\in {}_a\Cc_{c}$ such that

\noindent 1)(Associativity):

If $f\in {}_a\Cc_{b}, g\in {}_b\Cc_{c},h\in {}_c\Cc_{d}$, then
$(f\cdot g)\cdot h=f\cdot (g\cdot h)$.

\noindent 2)(Identity):

For each object $a$, there is a morphism $\textrm{id}_a\in
{}_a\Cc_{a}$ such that for any $f\in {}_a\Cc_{b}$ and $g\in
{}_c\Cc_{a}$, $\textrm{id}_a \cdot f=f$ and $ g\cdot
\textrm{id}_a=g$. %\qed

We denote the objects of $\Cc$ by $\Cc^0$ and write $a\in \Cc^0$
for an object of $\Cc$.  We use $\Cc^1$ to denote the disjoint
union of all the sets ${}_a\Cc_{b}$ . The morphism $f\cdot g\in
{}_a\Cc_{c}$ is usually called the {\it composition} of $f\in
{}_a\Cc_{b}$ and $g\in {}_b\Cc_{c}$, but our notation $f\cdot g$
is different from the usual convention $g\cdot f$ as we imagine
the composition as the join of two consecutive arrows rather than
the composition of two functions. This convention is convenient
when the composition in ${}_a\Cc_{a}$ of a linear category is
regarded as a multiplication to turn ${}_a\Cc_{a}$ into an
algebra.

A category $\Cc$ is a {\it linear category} if each morphism set
${}_a\Cc_b$ is a finitely dimensional vector space, and the
composition of morphisms is a bilinear map of vector spaces.  It
follows that for each object $a$, ${}_a\Cc_a$ is a finitely
dimensional unital algebra. It follows that a finitely dimensional
unital algebra can be regarded as a linear category with a single
object. Another important linear category is the category of
finitely dimensional vectors spaces $\V$.  An object of $\V$ is a
finitely dimensional vector space $V$.  The morphism set $\M(V,W)$
between two objects $V,W$ is $\textrm{Hom}(V,W)$.  More generally,
given any finite set $I$, consider the linear category $\V[I]$ of
$I$-graded vector spaces, which is a categorification of the group
algebra $\C [G]$ if $I$ is a finite group $G$. An object of
$\V[I]$ is a collect of finitely dimensional vector spaces
$\{V_i\}_{i\in I}$ labelled by elements of $I$, and the morphism
set $\M(\{V_i\}_{i\in I},\{W_j\}_{j\in I})$ is the (graded) vector
space of linear maps $\oplus_{i\in I}\textrm{Hom}(V_i,W_i)$.  In
the following all categories will be linear categories, and we
will see that any semisimple linear category with finitely many
irreducible representations is isomorphic to a category of a
finite set graded vector spaces.

%\appendix{Linear representations of category}\label{repofcategory}

%\section{Linear representations of category}%\label{repofcategory}

\subsection{General representation theory}

\begin{definition} A (right) representation of a linear category $\Cc$ is a functor $\rho:
\Cc \rightarrow \V$, where $\V$ is the category of finitely
dimensional vector spaces.  The action is written on the right:
$\rho(a)=V_a$ and given an $f\in {}_a\Cc_b,
v.\rho(f)=v.f=v\cdot\rho(f): V_a\rightarrow V_b$ for any $v\in
V_a$.
\end{definition}

The 0-representation of a category is the representation which
sends every object to the 0-vector space.  Fix an object $a\in
\Cc^0$, we have a representation of the category $\Cc$, denoted by
$a\Cc$: the representation sends $a$ to the vector space
${}_a\Cc_{a}$, and any other $b\in \Cc^0$ to ${}_a\Cc_b$.  An
important construct which gives rise to all the representations of
a semi-simple linear category is as follows: fix an object $a\in
\Cc^0$ and a right ideal $J_a$ of the algebra ${}_a\Cc_{a}$, then
the map which sends each object $b\in \Cc^0$ to $J_a\cdot
{}_a\Cc_{b}\subseteq {}_a\Cc_{b}$ affords $\Cc$ a representation,
where $J_a\cdot {}_a\Cc_{b}$ is the subspace of ${}_a\Cc_{b}$
generated by all elements $f\cdot g, f\in J_a\subseteq
{}_a\Cc_{a},g\in {}_a\Cc_{b}$.  If the right ideal $J_a$ is
generated by an element $p_a\in {}_a\Cc_{a}$, then the resulting
representation of $\Cc$ will be denoted by $p_a \Cc $.  In
particular if $J_a={}_a\Cc_{a}$, we will have the regular
representation $a\Cc$.

The technical part of the paper will be the analysis of the
representations of certain picture categories. In order to do
this, we first recall the representation theory for an algebra---a
linear category with a single object.

\begin{definition}  Let $A$ be an algebra, an element $e\in A$ is an
idempotent if $e^2=e\neq 0$.  Two idempotents $e_1,e_2$ are
orthogonal if $e_1e_2=e_2e_1=0$.  An idempotent is minimal if it
is not the sum of two orthogonal idempotents.
\end{definition}

Given an idempotent $e$ of a finitely dimensional semi-simple
algebra $A$, the right ideal $eA$ is an irreducible representation
of $A$ if and only if the idempotent $e$ is minimal. Since every
irreducible right representation of $A$ is isomorphic to a right
ideal $eA$ for some idempotent of $A$,
 the representations of
$A$ are completely known once we find a collection of pairwise
orthogonal minimal idempotents $e_i$ of $A$ such that
$1=\oplus_{i}e_i$.  It follows that $A=\oplus_{i=1}^n e_iA$.

Let $p(x)$ be a polynomial of degree=$n$ with $n$ distinct roots
$a_1,a_2,\cdots, a_n$ and $A$ be the quotient algebra $\C
[x]/(p(x))$ of the polynomial algebra $\C [x]$. Let
$u_j=\prod_{i=1,i\neq j}^n (x-a_j), \lambda_j=\prod_{i=1,i\neq
j}^n(\lambda_j-\lambda_i)$ and $e_j=\frac{u_j}{\lambda_j}$.  Then
we have the following lemma.
\begin{lemma}\label{decomposition}
The idempotents $\{e_j\}_{j=1}^n$ of $A$ are pair-wise orthogonal
and $\oplus_{j=1}^n e_j=1$.  It follows that $A$ is semi-simple
and a direct sums of $\C$'s.  Note that $e_j$ is an eigenvector of
the element $x\in A$ associated with the eigenvalue $a_j$.
\end{lemma}

\begin{proof} Since $u_j(x-a_j)=p(x)=0$, so $u_j\cdot x=u_j\cdot
a_j=a_j u_j$.  It follows that $u_j^2=u_j\prod_{i=1,i\neq
j}^n(x-a_i)=\lambda_j u_j$, therefore $e_j^2=e_j\neq 0$. Now
consider $u_i\cdot u_j$, in $u_j$ there is the factor $(x-a_i)$ if
$i\neq j$, but $u_i(x-a_i)=p(x)=0$, hence $u_iu_j=0$.

The polynomial $g(x)=(\sum_{j=1}^n e_j)-1$ is a polynomial of
degree $n-1$, but it has $n$ distinct roots $a_1,a_2,\cdots, a_n$,
so $g(x)$ is identically 0.\end{proof}

A representation $\rho$ of $\Cc$ is {\it reducible} if $\rho$ is
the direct sum of two non-zero representations of $\Cc$. Otherwise
$\rho$ is {\it irreducible}. A linear category $\Cc$ is {\it
semi-simple} if every representation $\{\rho,\V \}$ of $\Cc$ is a
direct sum of irreducible representations.

\begin{definition}
$\Lambda$ has a positive definite Hermitian inner product (pdhi) iff
each morphism set ${}_a\Lambda_b$ has a finite dimensional pdhi and
composition ${}_a\Lambda_b\bigotimes {}_b\Lambda_c \stackrel{P}{\rightarrow}
{}_a\Lambda_c$ satisfies the compatibility $<P({}_am_b\bigotimes
{}_bm_c), {}_cm_d>=<{}_am_b, P({}_bm_c\bigotimes {}_cm_d)>,
{}_im_j\in {}_i\Lambda_j$, and for all $i,j\in \Lambda^0, {}_i\Lambda_j $ is
identified with $\overline{{}_j\Lambda_i}$.
\end{definition}

\begin{lemma}\label{starsemi}
Suppose $\Lambda$ has positive definite Hermitian inner
product, then $A$ is semi-simple.
\end{lemma}

\begin{proof} If $\Lambda$ has a pdhi, then any (finite dimensional)
representation $\{\rho, V\}$ of $\Lambda$ may also be given a pdhi
structure.  This means that the $V_i$ are individually pdhi-spaces
and that for all morphisms $m, (\rho(m))^{\dag}=\rho({\bar{m}})$.
One may check that any collections of pdhi-structures on $\{V\}$
which are averaged under the invertible morphisms (and therefore
invariant) satisfies this condition. \end{proof}

\iffalse
\begin{definition}
A category $\Cc$ is a $*$-category if for each ${}_a\Cc_b$, there
is an anti-linear $*$ map from ${}_a\Cc_b$ to ${}_b\Cc_a$ such
that the paring $f\cdot g^{*}: {}_a\Cc_b\times {}_a\Cc_b
\rightarrow {}_a\Cc_a$ for $f,g\in {}_a\Cc_b$ is non-degenerate,
i.e. $f\cdot g^{*}=0$ for all $g$, then $f=0$.
\end{definition}

\begin{lemma}\label{starsemi}  If $\Cc$
is a $*$ category, then $\Cc$ is semi-simple.
\end{lemma}
\fi

Most of the usual machinery of linear algebra, including Schur's
 lemma, holds for $\C$-linear categories.

\begin{lemma}
\label{schur} {\bf (Schur's Lemma for $\C$-linear categories)}
Suppose $\{\rho_m,
 V_i\}$ and $\{ \chi_m, W_i\}, i\in \textrm{obj}(\Lambda), m \in
 \textrm{Morph}(i,j),$ are irreducible representations of a
 $\C$-linear category $\Lambda$ (called an algebroid by some authors
 e.g. [BHMV]).  Irreducibility means no $\rho_m$ invariant class
 of proper subspaces $V_i'\subset V_i$ exists.  Suppose that
 $\phi: \{V\} \rightarrow \{W\}$ is a $\Lambda$-map commuting with the
 action $\Lambda$.  That is for $m\in \textrm{Morph}(i,j)$ and $v_i\in
 V_i$ we have $\chi_m(\phi_i(v_i))=\phi_j \cdot \rho_m(v_i)$.
 Then either $\phi$ is identically zero for all $i,
 \phi_i:V_i\rightarrow W_i$, or $\phi$ is an isomorphism.  If
 $\{V\}=\{W\}$ then $\phi=\lambda \cdot \textrm{id}$ for some
 $\lambda\in \C$.

 \end{lemma}

\begin{proof} As in the algebra case $\textrm{ker}(\phi)$ (and
 $\textrm{image}(\phi)$) are both invariant families of subspaces
 (indexed by $i\in \T{obj}(\Lambda)$.) So if either is a nontrivial
 proper subspace for any $i$ irreducibility of $\{V\}$ (or
 $\{W\}$) fails.   For the second assertion, since $\C$ is
 algebraically closed for any $i\in \T{obj}(\Lambda)$, the
 characteristic equation $\T{det}(\phi_i-x\T{I})=0$, has roots,
 call one $\lambda_i$.  The $\Lambda$-map $\phi-\lambda \T(I)$ has
 non-zero kernel (at least at object $i$) so by part one,
 $\phi-\lambda I=0$ identically or $\phi=\lambda I$. \end{proof}

 \begin{corollary}\label{uniqueness}
   Suppose the representation $\{\rho,V\}$ of $\Lambda$ has
 decomposition $\{V\}=V_{a_1}\bigotimes \{V_1\}\bigoplus \cdots \bigoplus
 V_{a_k}\bigotimes \{V_k\}$, where the $V_l$ are distinct (up to
 isomorphism) irreducible representations, $l$ is finite index,
 $1\leq l\leq k$, and the $V_{a_l}$ are ordinary $\C$-vector
 spaces with no $\Lambda$-action, ( Dimension $(V_{a_l})=: d(a_l)$ is
 the multiplicity of $V_l$.) The decomposition is unique up to
 permutation and of course isomorphism of $V_{a_l}$ and scalars
 acting on $\{V_l\}$.
 \end{corollary}

\begin{proof} Suppose $\{V\}=\bigoplus_m W_{a_m}\bigotimes
\{W_j\}.$  Apply
 Schur's lemma to compositions:

\begin{equation*}
v_{a_l}\bigotimes \{V_l\}\rightarrow \{V\} \rightarrow
w_{a_m}\bigotimes\{W_m\}.
\end{equation*}
for all $v_{a_l}\in V_{a_l}$ and $w_{a_l}\in W_{a_l}$ to conclude
that given $l$, $V_l\cong W_m$ for some $m$ and $d(a_l)=d(a_m)$.
This established uniqueness. \end{proof}

Now we state a structure theorem from \cite{walker06} for the
representation theory of semisimple linear categories.  Both the
statement and the proof are analogous to those for the semi-simple
algebras.  A {\it right ideal} of $\Cc$ is a subset $J$ of $\Cc^1$
such that for each object $a\in \Cc^0$, $J\cap {}_a\Cc_{a}$ is a
right ideal of ${}_a\Cc_{a}$.  Note that each right ideal of $\Cc$
affords $\Cc$ a representation.

\begin{theorem}\label{structuretheorem}

1): Let $\Cc$ be a semi-simple linear category, and $\{X_i\}_{i\in
I}$ be a complete set of representatives for the simple right
ideals of $\Cc$. Then $\Cc$ is naturally isomorphic to the
category of the finite set $I$-graded vector spaces with each
object $a\in \Cc^0$ corresponding to the graded vector space
$X_{ia}$, where $X_{ia}$ is $X_i\cap {}_a\Cc_{a}$.

2):  Each irreducible representation $\rho$ of $\Cc$ is given by a
right ideal of the form $e_a\Cc$ for some object $a\in \Cc^0$, where
$e_a$ is a minimal idempotent of ${}_a\Cc_{a}$.  If for some $b\in
\Cc^0$, which may be $a$, and $e_b$ is a minimal idempotent of
${}_b\Cc_b$, then the irrep $e_b\Cc$ of $\Cc$ is isomorphic to
$e_a\Cc$ if and only if there exist $f\in {}_a\Cc_b$ and $g\in
{}_b\Cc_a$ such that $f\cdot g=e_a$, $g\cdot f=e_b$.

\end{theorem}

\section{Gluing and Maslov
index}\label{topology}

%\section{Gluing and Maslov index}\label{topology}

\subsection{Gluing}
Gluing of $3$-manifolds needs to be addressed carefully due to anomaly.  The basic problem
is when $X$ is a bordism, the canonical Lagrangian subspace $\lambda_X \in H_1(\p X;\R)$ is in general
not a direct sum.  $\lambda_X$ is determined by the intrinsic topology as it is the kernel of the inclusion
homomorphism: $H_1(\p X; \R)\rightarrow  H_1(X; \R)$.  But the anomaly is related to the parameterizations of the
bordisms, which are extrinsic.

Suppose $X_i, i=1,2$ are bordisms from $-Y_i$ to $Y_{i+1}$ extended by
$\lambda_j,j=1,2,3$.  The canonical Lagrangian
subspace $\lambda_{X_i}$ defines a Lagrangian subspace of $H_1(Y_2;\R)$ as follows:
let $\lambda_{-}X_1=\{b\in H_1(Y_2;\R)|\tn{for some}\;\; a \in \lambda_1, (a,b)\in \lambda_{X_1}\}$, and
$\lambda_{+}X_2=\{c\in H_1(Y_2;\R)|\tn{for some}\;\; d \in \lambda_3, (c,d)\in \lambda_{X_2}\}$.  Then we have
three Lagrangian subspaces in $H_1(Y_2;\R)$ together with $\lambda_2$.

More generally, let $(Y_i, \lambda_i)$ be extended sub-surfaces of $(\p X; \lambda_X)$, and
$f: (Y_1;\lambda_1)\rightarrow (-Y_2;\lambda_2)$ be a gluing map.  Then we have three Lagrangian subspaces
in $H_1(Y_1;\R)\oplus H_1(Y_2;\R)$: the direct sum $\lambda_1 \oplus \lambda_2$, the anti-diagonal
$\Delta=\{(x\oplus -f_{*}(x)\}$, and $K$---the complement of $\lambda_i$ in $\lambda_X$ mapped here.

\subsection{Maslov index}
Given three isotropic subspaces $\lambda_i,i=1,2,3$ of a
symplectic vector space $(H,\omega)$, we can define a symmetric
bilinear form $<,>$ on $(\lambda_1+\lambda_2)\cap \lambda_3$ as
follows: for any $v,w\in (\lambda_1+\lambda_2)\cap \lambda_3$,
write $v=v_1+v_2, v_i\in \lambda_i$, then set
$<v,w>=\omega(v_2,w)$.  The Maslov index
$\mu(\lambda_1,\lambda_2,\lambda_3)$ is the signature of the
symmetric bilinear form $<,>$ on $(\lambda_1+\lambda_2)\cap
\lambda_3$.

\end{document}